\documentclass[12pt, reqno]{amsart}
\usepackage{amsfonts}
\usepackage{bbm}
\usepackage{amscd,amsfonts}
\usepackage{amssymb, eucal, amsfonts, amsmath, xypic, latexsym}
\usepackage{pifont}
\usepackage{color}
\usepackage{mathrsfs}
\usepackage{amsthm,indentfirst,bm,fancyhdr,dsfont}
\usepackage{graphicx}
\usepackage[all]{xy}
\usepackage[CJKbookmarks=true]{hyperref}

\usepackage{mathrsfs,color}

\newtheorem{thm}{Theorem}[section]
\newtheorem{lem}[thm]{Lemma}
\newtheorem{cor}[thm]{Corollary}
\newtheorem{prop}[thm]{Proposition}

\newtheorem{super-derivations}[thm]{Super-Derivation}

 \theoremstyle{definition}

\newtheorem{defn}[thm]{Definition}

\newtheorem{rem}[thm]{Remark}

\newtheorem{conv}[thm]{Convention}
\newtheorem{Kac claim}[thm]{Kac claim}

\numberwithin{equation}{thm}

\newcommand{\de}{\delta}
\def\N{\mathscr N}
\def\ggg{\mathfrak{g}}

\def\nnn{\mathfrak{n}}
\def\hhh{\mathfrak{h}}
\def\ker{{\rm Ker\,}}
\def\im{{\rm Im\,}}

\def\max{{\rm max}}
\def\ad{{\rm ad}}
\def\ch{{\rm ch}}

\def\Hom{{\rm Hom}}
\def\Ext{{\rm Ext}}

\def\pr{\partial}

\def\bbf{\mathbb F}
\def\bbz{\mathbb Z}

\def\bbn{\mathbb{N}}

\def\co{\mathcal O}
\def\calr{\mathcal{R}}

\def\calv{\mathcal{V}}

\newcommand{\mf}{\mathfrak}

\newcommand{\al}{\alpha}
\newcommand{\be}{\beta}
\newcommand{\ot}{\otimes}
\newcommand{\ka}{\kappa}
\newcommand{\cw}{\curlywedge}
\newcommand{\dis}{\displaystyle}
\newcommand{\ga}{\gamma}
\newcommand{\mbf}{\mathbf}

\def\C{\mathbb{C}}
\def\N{\mathbb{N}}
\def\bbn{\mathbb{N}}

\def\Z{\mathbb{Z}}
\def\bbz{\mathbb{Z}}

\newcommand{\op}{\oplus}
\newcommand{\cd}{\cdot}
\newcommand{\bgop}{\bigoplus}

\newcommand{\mcal}{\mathcal}

\newcommand{\la}{\lambda}
\newcommand{\La}{\Lambda}

\newcommand*{\Perm}[2]{{}^{#1}\!P_{#2}}

\newcommand{\bW}{\textbf{W}}
\newcommand{\bS}{\textbf{S}}
\newcommand{\bCS}{\textbf{CS}}
\newcommand{\bH}{\textbf{H}}
\newcommand{\bCH}{\textbf{CH}}
\newcommand{\bK}{\textbf{K}}

\newcommand{\Card}{\text{Card}}
\newcommand{\spann}{\text{span}}
\newcommand{\soc}{\text{Soc}}
\newcommand{\mfk}{\mathfrak}
\newcommand{\str}{\textsf{str}}
\newcommand{\diag}{\text{diag}}

\def\sfp{\textsf{P}}

	\def\lc{\mathcal{L\kern -.3em   C}}
	
	\def\gl{\mathfrak{gl}}
	
	\def\barz{{\bar{0}}}
	\def\baro{{\bar{1}}}
	
	\def\bz{{\bar{0}}}
	\def\bo{{\bar{1}}}
	\def\hmod{\text{-\bf{mod}}}
	
	\def\mcale{\mathcal{E}_\circ}
	\def\cale{\mathcal{E}}	
	\def\scrs{\mathscr{S}}

\begin{document}
		
		\title[Representations of Lie superalgebras]{Representations of a class of infinite-dimensional primitive Lie superalgebras}
		
		\author{Priyanshu Chakraborty, Yuhui shen and Bin Shu}
		
		\address{School of Mathematical Sciences, Ministry of Education Key Laboratory of Mathematics and Engineering Applications \& Shanghai Key Laboratory of PMMP,  East China Normal University, No. 500 Dongchuan Rd., Shanghai 200241, China} \email{priyanshu@math.ecnu.edu.cn}
		\address{School of Mathematical Sciences, Ministry of Education Key Laboratory of Mathematics and Engineering Applications \& Shanghai Key Laboratory of PMMP,  East China Normal University, No. 500 Dongchuan Rd., Shanghai 200241, China}\email{51215500073@stu.ecnu.edu.cn}
		\address{School of Mathematical Sciences, Ministry of Education Key Laboratory of Mathematics and Engineering Applications \& Shanghai Key Laboratory of PMMP,  East China Normal University, No. 500 Dongchuan Rd., Shanghai 200241, China} \email{bshu@math.ecnu.edu.cn}

		

		\begin{abstract} In \cite[\S5.4]{Kac77} and \cite{Kac98}, V. G. Kac tried to raise,  and  finished a classification of infinite-dimensional primitive Lie superalgebras. The series $\bW(m,n)$ with $m,n$ being positive integers  are the fundamental ones.
			
			In this article, we introduce the BGG category $\co$ of modules over $\bW(m,n)$, and try to systematically investigate the representations of $\bW(m,n)$ in this  category, analogue of the study in \cite{DSY24} dealing with finite-dimensional Lie superalgebra case  $\bW(0,n)$, or analogue of the study in \cite{DSY20} dealing with infinite-dimensional Lie algebra case $\bW(m,0)$.  Beyond a compound of the arguments in \cite{DSY20} and in \cite{DSY24}, it is nontrivial to understand irreducible modules in $\co$, which is the main goal of this article. We solve the question with  aid of homological analysis on costandard modules along with  extending Skryabin's theory on independence of operators for graded differential operator Lie algebras in  \cite{Skr} to the super case.   After classifying irreducible modules in this category and  describing their structure, we finally obtain irreducible characters. In the end,  by confirming the semi-infinite character property, and applying Soergel's tilting module theory  in \cite{Soe}, we study indecomposable tilting modules in $\co$, obtaining their character formulas.
			
		\end{abstract}
		
		\maketitle
		\setcounter{tocdepth}{1}\tableofcontents	
		\setcounter{section}{0}

		\section*{Introduction and preliminaries}
		
		\subsection{} Since Kac's classification of finite-dimensional simple Lie superalgebras over an algebraically closed field, the study of representations for those Lie superalgebras has developed extensively and thoroughly (see  \cite{BWW, B03, BLW, CCC, CKW17, CLW15}, {\sl{etc}}, in particular \cite{CW12, Mu12} and the references therein). In the same seminal paper \cite{Kac77}, Kac raised the classification of infinite-dimensional primitive Lie superalgeberas in its second last section \S5.4.  Recall that Lie superalgebra
		$L$ with a distinguished subalgebra $L_0$, is called primitive if $L_0$ is a maximal subalgebra and it does not contain nontrivial ideals of $L$.
		Let $\calr:=\bbf[x_1,\ldots, x_n]\otimes \bigwedge(y_1,\ldots,y_n)$ be the commutative superalgebra with indeterminants $x_i, i=1,\ldots,m;  y_s, s=1,\ldots,n$, and  relations $x_ix_j=x_jx_i$, $y_sy_t=-y_ty_s$ and $x_iy_s=y_sx_i$. The degree conventions  $\deg x_i= \deg y_j= 1$
		determine on $\calr$ a $\bbz$-grading, while  the parity conventions are given as follows:  the parities of $x_i$ and $y_j$ are $\bz$ and $\bo$ respectively. The superspace of super derivations of $\calr$ becomes an infinite-dimensional Lie superalgebra, denoted by $\bW(m,n)$. There are other infinite-dimensional Lie superalgeras, as subalgebras of $\bW(m,n)$ which are defined via exterior  differential forms, for example
		$$ \bS(m, n), \bCS(m, n), \bH(m, n), \bCH(m, n), \text{ and }\bK(m, n)$$
		with $m> 0$.  Under the conditions $(1^\circ)-(5^\circ)$ of \cite[\S5.4]{Kac77}, Kac made the following claim ({\sl{ibid.}})
			
			\begin{Kac claim} Let $L = \bigoplus_{i\geq -d}L_i$ be an infinite-dimensional $\bbz$-graded Lie
				superalgebra having properties  $(1^\circ)-(5^\circ)$ . Then $L$ is isomorphic as $\bbz$-graded superalgebra
				to one of $\bW(m, n)$, $\bS(m, n)$, $\bCS(m, n)$, $\bH(m, n)$,
				$\bCH(m, n)$, or $\bK(m, n)$ with $m>0$.
			\end{Kac claim}
			His final classification of linearly compact Lie superlagebras was properly formulated in \cite{Kac98} and  proved there, 	based on his claim along with \cite{Shc1, Shc2} and their works jointly with Cheng \cite{ChKac98, ChKac99} (the list of the complete classification contains eight series of completed graded superalgebras, two series of filtered deformations, and six  exceptional Lie superalgebras (see \cite[Theorem 0.1]{Kac98})). 		This classification can be regarded a super counterpart of infinite-dimensional  primitive Lie algebras related to primitive Lie pseudogroups (see \cite{G70, GQS66, KoN, SS65}, {\sl{etc}}.).
		

		\subsection{} In this paper, we investigate the representations of those infinite-dimensional  primitive Lie superalgebras, beginning with the BGG category of $\bW(m,n)$, in the sense of analogue of \cite{DSY20}, \cite{DSY24}  tracing back to  \cite{BBG}.
		
		Let $\ggg=W(m,n)$ for $m,n\in \bbz_{>0}$. As introduced above, $\ggg$ is endowed with both $\bbz$-gradation, and $\bbz_2$-structure
		\begin{align*}
			\ggg&=\sum_{i=-1}^\infty \ggg_i\cr
			&=\ggg_\bz\oplus \ggg_\bo
		\end{align*}
		both of which are compatible. This means $\ggg_\bz=\sum_{k=0}^\infty\ggg_{2k}$, and $\ggg_\bo=\sum_{k=-1}^\infty\ggg_{2k+1}$.
		Especially, $\ggg_0\cong \gl(m|n)$ and $\ggg_{-1}\cong \bbf^{m|n}$. Consider $\sfp=\ggg_{-1}+\ggg_0$ which is a  subalgebra, regarded a mimic of the minimal parabolic subalgebra in the case  $W(0,n)$ (see \cite{DSY24} where all parabolic subalgebras containing $\ggg_0$ are classified into two).  By definition, the BGG category of $\ggg$ contains objects which are locally $\sfp$-finite, $\hhh$-semisimple, and endowed
		with $\bbz$-structure compatible with the $\bbz$-gradation of $\ggg$ (see Definition \ref{category O}). Here $\hhh\cong \bbf^m\oplus \bbf^n$
		denotes  the standard Cartan subalgebra of $\ggg_0$ which is equivalent to the one consisting of all diagonal matrices $\text{diag}(a_1,\ldots,a_m)\oplus \diag(c_1,\ldots,c_n)\in \gl(m)\oplus \gl(n)$ in $\gl(m|n)$ (see \S\ref{sec: gl}). In particular,
		$\hhh^*$ admits a standard basis $\{\epsilon_i, \delta_s\mid i=1,\ldots,m; s=1,\ldots,n\}$ which can be equivalently defined as $\epsilon_i(\diag(a_1,\ldots,a_m)+\diag(c_1,\ldots,c_n))=a_i$ and $\delta_s(\diag(a_1,\ldots,a_m)+\diag(c_1,\ldots,c_n))=c_s$.
		It is well-known that the isomorphism classes of finite-dimensional irreducible modules  of $\gl(m|n)$  are in a one-to one correspondence to the set $\Lambda^+$ of dominant integral weights with respect to  the standard positive root system of $\gl(m|n)$.
		The irreducible module corresponding to $\lambda\in \Lambda^+$ is denoted by $L^0(\lambda)$ (see for example,  \cite{CW12}).

		Roughly, one can naturally define the standard objects $\Delta(\lambda)$ and costandard objects $\nabla(\lambda)$ as usually by inducing and coinducing from $L^0(\lambda)$, respectively (see \S\ref{sec: standard obj}-\ref{sec: costandard obj}),
		where $\lambda$ runs through $\Lambda^+$.
		The simple objects $L(\lambda)$ in $\co$ can be determined as simple heads of standard objects or simple socles of costandard objects (see Lemmas \ref{lem1} and \ref{lem: simple socles}).   However, it is  challenging to investigate the structure of irreducible  modules and to describe the irreducible characters for this category.
		
		\subsection{} The main purpose  of the present paper is to accomplish the challenging task as mentioned in the previous paragraph.  The primary idea is to  describe the costandard modules $\nabla(\lambda)$. There is a non-trivial result that $\nabla(\lambda)$ can be realized by the so-called mixed-product module $\calv(\lambda):=\calr\otimes_\bbf L^0(\lambda)$ (see Proposition \ref{p3.6}), as a super counterpart of Professor Guang-Yu Shen's mixed-product modules in \cite{Shen1}.

		Thus, it becomes a theme to investigate $\calv(\lambda)$. In order to investigate $\calv(\lambda)$, we develop  the theory introduced by Skryabin on independence of differential operators in \cite{Skr} (see \S\ref{sec: skry theory}). Such development is based on a large computation, most of which is listed in the Appendix B. With this, we  prove that  $\calv(\lambda)$ is irreducible if and only if $\lambda$ is non-exceptional weights (see below or Theorem \ref{thm: 5.5}). Thus, we classify $\calv(\lambda)$ into two classes
		\begin{itemize}
			\item[(1)] Exceptional modules $\calv(\omega_k)$ and $\calv(\theta_q)$ with exceptional weights $\omega_k$ and $\theta_q$ with $k$ and $q$ being non-negative integers:
			\begin{align*}
				\omega_k:=\begin{cases} 0  \; &\text{ for }k=0, \cr
					\sum_{i=1}^k\epsilon_i   \; &\text{ for }k=1,\ldots,m,\cr
					\sum_{i=1}^m\epsilon_i + (k-m)\delta_1 \; &\text{ for }k>m.
				\end{cases}
			\end{align*}
			and
			$$\theta_q=\sum_{i=1}^m\epsilon_{i}-\sum_{s=1}^n\delta_s-q\delta_n$$
			for $q\in \bbn$.
			
			\item[(2)] Non-exceptional modules $\calv(\lambda)$ with $\lambda$ not belonging to exceptional weights.
		\end{itemize}

		\subsection{} In order to get the precise structure information of all irreducible modules and to obtain their characters, the other key ingredient is to investigate the decomposition series of $\calv(\omega_k)$ (type I) and $\calv(\theta_q)$ (type II),  which are dealt with separately. We construct two kinds of  $\ggg$-module chain complexes:
		$$0\longrightarrow\mathcal{V}(\omega_0)\xrightarrow{\,\,\,\mfk d_0\,\,\,}\mathcal{V}(\omega_1)\xrightarrow{\,\,\,\mfk d_1\,\,\,}\cdots\cdots \mathcal{V}(\omega_k)\xrightarrow{\,\,\,\mfk d_k\,\,\,}\mathcal{V}(\omega_{k+1})\xrightarrow{\mfk d_{k+1}}\cdots\cdots, $$
		and
		\begin{align*}
			\mcal \cdots\cdots\longrightarrow\mathcal{V}(\theta_{q+1})\xrightarrow{\,\,\,\mathbf d_{q+1}\,\,\,}\mathcal{V}(\theta_q)\xrightarrow{\,\,\,\mathbf d_q\,\,\,} \mathcal{V}(\theta_{q-1})\xrightarrow{\,\,\,\mathbf d_{q-1}\,\,\,}\cdots\cdots\mathcal{V}(\theta_{1})\xrightarrow{\,\,\,\mathbf d_{1}\,\,\,}\mcal V(\theta_0) \xrightarrow{\,\,\,\mathbf d_{0}\,\,\,} 0
		\end{align*}
		(see \S\ref{sec: differential type I} and \S\ref{sec: differential type II} for the differentials).  The first chain complex above is somehow typical, similar to de Rham complex
		by a lot of computation most of which is listed in the Appendix A. Consequently, the corresponding irreducible modules and irreducible characters can be read off from the long exact sequence.  However, the second chain complex is not an exact sequence  with singularity at $q=m$,
		which means that it loses the  exactness at this point. At any regular point $q\neq m$, the complex turns out to be exact.
		We need to analyse more on the sub-quotients of $\calv(\theta_m)$ arising from the second complex above. The analyses based on computations show that at singular point $q=m$,  $\calv(\theta_m)$ admits three composition factors:
		$$\calv(\theta_m)\slash \ker\mbf d_m, \; \ker\mbf d_m/\im\mbf d_{m+1}, \text{ and }\im\mbf d_{m+1}$$
		(see  Theorem \ref{thm on V(la_q)}).  According to the property of the second complex mentioned previously along with the composition factors of $\calv(\theta_m)$,  we finally read off all irreducible characters for irreducible modules arising from the type II highest weights. Hence, we  obtain the structure theorem for irreducible modules in $\co$ (Theorem \ref{Final irr thm}), and   all irreducible characters (Theorem \ref{thm: irr char}).

		\subsection{Tilting modules} Tilting module theory in classical Lie theory has becomes important since its appearance. Soergel developed the tilting module theory for graded Lie algebras with semi-infinite character property, which enabled him to succeed in the study of character formulas for indecomposable tilting modules of affine Kac-Moody algebras.
		We continue to investigate the tilting modules of the BGG category $\co$ for $\bW(m,n)$, which is also an expectation raised in \cite{DSY20}. In the case  $\ggg=\bW(m,n)$,  according to Soergel's construction in \cite{Soe} and Brundan's super verion (see \cite{JB}) we can talk the tilting objects in $\co$ corresponding to each simple objects, denote by ${}^dT(\lambda)$. Furthermore,  semi-infinite character requirement turns out satisfied, by routine but tedious  computation which is listed in Appendix C. With the above results on costandard modules and irreducible modules for $\co$ over $\ggg$, we then obtain the characters of indecomposable tilting modules (Theorem \ref{thm: tilting character}),  with aid of  Soergel reciprocity (Proposition \ref{soegrel formula}).
		
		\subsection{Main results} In the following, we summarize our main result.
		\begin{thm} \label{thm: main 1}
			Let $\ggg=\bW(m,n)$. The following statements.
			
			\begin{itemize}
				\item[(1)]  The iso-classes of irreducible modules in $\mathcal{O}$ are parameterized by ${\Lambda^+}\times \mathbb{Z}$. More precisely, each irreducible module $S$ in $\mathcal{O}$ is of the form $L(\mu)$ for some $\mu\in \La^+$ with the depth $d$, which belongs to the subcategory $\co_{\geq d}$ (see \S\ref{sec: sub O} for the notation).
				
				\item[(2)] Forgetting the depth,  an irreducible module $L(\lambda)$ can be described as: $L(\lambda)\cong\calv(\lambda)$ if $\lambda\notin\{\omega_k,\theta_q\mid k,q\in\Z_{>0}\}$, $L(\omega_k)\cong\im \mfk d_{k-1}$, $L(\theta_q)\cong \im \textbf{d}_{q+1}$, { for all $k,q \geq 1$.}
				\item[(3)] The irreducible characters are presented below.
				
				{\begin{itemize}
						
						\item[(3.1)] \ch$(L(\la))	=\Gamma \ch(L^0(\la))$, for $\la \neq \theta_k, \omega_k$, $k \geq 1$.
						\item[(3.2)] $\ch(L(\omega_k))=\dis{\sum_{i=0}^{k} } (-1)^{2k-i}\Gamma \ch(L^0(\omega_{k-i})),  \,\,\, \forall \, k \geq 1.$
						\item[(3.3)] $\ch(L(\theta_q))=\dis{\sum_{i=0}^{q-1} } (-1)^{2q-i}\Gamma \ch(L^0(\theta_{q-i})) +(-1)^q\Gamma e^{\mcale},  \,\,\, \forall \, q <m.$
						\item[(3.4)] $\ch(L(\theta_q))=\dis{\sum_{i=0}^{q-1} } (-1)^{2q-i}\Gamma \ch(L^0(\theta_{q-i})) +(-1)^q\Gamma e^{\mcale} +(-1)^{q-m+1} \ch(L(0)),  \, \forall \, q \geq m,$
						
					\end{itemize}
					where  $\Gamma=\ch \calr$ and $\mcale= \dis{\sum_{i=1}^{m}}\epsilon_i - \dis{\sum_{j=1}^{n}}\de_j.$
				}
			\end{itemize}
		\end{thm}
		The above theorem is a summary of Theorems \ref{Final irr thm} and \ref{thm: irr char} along with Lemmas \ref{lem1} and \ref{lem: shift functor}.

		
		For indecomposable tilting modules ${}^dT(\lambda)$ with $(\lambda,d)\in \Lambda^+\times \bbz$, we have the following results on their character formula, ignoring their depth $d$.

		\begin{thm}\label{thm: tilting character}
			Maintain the notation as above. Let  $m,n> 1$ and $\lambda\in\Lambda^+$. Then the following statements hold.
			\begin{itemize}
				\item[(1)] If $\la=-2\sum\limits_{i=1}^k\epsilon_{m+1-i}-\sum\limits_{j=k+1}^m\epsilon_{m+1-i} + \sum\limits_{i=1}^n\de_{n+1-i}$  for some $k$ with $1\leq k< m$. Then
				$$ \ch(T(\la))=\Upsilon(\ch(L_0(\lambda))+\ch(L_0(\lambda+\epsilon_{m-k}))).  $$
				\item[(2)] If  $ \la=-\sum\limits_{i=1}^m\epsilon_{i} -\mcal E_{\circ} -(k-m)\de_n $ for some $k \geq m$. Then
				$$ \ch(T(\la))=\Upsilon(\ch(L_0(\lambda))+\ch(L_0(\lambda+\de_{n}))).  $$

				\item[(3)] If $\la=-2\mcal E_{\circ} +k\de_1 $ and $k \neq m$, $k \in \Z_+$. Then
				$$ \ch(T(\la))=\Upsilon(\ch(L_0(\lambda))+\ch(L_0(\lambda-\de_{1}))).  $$

				\item[(4)]  If $\la=-2\mcal E_{\circ} +m\de_1. $ Then
				$$ \ch(T(\la))=\Upsilon(\ch(L_0(\lambda))+\ch(L_0(\lambda-\de_{1}))+chL_0(\la-\mcal E_{\circ} +m \de_1)).  $$
				\item[(5)] If $\la$ be none of type (1)-(4). Then
				$$ \ch(T(\la))=\Upsilon(\ch(L_0(\lambda))). $$
			\end{itemize}
{\color{red}Here $\Upsilon=\prod\limits_{\alpha\in\Phi^{\geq 1}_{\bar 1}}
				(1+e^{\alpha})^{-1} \prod\limits_{\alpha\in\Phi^{\geq 1}_{\bar 0}}
				(1-e^{\alpha})^{-1}$.}
		\end{thm}
		
		The proof of Theorem \ref{thm: tilting character} will be given in \S\ref{sec: soeg reci}-\S\ref{sec: pf of tilt ch}.

		\subsection{Organization of the present paper} In the first section, we introduce some basic material on $\bW(m,n)$. In the second section, the BGG category $\co$ for $\bW(m,n)$ is introduced. In particular,  costandard modules are  proved to be isomorphic to the corresponding mixed-product modules. The third section is devoted to dealing with exceptional weight modules, where two types of complexes are constructed and analysed. The main results in this section are Theorems \ref{thm: 4.3} and \ref{thm on V(la_q)} which present all composition factors of exceptional mixed-product modules of type I and of type II, respectively. In the fourth section, we extend Skryanbin's arguments on independence of differential operators in \cite{Skr} to the super case $\bW(m,n)$, leading to the result that all non-exceptional mixed-product modules are irreducible (Theorem \ref{thm: 5.5}). The fifth section is devoted to the arguments on tilting modules and the proof of Theorem \ref{thm: tilting character}.
		
		\subsection{General notations}
		\begin{itemize}

			\item Throughout this paper we will work over the base field $\bbf$ which is an algebraically closed field of characteristic $0$.
			\item Let $\mathbb{R}, \Z$ denote the set of complex numbers, set of real numbers and set of integers respectively.
			\item Let $\mathbb N $ and $\Z_+$ denote set of non-negative integers and set of positive integers respectively. For $r \in \N$, $\bbf^r = \{(a_1, \ldots , a_r) : a_i \in \bbf, 1 \leq i \leq r\}$ and
			$\mathbb{R}^r, \Z^r, \N^{r}$ and $\Z_{+}^{r}$ are defined similarly.
			\item In general, elements of $ \mathbb N^r$ are denoted by Greek letters $\al, \be, \ga$ with the meaning that $\al=(\al_1, \ldots, \al_r)$. Set  $||\al||=\al_1+ \cdots + \al_r$ for $\alpha\in\bbn$.
			\item For $S \subseteq \N$, denote $\Card(S)$ as the cardinality of $S$.
			\item For any Lie algebra $\mathfrak{g}$, $U(\mathfrak{g})$ will denote universal enveloping algebra of $\mathfrak{g}$.
			\item For $a, b \in \Z$, denote $[a,b]=\{x \in \Z: a \leq x \leq b\}$.

			\item Denote the subsets of $[1,n]$ by the letters $\eta, \mu$ etc and set $|\mu|= \Card(\mu)$.
			
			\item For a superspace $V=V_\bz+V_\bo$, we denote the parity of a homogeneous vector $v\in V$ by $\wp(v)$.
			
		\end{itemize}
		
		\subsection{Conventions}
		
		For convenience we make the following convention on notations.
		\begin{conv}\label{conv}
			\begin{itemize}
				\item[(1)] For $\alpha=(\alpha_1,\cdots,\alpha_m)\in\bbn^m$, define
				
				\begin{align*}
					\partial^{\alpha}=\begin{cases}\prod_{i=1}^n \partial_i^{\alpha_i}, &\text{ if }\alpha\in\bbn^m;\cr
						0, &\text{ otherwise}.
					\end{cases}
				\end{align*}
				\item[(2)] Fix the notation $\al!=\al_1! \dots \al_m!$, for $\al=(\al_1,\al_2,\dots,\al_m) \in \N^m$.
				\item[(3)] For $\alpha=(\alpha_1,\ldots,\alpha_m),\kappa=(\kappa_1,\ldots,\kappa_m)\in\bbn^m$,
				define $\kappa\leq \alpha$ if and only if $\kappa_i\preceq \al_i$ for $i=1,\ldots,m$.
				\item[(4)] Let $\kappa\preceq \alpha$ as in (3). Define
				$$ C_\ka^\al:=\prod_{i=1}^m
				C_{\ka_i}^{\al_i}$$
				and
				$$P_\ka^\al:=\prod_{i=1}^m
				P_{\ka_i}^{\al_i}$$
				with the combinatorial numbers $C_{\ka_i}^{\al_i}:= {\frac{\al_i!}{\kappa_i!(\al_i-\kappa_i)!}}$, and $P_{\ka_i}^{\al_i}:= {\frac{\al_i!}{(\al_i-\kappa_i)!}}$.
				\item[(5)] For a subset $\eta= \{i_1, \dots, i_k\} $  of $[1,n]$ with $i_1 < \cdots < i_k,$ we define $D_\eta=D_{i_1}\dots D_{i_k}$.
				
				\item[(6)] We always denote by $\delta_{ij}$ the character function on $\bbz\times \bbz$ with value in $\{0,1\}$, taking value $1$ at point $(i,j)$ with $i=j$, and $0$ otherwise.

				\item[(7)]
				We denote  $\epsilon_i=(\delta_{1i}, \ldots, \delta_{mi})\in \bbn^m$.  This means that for $\alpha=(\alpha_1,\ldots, \alpha_m)\in \bbn^m$, we have $\alpha=\sum_{i=1}^m \alpha_i \epsilon_i$.
			\end{itemize}

		\end{conv}

	\subsection*{\textsc{Acknowledgement}}	This work is partially supported by the National Natural Science Foundation of China (Grant No. 12071136, and 12271345)  by Science and Technology Commission of Shanghai Municipality (No. 22DZ2229014). BS expresses his deep thanks to Shun-Jen Cheng for his helpful discussions and his introduction of the references \cite{ChKac98, ChKac99, Kac98}.

\subsection*{Data availability} This paper has no associated data.

\subsection*{Conflict of interest} The authors declare no conflict of interest.

		
		\section{Infinite-dimensional primitive Lie superalgebra  $\bW(m,n)$ }\label{sec: first}

		In this paper, we always assume that the ground field $\bbf$ is algebraically closed, and of characteristic $0$. All super vector spaces (modules) are over $\bbf$. For $V=V_{\bar 0}\oplus V_{\bar 1}$, the party of a $\bbz_2$-homogeneous vector $v\in V$ is denoted by $\wp(v)$. Moreover, if dim $V_{\bar 0}=m$ and dim $V_{\bar 1}=n$, then we write $V$ as $ \bbf^{m|n}.$
		
		\subsection{ }\label{CartanType}
		Let $m$ and $n$ be a pair of non-negative integers with $m+n>0$, and $\calr=\bbf[x_1,\cdots, x_m; y_1,\ldots, y_n ]$ be the polynomial superalgebra on  $m$ indeterminants $x_1,\ldots,x_n$ and $n$ infinitesimal indeterminants $y_1,\ldots,y_n$, which means
		that $\calr$ is a commutative $\bbf$-superalgebra with generators $x_i$, $y_s$ $(i=1,\ldots,m; s=1,\ldots,n)$ and defining relations for $i,j=1,\ldots,m$ and $s,t=1,\ldots,n$
		\begin{align*}
			x_ix_j&=x_jx_i,\cr
			x_iy_s&=y_sx_i,\cr
			y_s^2&=0,\cr
			y_sy_t&=-y_ty_s.
		\end{align*}
		So $\calr=\bbf[x_1,\ldots,x_m]\otimes_\bbf \bigwedge(y_1,\ldots,y_n)$. As a polynomial superalgebra, $x_i, y_s$ $(i=1,\ldots,m; s=1,\ldots,n)$ are endowed with degree $1$. This gives $\calr$ a $\bbz$-gradation. What is especially important is that, we appoint to $x_i$ and $y_s$ the parities $\barz$ and $\baro$ respectively.  This endows $\calr$ with the structure of $\bbf$-super space. We can
		express basis elements $x_1^{\al_1}\cdots x_m^{\al_m}y_{j_1} \cdots y_{j_k}$ by $x^\alpha y_\mu$ for simplicity, where
		$\alpha=(\al_1,\ldots,\al_m) \in \N^m$ and $\mu=\{j_1, \cdots , j_k\} \subseteq [1,n]$.
		Then we have the natural superspace  $\calr=\mathcal{R}_{\bar 0} \oplus \mathcal{R}_{\bar 1}$ with
		$\calr_{\bar 1}=\spann_\bbf\{x^\alpha y_\mu: |\mu |\in 2\bbz\}$,
		and $\calr_{\bar 1}=\spann_\bbf\{x^\alpha y_\mu : |\mu|\in 2\bbz+1\}$. Hence for a homogeneous polynomial $f=x^\al y_\mu \in \mcal R$, the parity of $f$ is $ |\mu| (\mod  2)$.
		On the other hand,  $\calr$ admits a natural grading via the $\bbz$-degree: $\calr=\sum_{i\geq 0}\calr_i$ with $\calr_0=\bbf$ and $\calr_i$ consists of all homogeneous polynomials of degree $i$ for $i>0$.
		
		Denote by $\mf g=\bW(m,n)$ the space spanned all  super derivations  on $\calr$. This means $\mf g=\mf g_{\bar 0}\oplus \mf g_{\bar 1}$ with $\mf g_{\bar 0}$ and $\mf g_{\bar 1}$ consisting of even and odd derivations, respectively, in the following sense:
		
		\begin{super-derivations}
			A $\bbz_2$-homogenous derivation  $E\in \mf g_{|E|}$ is a linear operator on $\calr$ with the following axiom
			$$ E(f_1f_2)=E(f_1)f_2+(-1)^{\wp(E)\wp(f_1)}f_1 E(f_2) $$
			where $E\in \mf g_{\wp(E)}$, and $f_i\in \calr_{\wp(f_i)}$, $i=1,2$.
		\end{super-derivations}
		Then it is readily verified that $\mf g$ becomes a Lie superalgebra under the Lie bracket for $\bbz_2$-homogeneous derivations:
		$$[E_1,E_2]=E_1\circ E_2-(-1)^{\wp(E_1)\wp(E_2)}E_2\circ E_1$$ \
		where $E_i\in \mf g_{\wp(E_i)}$ for $i=1,2$.

		
		When $m=0$, then $\bW(0,n)$ becomes a finite-dimensional simple Lie superalgebra which is one of Cartan series in  the classification of finite-dimensional simple Lie superalgebras over $\bbf$ by Kac in \cite{Kac77}. The representations in this case were investigated thoroughly in \cite{BL83}, \cite{Shap82},   \cite{Ser05}, \cite{DSY24}, {\sl{etc}}.
		
		When $n=0$, then $\bW(m,0)$ becomes an infinite-dimensional primitive Lie algebra. The representations in this case were investigated thoroughly in \cite{Rud},  \cite{Fu86}, \cite{Shen3}, \cite{DSY20}, {\sl{etc}}.
		
		
		\subsection{}
		So we always assume $m,n>0$. The polynomial superalgebra $\calr$ is an infinite-dimensional superspace. So $\bW(m,n)$ is naturally an infinite-dimensional Lie superalgebra.
		

		By definition, $\mf g$ is actually a free $\calr$-module with basis $\{\partial_j\mid 1\leq j\leq m\}\bigcup \{D_s\mid \partial_t\mid 1\leq t\leq n\}$, where $\partial_i$ is the partial (even) derivation with respect to $x_i$, i.e., $\partial_i(x_j)=\delta_{ij}$ for $1\leq i,j\leq n$ with $\partial_i(x_jx_l)=\partial_i(x_j)x_l+x_j\partial_i(x_l)$, $i,j,l=1,\ldots,m$; and $D_t$ is the partial (odd) derivation with respect to $y_t$, i.e., $D_t(y_s)=\delta_{ts}$ for $1\leq t,s\leq n$ with $D_t(y_sy_q)=D_t(y_s)y_q-y_sD_t(y_q)$, $t,s,q=1,\ldots,n$.
		Then $\mf g$ is spanned by the following basis elements:
		$x^\alpha y_\mu \partial_j;  x^\alpha y_\mu D_t$ for $\alpha\in\bbn^m$, $\mu \subseteq [1,n]$,  $j\in [1,m]$ and $t\in [1,n]$.
		
		So $\mf g$ can be clearly described as a super space with the even part $\mf g_{\barz}$ spanned by:
		$$\{ x^\alpha y_\mu \partial_j, x^\be y_\eta D_t: |\mu|\in 2\bbz, |\eta|\in 2\bbz+1 \}$$
		and the odd part $\mf g_{\baro}$ spanned by:
		$$\{ x^\alpha y_\mu \partial_j, x^\be y_\eta D_t: |\mu|\in 2\bbz+1, |\eta|\in 2\bbz \}.$$
		\subsection{$\bbz$-gradations} From the above structure, the $\bbz$-gradation of $\calr$ gives rise to the $\mathbb{Z}$-grading on $\mf g$, i.e., $\mf g=\bigoplus\limits_{i=-1}^{\infty}\mf g_{i}$, where $\mf g_{i}=\text{span}_{\mathbb{F}}\{f_i\partial_i, g_sD_s : 1\leq i\leq m, 1\leq s\leq n;  f_i, g_s\in \calr, \deg(f_i)=\deg(g_s)=i+1\}$. Then  $\ggg_{-1}=\sum_{i=1}^m\bbf\partial_i\oplus \sum_{t=1}^n\bbf D_t$. So $\ggg_{-1}\cong \bbf^{m|n}$ and
		$\ggg_{0}=\langle x_i\partial_j, y_sD_t, x_iD_t, y_s\partial_j\mid i,j=1,\ldots,m; s,t=1,\ldots,n
		\rangle$.  In particular, $\sfp:=\ggg_{-1}+\ggg_0$ is a subalgebra which is a ``minimal" parabolic subalgebra in the sense of \cite{DSY24}. The subalgebra $\sfp$ will play an important role in the BGG category introduced later.

		Moreover,  $\{\ggg_{\geq k}:=\bigoplus\limits_{i \geq k}^{\infty}\mf g_{i}\mid k=-1,0,1,\ldots\}$ forms a decreasing filtration of subalgebras of $\mf g$.

		\subsection{The graded-zero component  $\ggg_0$}\label{sec: gl}  By the $\bbz$-graded construction of $\ggg=\bW(m,n)$,
		$\ggg_{0}=\spann_\bbf \{ x_i\partial_j, y_sD_t, x_iD_t, y_s\partial_j\mid i,j=1,\ldots,m; s,t=1,\ldots,n\}$. Then $\mf g_0$ is naturally isomorphic to the Lie superalgebra $\mf{gl}(m,n)$ via the isomorphism given by
		\begin{align}\label{eq: iso gl}
			& x_i\pr_j \mapsto E_{ij}, \qquad x_i D_s \mapsto E_{i,m+s}, \cr &y_r\pr_j \mapsto E_{m+r,j},   \qquad y_rD_s \mapsto E_{m+r,m+s}
		\end{align}
		for all $i,j \in [1,m]$ and $r,s \in [1,n]$. Here $E_{st}$ are the unit matrix with $(s,t)$th entry equal to $1$ and all other $0$ for $s,t=1,\ldots, m+n$. Let	$\hhh$ be the standard Cartan subalgebra of $\ggg_{0}$ which is $\bbf$-spanned  by  $x_i\partial_i, y_sD_s, i=1,\ldots,m; s=1,\ldots,n$. Now we define unit linear functions $\epsilon_i, \de_r$ on $\mf h$ by defining $\epsilon_i(x_j\pr_j)=\de_{ij}$, $\epsilon_i(y_sD_s)=0$ and $\de_r(x_j\pr_j)=0$, $\de_r(y_sD_s)=\de_{rs}$, for $i,j \in [1,m]$ and $r,s \in [1,n] $.  Here we abuse the notation $\epsilon_i$ as the same as in (7) of Convention \ref{conv} because via the isomorphism given by (\ref{eq: iso gl}),   $\epsilon_i$ can be  regarded the $i$th coordinate function on the space $\sum_{i=1}^m \bbf E_{ii}\cong \bbf^m$.

		\section{The BGG category $\co$}\label{sec: sec}
		
		Maintain the notations and conventions as before. In particular, $\ggg=\bW(m,n)$ with $m,n$ being positive integers. In this section, we introduce the BGG category $\co$ of $\ggg$-modules.

		\subsection{}
		The following notion of Category $\mcal O$ is an analogy of the BGG category for complex finite dimensional semi-simple Lie algebras. Further, an interesting reader can see the reference \cite{DSY20} for Category $\mcal O$ in case of the Lie algebra of polynomial vector fields.
		
		\begin{defn}\label{category O}
			Denote by $\mathcal{O}$ the BGG category of $\ggg$-modules, whose objects are $\bbz_2$-graded vector spaces $M=M_{\bar 0} \oplus M_{\bar 1}$ with the following properties satisfied.
			\begin{itemize}
				\item[(1)] $M$ is an admissible $\mathbb{Z}$-graded $\ggg$-module, i.e., $M=\bigoplus\limits_{i\in\mathbb{Z}} M_{i}$ with $\dim M_{i}<+\infty$, and
				$\ggg_{i}M_{j}\subseteq M_{i+j},\forall\,i,j \in \Z$.
				\item[(2)] $M$ is locally finite for $\sfp$ where $\sfp$ is introduced before.
				{ \item[(3)] $M$ is $\hhh$-semisimple, i.e., $M$ is a weight module: $M=\bigoplus_{\lambda\in\hhh^*}M_{\lambda}$.}
				
			\end{itemize}
			The morphisms in $\mathcal{O}$ are always assumed to be even and they are $\ggg$-module morphisms that respect the $\mathbb{Z}$-gradation i.e.,
			$$\Hom_{\mathcal{O}}(M, N)=\{f\in\Hom_{U(\ggg)}(M, N)\mid f(M_{i})\subseteq N_{i},\,\\ \,\forall\,i\in\mathbb{Z}\},\, \forall\,M,N\in\mathcal{O}.$$
			
		\end{defn}
		Additionally, we say an object $M \in \mcal O$ has depth $d$ if $M =\bigoplus\limits_{i \geq d} M_{i}$	with $M_d \neq 0$ and $M_i=0$, for all $i < d$.
		
		\subsection{Standard objects in $\co$}\label{sec: standard obj}
		\subsubsection{}\label{basic notations}
		
		Let $\Lambda^+$ be the set of dominant integral weights relative to the standard Borel subalgebra $\mathfrak{b}:=\hhh+\nnn^+$ of $\ggg_{0}$, i.e.
		$$\La^+=\{\dis{\sum_{i=1}^{m}}\la_i\epsilon_i + \dis{\sum_{j=1}^{n}}\zeta_j\de_j: \la_i -\la_{i+1} \in \N, \, \zeta_j -\zeta_{j+1} \in \N \, \text{ for all possible } i, j\}.$$

		It is well-known that $\Z$-graded finite-dimensional irreducible $\ggg_{0}$-modules are parameterized by ${\Lambda^+}\times\mathbb{Z}$. For any $\lambda\in {\Lambda^+},$ let ${}^dL^0(\lambda)$ be the irreducible  $\ggg_{0}$-module concentrated in a single degree $d$ with highest weight $\lambda$.  Set $${}^d\Delta(\lambda)=U(\ggg)\otimes_{U(\sfp)}{}^dL^0(\lambda),$$
		where ${}^dL^0(\lambda)$ is regarded as a $\sfp$-module with trivial $\ggg_{{-1}}$-action. Then ${}^d\Delta(\lambda)$ is called a standard module of depth $d$ and
		$\{{}^d\Delta(\lambda)\mid \lambda\in {\Lambda^+},d\in\mathbb{Z}\}$ constitute a class of so-called standard modules of depth $d$ for $\ggg$ in the usual sense.  The following result will follow with the parallel lines of \cite[Lemma 3.1]{DSY20}.
		\begin{lem}\label{lem1}
			Let $\lambda\in{\Lambda^+}, d\in\mathbb{Z}$. The following statements hold.
			\begin{itemize}
				\item[(1)] The standard module ${}^d\Delta(\lambda)$ is an object in $\mathcal{O}$.
				\item[(2)] The standard module ${}^d\Delta(\lambda)$ has a unique irreducible quotient, denoted by ${}^dL(\lambda)$.
				\item[(3)] The iso-classes of irreducible modules in $\mathcal{O}$ are parameterized by ${\Lambda^+}\times \mathbb{Z}$. More precisely, each irreducible module $S$ in $\mathcal{O}$ is of the form $L(\mu)$ for some $\mu\in \La^+$ with the depth $d$.
			\end{itemize}
		\end{lem}

		\subsubsection{Shift functors by shifting depthes}\label{sec: sub O}
		Set $$\mathcal{O}_{\geq d}:=\{M\in\mathcal{O}\mid M=\sum\limits_{i\geq d}M_{i}\},$$ which consists of objects admitting depths not smaller than $d$.

		Consider the shift functor  $T_{d',d}:\mathcal{O}_{\geq d}\longrightarrow\mathcal{O}_{\geq d'}$ which maps the object  $M=\sum_{k}M_k\in \mathcal{O}_{\geq d}$   to $M[d'-d]\in\mathcal{O}_{\geq d'}$, where $M[d'-d] =\sum_{k}(M[d'-d])_k$ and $M[d'-d]_k=M_{k+d'-d}$.  The following fact is clear.
		
		\begin{lem}\label{lem: shift functor}
			The functor $T_{d',d}$ induces a category equivalence.
		\end{lem}
		
		With the shift functors, we can focus our concern on $\mathcal{O}_{\geq 0}$ (or $\mathcal{O}_{\geq d}$ for some specific depth $d$) when we make arguments on module structures. Therefore from this point we only consider the depth of the modules as zero.
		
		\subsection{Costandard objects in $\co$ and their mixed-product realization}\label{sec: costandard obj}
		
		\subsubsection{Costandard modules} Let $\lambda\in{\Lambda^+}$, define the costandard $\ggg$-module corresponding to $\lambda$ as $$\nabla(\lambda):=\Hom_{U(\ggg_{\geq0})}(U(
		\ggg), L^0(\lambda)),$$
		where $L^0(\lambda)$ is regarded as a $\ggg_{0}$-module with trivial $\ggg_{\geq 1}$-action. The module structure on $\nabla(\la)$ is given by:
		$$ (X.\phi)(Y)=\phi(YX)  \, \,\,\,  \forall \, X \in \mf g, \, Y \in U(\mf g), \, \phi \in \nabla(\la).  $$
		Then by the arguments parallel to \cite[\S4.1]{DSY20}, one can show  that $\nabla(\lambda)\in\mathcal{O}$.  Thanks to  \cite[Lemma 3.5]{JB}, we have the following Lemma.
		
		{\begin{lem}\label{description of socle}\label{lem: simple socles}
				The costandard module  ${}^d\nabla(\la)$ admits a simple socle which is isomorphic to ${}^dL(\la)$.
			\end{lem}
		}

		{	\subsubsection{Mixed-product modules and their realizations for costandard modules} In this subsection, we introduce  a kind realization of costandard modules $\nabla(\lambda)$ for $\lambda\in{\Lambda^+}$ via prolonging $L^0(\lambda)$ as below.} For $\la \in \La^+$ consider a $\mf g_0$-module  $L^0(\la)$ and set $$\mathcal{V}(\lambda)=\calr\otimes_\bbf L^0(\lambda).$$
		Now define a $\mf g$-module structure on $\mathcal{V}(\lambda)$ via
		\begin{align}\label{W(n)-module structure}
			f\pr_i.(g\otimes v)=&f\pr_i(g)\otimes v+\sum\limits_{j=1}^m \partial_j(f)g\otimes \xi(x_j\partial_i)v\cr
			&+(-1)^{\wp(f)+\wp(g)+1} \sum\limits_{j=1}^n D_j(f)g\otimes \xi(y_j\pr_i)v
		\end{align}
		and
		\begin{align}\label{W(n)-module structure-2}
			fD_i.(g\otimes v)=&fD_i(g)\otimes v+(-1)^{\wp(g)}\sum\limits_{j=1}^m \partial_j(f)g\otimes \xi(x_jD_i)v \cr &+(-1)^{\wp(f)+1} \sum\limits_{j=1}^n D_j(f)g\otimes \xi(y_jD_i)v
		\end{align}
		for all $f,g\in \calr, v\in L^0(\lambda)$,
		where $(L^0(\lambda),\xi)$ is the representation of $\mf g_{0}$.
		It is routine (but tremendous calculation) to check that this forms a module action. This structure can be regarded a super version of Guang-Yu Shen's mixed-products for a class of graded Lie algebras (see \cite{Shen1}). We denote this representation by $(\mcal V(\la), \rho^g)$.
		{
			We will call $\mathcal{V}(\lambda)$ as mixed-product modules, in memory of Professor Guang-Yu Shen for his contribution to  the modular and infinite-dimensional representation theory of graded Lie algebras with aid of his mixed-product modules (see \cite{Shen1}-\cite{Shen3}).}
		
		\begin{prop}\label{p3.6}
			$	\nabla(\lambda) \cong \mathcal{V}(\lambda)$ as $\ggg$-modules.
		\end{prop}
		We postpone   the proof of the above proposition till \S\ref{sec: proof of prop}. Before that, we need some preparation lemmas.

		\subsubsection{Preparation lemmas}		Let us define  $\psi:\mathcal{V}(\lambda) \to 	\nabla(\lambda)$  by:
		$$  x^\al y_\mu \ot v \mapsto \psi(x^\al y_\mu\ot v), \, \forall \al \in \N^m, \mu \subseteq [1,n], \, v \in L^0(\la), $$
		where $$\psi(x^\al y_\mu\ot v)(\pr^aD_\eta)=
		\begin{cases}
			\pr^\al(x^\al)D_\mu(y_\mu)v, \ \ \ \ &\text{ if } a=\al, \eta =\mu, \\
			0, &\text{otherwise}.
		\end{cases}$$
		
		Note that,  $\psi(x^\al y_\eta \ot v)$ is well defined due to the isomorphism	{\color{red}
		$$\nabla(\lambda)=\Hom_{U(\mf g_{\geq 0})	}(U(\mf g),L^0(\lambda)) \simeq \Hom_{U(\mf g_{\geq 0})	}(U(\mf g), \bbf) \ot_ \bbf  L^0(\lambda ).$$
		}

		The following two lemmas are easy to check by induction principal, which are necessary to prove the map $\psi$ is a $\mf g$-module map.

		\begin{lem}\label{l3.7} For $\vartheta\in \{\pr_i,D_i\}$, the following formula holds:
			{$$\pr^\alpha(x^\beta y_\eta \vartheta)=\sum_{0\preceq\kappa\preceq\alpha} C_\kappa^\alpha P_\kappa^\beta x^{\beta -\kappa}y_\eta\vartheta\pr ^{\alpha-\kappa}.$$}
		\end{lem}
		
		\begin{lem}\label{l3.8}
			Let $\eta \subseteq [1,n]$ and $D_{\eta_1}=D_{m_1} \cdots D_{ m_r}$, $ \eta_1=\{m_1, \ldots, m_r\}\subseteq \eta$
			. Then \\
			$(i) D_{\eta_1}x^\al y_\eta \pr_i=x^{\al } D_{\eta_1}(y_{\eta})\pr_i +  \dis{\sum_{k_1=0}^{r-1}c_{k_1}x^\al y_{\eta \setminus A_{k_1}} \pr_i D_{m_{r-k_1}} } + \dis{\sum_{k_i \in \eta_1} }c_{k_1,\dots,k_u} x^\al y_{\eta \setminus A_{k_1,\dots,k_u}} \pr_i D_{A_{k_1,\dots,k_u}^c} .$\\
			{$(ii)  \ D_{\eta_1}x^\al  y_\eta D_i=x^{\al }D_{\eta_1}(y_{\eta})D_i +  \dis{\sum_{k_1=0}^{r-1} c_{k_1}x^\al y_{\eta \setminus A_{k_1}} D_i D_{m_{r-k_1}} } + \dis{\sum_{k_i \in \eta_1} }c_{k_1,\dots,k_u} x^\al y_{\eta \setminus A_{k_1,\dots,k_u}} D_i D_{A_{k_1,\dots,k_u}^c} ,$}
			where  $c_{k_1,\dots,k_u}$ are non-zero constant depending on $A_{k_1, \cdots, k_u}$ and $\eta$,  $A_{k_1,\dots,k_u}=\{m_{r-i}:i \neq k_1,\dots,k_u\}$ and $A_{k_1,\dots,k_u}^c=\eta_1 \setminus A_{k_1,\dots,k_u}$.
		\end{lem}
		
		
		\begin{lem}\label{l3.9}
			Let $\eta, \mu, \mu_1 \subset \{1,2,\dots n\}$ such that $\mu_1\cup \eta_2 \neq \mu$ for all subset $\eta_2$ of $\eta$. Also let $\eta_1=\{r_1,\dots , r_k\}\subseteq \eta$. Then
			\begin{enumerate}
				\item
				$\psi(x^\be y_\mu \ot v)(\pr^aD_{\eta_1}x^\al y_\eta D_i)D_{\mu_1}= \psi(x^\be y_\mu \ot v)(\pr^ax^\al D_{\eta_1}(y_{\eta })D_i)D_{\mu_1}$.
				\item 	$\psi(x^\be y_\mu \ot v)(\pr^aD_{\eta_1}x^\al y_\eta \pr_i)D_{\mu_1}= \psi(x^\be y_\mu \ot v)(\pr^ax^\al D_{\eta_1} (y_{\eta})\pr_i)D_{\mu_1}$.
				
			\end{enumerate}
		\end{lem}		
		\begin{proof}
			We prove (1) of this lemma on the induction of cardinality of $\eta_1$ ((2) follows with similar proof due to Lemma \ref{l3.7}).
			Note that for $\eta_1=\{r\}$ we have\\
			
			$	\psi(x^\be y_\mu \ot v)((\pr^aD_{r}x^\al y_\eta D_i)D_{\mu_1}) $\\
			
			$	=\psi(x^\be y_\mu \ot v)([\pr^ax^\al D_r(y_{\eta}) D_i +(-1)^{|\eta|+1}\pr^ax^\al  y_{\eta} D_iD_r]D_{\mu_1})$\\
			
			$=\psi(x^\be y_\mu \ot v)[(\pr^ax^\al D_r(y_{\eta})D_i)D_{\mu_1}+(-1)^{|\eta|+1}(\dis \sum_{0\preceq k\preceq a} C^a_k P^{\al}_{k} x^{\al -k}y_\eta D_i\pr ^{a-k}D_rD_{\mu_1} )]$\\
			
			$=\psi(x^\be y_\mu \ot v)(\pr^ax^\al D_r(y_{\eta}) D_iD_{\mu_1}),$ \\
			
			since $r \notin \mu$, by the definition of $\psi(x^\be y_\mu \ot v)$ second term is zero.\\
			
			Now consider $\eta_1=\{r_1,\dots,r_k\}$ and assume that the result is true for cardinality of $\eta_1 \leq k-1.$ Then we have\\
			
			$ \psi(x^\be y_\mu \ot v)(\pr^aD_{\eta_1}x^\al y_\eta D_i)D_{\mu_1} $\\
			
			$= \psi(x^\be y_\mu \ot v)([\pr^aD_{r_1}\dots D_{r_{k-1}}x^\al D_{r_k}(y_{\eta}) D_i +(-1)^{|\eta|+1}\pr^aD_{r_1}\dots D_{r_{k-1}}x^\al  y_{\eta} D_iD_{r_k}]D_{\mu_1})$.\\
			
			Note that to complete the induction hypothesis, it is sufficient to prove that the second part of the above formula is zero. Now we have\\
			
			$\psi(x^\be y_\mu \ot v)(\pr^aD_{r_1}\dots D_{r_{k-1}}x^\al  y_{\eta} D_i)D_{r_k}D_{\mu_1}$\\
			
			$=\psi(x^\be y_\mu \ot v)( \pr^ax^\al D_{r_1}\dots D_{r_{k-1}} (y_{\eta})D_i D_{r_k}D_{\mu_1}) =0 ,$\\
			
			by Lemma \ref{l3.7} and definition of $\psi(x^\be y_\mu \ot v)$. This completes the proof.
			
		\end{proof}	
		\begin{lem}\label{l3.10}
			The map $\psi$ is a $\mf g$-module map.
		\end{lem}
		\begin{proof}
			Let $\mu , \eta \subseteq \{1,2\dots,n\}$.\\
			{\bf Case I :}  Let $\mu \cap \eta =\emptyset$. Then \\
			$\psi(x^\al y_\eta \pr_i.x^\be y_\mu \ot v)$\\
			$= \psi(x^\al y_\eta \pr_i(x^\be y_\mu )\otimes v+\sum\limits_{j=1}^m \partial_j(x^\al y_\eta)x^\be y_\mu \otimes (x_j\partial_i.v) +(-1)^{|\eta|+|\mu|+1} \sum\limits_{j=1}^n D_j(x^\al y_\eta)x^\be y_\mu \otimes (y_j\pr_i.v) ) $	\\
			$=\psi(\be_ix^{\al +\be-\epsilon_i}y_\eta y_\mu \ot v+\sum\limits_{j=1}^m \al_j x^{\al +\be-e_j}y_\eta y_\mu \ot x_j\pr_i.v +(-1)^{|\eta|+|\mu|+1} \sum\limits_{j=1}^n(-1)^{\eta(j)}x^{\al +\be}D_j(y_{\eta}) y_\mu \ot y_j\pr_i.v )  \hspace*{13cm}      (3.9.0)   $\\
			
			Therefore for any $A \subseteq [1,n]$ and $a \in \N^m$, we have\\
			
			$\psi(x^\al y_\eta \pr_i.x^\be y_\mu \ot v)(\pr^aD_A)$\\
			
			$ =\begin{cases}
				D_\eta D_\mu(y_\eta y_\mu)\pr^{\al+\be-e_i}(x^{\al +\be-e_i})(\be_iv + \al_i x_i\pr_i.v), \ \ \ D_A=D_\eta D_\mu, a=\al+\be-e_i \\
				D_\eta D_\mu(y_\eta y_\mu)\pr^{\al+\be-e_j}(x^{\al +\be-e_j}) \al_j x_j\pr_i.v, \ \ \ D_A=D_\eta D_\mu, a=\al+\be-e_j, j \neq i \\
				D_{\eta \setminus \{j\}} D_\mu(y_\eta y_\mu)\pr^{\al+\be}(x^{\al +\be})  y_j\pr_i.v, \ \ \ D_A=D_{\eta\setminus \{j\}} D_\mu, a=\al+\be, j \in \eta\\
				0,  \ \, \text{ otherwise}
			\end{cases}  $\\
			On the other hand we need to compute
			$x^\al y_\eta \pr_i.\psi(x^\be y_\mu \ot v)(\pr^aD_A) $. Let us write $A$ as $A=(A \setminus \eta \cap A ) \cup
			(A\cap \eta)=\mu_1 \cup \eta_1$ (say). Also let $D_A=D_{\eta_1}D_{\mu_1}$. Then \\
			
			$x^\al y_\eta \pr_i.\psi(x^\be y_\mu \ot v)(\pr^aD_{\eta_1}D_{\mu_1}) $\\
			
			$= \psi(x^\be y_\mu \ot v)(\pr^aD_{\eta_1}D_{\mu_1}x^\al y_\eta \pr_i) $\\
			
			$=(-1)^{|\eta||\mu_1|}\psi(x^\be y_\mu \ot v)(\pr^aD_{\eta_1}x^\al y_\eta \pr_i)D_{\mu_1}   $\\
			
			$=(-1)^{|\eta||\mu_1|} \psi(x^\be y_\mu \ot v)(\pr^ax^\al D_{\eta_1}(y_{\eta })\pr_i)D_{\mu_1}, $\\
			
			since $\mu$ can not be union of a subset of $\eta_1$ and $\mu_1$, otherwise $\mu \cap \eta \neq \emptyset$, hence last equality follows from Lemma \ref{l3.9}. \\
			
			$= (-1)^{|\eta||\mu_1|} \psi(x^\be y_\mu \ot v)(\dis \sum_{0\preceq k\preceq a} C^a_k P^{\al}_{k} x^{\al -k}D_{\eta_1}(y_{\eta })\pr_i\pr ^{a-k}  D_{\mu_1})  $\\

			Note that terms of this summation can be non-zero only when the following cases arise,\\
			
			$  \begin{cases}
				\ a-k=\beta, \al-k=e_r ,   i.e, \ a=\al+\be-e_r,  \mu_1=\mu,\eta=\eta_1,   1 \leq r \leq m,\\
				\ a-k=\beta,\al=k ,i.e, a=\al+\be,\ \mu_1=\mu,  \eta\setminus\eta_1=\{j\}\\
				\ a-k=\beta-e_i, \al=k , \ i.e, \ a= \al+\be-e_i,\mu_1=\mu,  \eta=\eta_1,
			\end{cases} $\\

			This follows by the definition of $\psi(x^\be y_\mu \ot v)$, the fact that $\psi(x^\be y_\mu \ot v)$ is $U(\mf g)_{\geq 0}$-module map and ${\mf g}_{\geq 1}$ acts trivially on $L_0(\la)$. Therefore we have\\
			
			$x^\al y_\eta \pr_i.\psi(x^\be y_\mu \ot v)(\pr^aD_A) $\\
			$=\begin{cases}
				(-1)^{|\eta||\mu|}D_{\eta}(y_\eta)\psi(x^\be y_\mu \ot v) (C^{\al+\be-e_i}_{\al-e_i}P^{\al}_
				{\al-e_i}x_i\pr_i\pr^\be +C^{\al+\be-e_i}_{\al}P^{\al}_{\al}\pr^\be )D_{\mu}, \\
				\hspace*{11cm}	a=\al+\be-e_i, D_A=D_\eta D_\mu\\
				
				(-1)^{|\eta||\mu|}D_{\eta}(y_\eta)\psi(x^\be y_\mu \ot v) (C^{\al+\be-e_r}_{\al-e_r}P^{\al}_{\al-e_r}x_r\pr_i\pr^\be D_{\mu}), \\
				\hspace*{9cm} \ a=\al+\be-e_r, \,  D_A=D_\eta D_\mu , \, r \neq i\\
				
				(-1)^{|\eta||\mu|+\eta(j)}D_{\eta}(y_\eta)\psi(x^\be y_\mu \ot v)(C^{\al+\be}_{\al}P^{\al}_{\al}y_j\pr_i\pr^\be D_{\mu}), \\
				\hspace*{9cm} a=\al+\be, \, D_A=D_{\eta \setminus \{j\}}D_\mu, \, j \in \eta.\\
				0, \, \text{ otherwise}
			\end{cases}
			$
			Now it follows from above two computations that	$\psi(x^\al y_\eta \pr_i.x^\be y_\mu \ot v)=x^\al y_\eta \pr_i.\psi(x^\be y_\mu \ot v)$, for all $\al , \be \in \N^m, \eta, \mu \subseteq [1,n]$.\\
			
			{\bf Case II :} Let $\eta \cap \mu \neq \emptyset
			$.
			For arbitrary $A \subseteq [1,n]$ and $a \in \N^m$, we can write $A=  (A \setminus \eta \cap A)\cup (\eta \cap A)= A_1\cup \eta_1$ (say). Also let $D_A=D_{\eta_1}D_{A_1}$. Now we have
			\begin{align}\label{a3.10.1}
				x^\al y_\eta \pr_i.\psi(x^\be y_\mu \ot v)(\pr^aD_{A})
				& = \psi(x^\be y_\mu \ot v)(\pr^aD_{\eta_1}D_{A_1}x^\al y_\eta \pr_i)  \cr
				& =(-1)^{|\eta||A_1|}\psi(x^\be y_\mu \ot v)(\pr^aD_{\eta_1}x^\al y_\eta \pr_iD_{A_1}) \cr
			\end{align}
			Now by Lemma \ref{l3.8}, we have
			$$\psi(x^\be y_\mu \ot v)(\pr^aD_{\eta_1}x^\al y_\eta \pr_iD_{A_1})  $$	
			$$= \psi(x^\be y_\mu \ot v)(\pr^ax^\al D_{\eta_1}(y_\eta) \pr_iD_{A_1})   $$	
			(since 	$A_1 \cup A_{k_1,\dots,k_u}^c=\mu$ does not hold for any subset $A_{k_1,\dots,k_u}^c=\eta_1 \setminus A_{k_1,\dots,k_u}$, otherwise $\mu \cap \eta =\emptyset$.)
			$$ =\psi(x^\be y_\mu \ot v) (\dis\sum_{0\preceq k\preceq a} C^a_k P^{\al}_{k} x^{\al -k}D_{\eta_1}(y_\eta) \pr_i \pr ^{a-k} D_{A_1})  $$
			
			It is clear with the similar arguments as Case I that  non-zero terms from the above summation   can arise only when\\
			
			$  \begin{cases}
				(i)	\ a-k=\beta, \al-k=e_r ,   i.e, \ a=\al+\be-e_r,  A_1=\mu,\eta_1=\eta,   1 \leq r \leq m,\\
				(ii)	\ a-k=\beta,\al=k ,i.e, a=\al+\be, A_1=\mu,  \eta\setminus \eta_1=\{j\}, \text{for some }\, j\\
				(iii)	\ a-k=\beta-\epsilon_i, \al=k , \ i.e, \ a= \al+\be-e_i,A_1=\mu,  \eta_1=\eta.
			\end{cases} $\\
			
			Observe that cases (i) and (iii) can not appear, since in this case $\mu \cap \eta =A_1 \cap \eta_1= \emptyset$. Moreover, from (ii) and  the fact $A_1 \cap \eta_1 = \emptyset, $ $A_1 \cap \eta=\mu \cap \eta \neq \emptyset$, it follows that $\mu \cap \eta =\{j\}$. In particular, $|\mu \cap \eta| \geq 2$ implies that $x^\al y_\eta \pr_i.\psi(x^\be y_\mu \ot v)(\pr^aD_{A})=0 $ for all $a\in \N^m$ and $A \subseteq [1,n]$.\\

			Now we consider the case $\mu \cap \eta =\{j\} $.
			Without loss of generality assume that $ \mu =\{s_1, \dots, s_t\}$, $\eta= \{m_1,\dots,m_r\}$ such that $m_r=s_1=j$. Then for $a=\al+\be$ and $D_A=D_{\eta \setminus \{j\}}D_\mu$ we have  \\

			$x^\al y_\eta \pr_i.\psi(x^\be y_\mu \ot v)(\pr^{\al +\be}D_{\eta \setminus \{j\}}D_\mu) $\\
			
			$=\psi(x^\be y_\mu \ot v)(\pr^{\al +\be}D_{\eta \setminus \{j\}}D_\mu x^\al y_\eta \pr_i ) $\\
			
			$=\psi(x^\be y_\mu \ot v)(\pr^{\al +\be}D_{\eta }D_{\mu \setminus \{j\}} x^\al y_\eta \pr_i ) $\\
			
			$=(-1)^{(|\mu|-1)||\eta|} \psi(x^\be y_\mu \ot v)(\pr^{\al+\be}D_{\eta \setminus \{j\}}D_j x^\al y_\eta \pr_i D_{\mu \setminus \{j\}})$\\
			
			$=(-1)^{(|\mu|-1)||\eta|} \psi(x^\be y_\mu \ot v) (\pr^{\al+\be}D_{\eta \setminus \{j\}}( x^\al D_j(y_{\eta}) \pr_i + (-1)^{|\eta|} x^\al y_\eta \pr_iD_j)D_{\mu \setminus \{j\}}) $\\
			
			$=(-1)^{(|\mu|-1)|\eta| +|\eta|} \psi(x^\be y_\mu \ot v) (\pr^{\al+\be}D_{\eta \setminus \{j\}}   x^\al y_\eta \pr_iD_{\mu})   $ \\
			
			( 1st term will vanish, since $\mu\setminus \{j\} \cup S \neq \mu$ for all $ S \subseteq \eta \setminus \{j\}$, see Lemma \ref{l3.9} ) \\
			
			$=(-1)^{|\mu|||\eta|  } \psi(x^\be y_\mu \ot v)(\pr^{\al+\be} x^{\al }D_{\eta \setminus \{j\}}(y_{\eta \setminus \{j\}})y_{ j} \pr_i    D_{\mu} ) $ (see Lemma \ref{l3.9}) \\
			
			$=(-1)^{|\mu||\eta|  } \psi(x^\be y_\mu \ot v)(\dis \sum_{0\preceq\kappa\preceq \al+\be} C^{\al+\be}_k P^{\al}_{k} x^{\al -k}D_{\eta \setminus \{j\}}(y_{\eta \setminus \{j\}})y_{ j} \pr_i\pr ^{\al+\be-k}   D_{\mu} ) $ \\
			
			$=(-1)^{|\mu||\eta| }D_{\eta \setminus \{j\}}(y_{\eta \setminus \{j\}})\psi(x^\be y_\mu \ot v)(C^{\al+\be}_{\al} P^{\al}_{\al} y_{ j} \pr_i\pr ^{\al+\be-\al}    D_{\mu})$ (all other terms are zero )\\
			
			$=(-1)^{|\mu||\eta| }D_{\eta \setminus \{j\}}(y_{\eta \setminus \{j\}})D_\mu(y_\mu)(\al+\be)!  y_j\pr_i.v$\\
			
			Therefore for $\mu \cap \eta =\{j\}$, we have \\
			$$x^\al y_\eta \pr_i.\psi(x^\be y_\mu \ot v)(\pr^aD_{A})$$
			$$= \begin{cases}
				(-1)^{|\mu||\eta| }D_{\eta \setminus \{j\}}(y_{\eta \setminus \{j\}})D_\mu(y_\mu)(\al+\be)!  y_j\pr_i.v, \, \text{for} \, a= \al +\be, \, D_A=D_{\eta\setminus \{j\}} D_\mu\\
				0, \, \text{otherwise}
			\end{cases}$$

			Now we compute $\psi(x^\al y_\eta \pr_i.x^\be y_\mu \ot v)$. Note that from equation (3.9.0) we have
			
			$$\psi(x^\al y_\eta \pr_i.x^\be y_\mu \ot v)=0, \, \, \text{ when} \, |\eta \cap \mu| \geq 2.$$
			
			Again for $\mu \cap \eta =\{j\}$ we have
			$$\psi(x^\al y_\eta \pr_i.x^\be y_\mu \ot v)(\pr^aD_A)$$
			$$=\begin{cases}
				(-1)^{|\mu||\eta| }D_{\eta \setminus \{j\}}(y_{\eta \setminus \{j\}})D_\mu(y_\mu)(\al+\be)!  y_j\pr_i.v, \ \ \ D_A=D_{\eta\setminus \{j\}} D_\mu, a=\al+\be, \\
				0,  \ \ \ \text{otherwise,}
			\end{cases}	$$\\
			(this follows from computation of Case I). \\
			
			This completes the proof of $x^\al y_\eta \pr_i.\psi(x^\be y_\mu \ot v)=\psi(x^\al y_\eta \pr_i.x^\be y_\mu \ot v)$ for all $ \al, \be \in \N^m, \eta , \mu \subseteq [1,n], v \in L^0(\la), i \in [1,m] $. Following this computation with the help of Lemmas \ref{l3.7},\ref {l3.8},\ref{l3.9} it is routine to check that $x^\al y_\eta D_i.\psi(x^\be y_\mu \ot v)=\psi(x^\al y_\eta D_i.x^\be y_\mu \ot v)$. This completes the proof.					
		\end{proof}

		\subsubsection{\bf Proof of Proposition \ref{p3.6}}	\label{sec: proof of prop}
		With the help of Lemma \ref{l3.10} to complete the proof we need to check that $\psi$ is a bijection. Note that $\nabla(\la)$ is spanned by $\{f_{q,v,\mu}:q=(q_1,\dots,q_m),\ v \in L^0(\la), \ \mu \subseteq [1,n]  \},$ where $	f_{q,v,\mu}(\pr^aD_\eta)$ is defined by \\
		$$	f_{q,v,\mu}(\pr^aD_\eta)=\de_{a,q}\de_{\mu,\eta}v , \ a \in \N^m,  \eta \subseteq [1,n]  .  $$
		This proves that $\psi$ is surjective since $cf_{q,v,\mu}$ is the image of the $x^q y_\mu\ot v$ for some non-zero scalar $c$. Now consider the relation $\psi(\dis{\sum_{i=1}^{k}} f_i \ot v_i )=0 , \ f_i \in \mcal R , \ v_i \in L^0(\la)$ with linearly independent $v_i$. Then we consider the action of $\pr^aD_\mu$ on $\psi(\dis{\sum_{i=1}^{k}} f_i \ot v_i )$ for various $a, \mu$, which implies that $f_i=0$ for all $i$. This completes the proof.
		
		\begin{cor}\label{Soc of V(la)}
			Any non-zero submodule of $\mcal V(\la)$ contains $1 \ot L^0(\la)$. In particular, Socle of $\nabla(\la)$ is generated by $1 \ot L^0(\la)$.
		\end{cor}
		\begin{proof}
			Let $M'$ be a non-zero submodule of $\mcal V(\la)$ and $w= \dis{\sum_{i=1}^{n} }f_i \ot v_i $ be a non-zero element of $M'$ with homogeneous $f_i \in \mcal R, v_i \in L^0(\la)$. Then using the action of $\pr_i, D_j$ sufficiently many times on $w$, we have a non-zero vector $1 \ot v \in M'$. Since $L^0(\la)$ is irreducible $\mf{gl}(m,n)$ module, we get that $1 \ot L^0(\la ) \in M'$. This clearly proves that Socle of $\nabla (\la)$ is generated by $1 \ot L^0(\la)$ by Lemma \ref{description of socle}.
		\end{proof}
			
			\subsection{A criterion for the irreducibility of mix-product modules} 		
			
			For any submodule $M$ of $\mcal V(\lambda),$ define $M_{\alpha, \mu}=\{v \in L^0(\lambda):  x^\al y_\mu \ot v \in M\}$. Then we have $M \supseteq \dis{\bigoplus_{\al, \mu}} x^\al y_\mu \ot M_{\al,\mu} $. { Let $M^\flat=\dis{\bigcap}_{\al, \mu}M_{\al,\mu} $
				\begin{lem}\label{l3.15}
					If $M^\flat\neq 0, $ then $M=\mcal{V}(\lambda)$.
				\end{lem}
			}
			\begin{proof}
				Let $v \in M^\flat$. Then $x^\al y_\eta \ot v \in M$ for all $\al, \eta$. Now to complete the proof we need to show that $x_s\pr_i.v, x_sD_j.v, y_r\pr_i.v,$ $ y_rD_j.v \in M^\flat$, for all $1 \leq s,i \leq m $ and $1 \leq r,j \leq n$. Because then $M^\flat$ becomes a submodule of $L^0(\lambda)$ and hence $M^\flat=L^0(\lambda)$ and $M=\mcal{V}(\lambda)$. In particular we show that $x_s\pr_i.v \in M^\flat$ and all other follows with similar methods. Let $\beta , \eta $ be arbitrary. Now consider the action:
				$$x_s\pr_i.(x^{\be} y_\eta \ot v)= \be_ix^{\be-\epsilon_i} y_\eta \ot v +x^{\be} y_\eta \ot x_s\pr_i. v \in M. $$
				This implies that $x^{\be} y_\eta \ot x_s\pr_i. v \in M,$ for all $\beta, \eta$ and hence $x_s\pr_i.v \in M^\flat$.				
			\end{proof}

			{
				
				\section{Exceptional mixed-product modules and their composition factors}}\label{sec: third sec}
			\subsection{Grassmann products in the category of finite $\gl(m|n)$-modules}
			Let $V\cong\bbf^{m|n}$ be the vector superspace of super-dimension $(m|n)$. Then $V$ has parity decomposition $V=V_\bz\oplus V_\bo$ where $V_\bz$ and $V_\bo$ are the even part and odd part of $V$, respectively, which are isomorphic to $\bbf^m$ and $\bbf^n$ as vector space, respectively.
			Let $e_i$ for $1\leq i\leq m$ and $e_{m+j} $ for $1 \leq j \leq n$ be the basis vectors of $V_\bz$ and $V_\bo$  respectively. It is easy to see that $V$ forms a module for  $\mathfrak{gl}(m|n)$ with the natural action given by:
			$$E_{ij}e_k=\delta_{j,k}e_i.$$
			We denote the natural representation of $\mf{gl}(m|n)$ by $(V, \rho)$.
			
			\subsubsection{} We first introduce the so-called Grassmann algebra over $V$
			$$\Omega(V):=\Omega^\bullet(V)=\bigoplus_{k\geq 0}\Omega^k(V)$$
			which is by definition the quotient of the tensor product $\bigotimes^\bullet(V)$ by the ideal generated by
			$$\{ u\otimes w +(-1)^{\wp(u)\wp(w)}w\otimes u\mid u\in V_{\wp(u)}, w\in V_{\wp(w)}\}.$$
			Hence $\Omega(V)$  is an {anti-commutative} superalgebra generated by $e_i$, $i=1,\ldots,m+n$.
			Let $\Omega^k(V)$ be the $k$th Grassmann product space over $V$, which is by definition the $k$th homogeneous part of $\Omega^\bullet(V)$.
			Hence, by definition we have the following isomorphism of vector spaces
			$$\Omega^k(V)\cong \sum_{i=0}^k \bigwedge^i (V_\bz)\otimes S^{k-i}(V_\bo)$$
			where $\bigwedge^i(V_\bz)$ denotes the $i$th exterior product space of $V_\bz$, and
			$S^{k-i}(V_\bo)$ denotes the $(k-i)$th symmetric product  space of $V_\bo$.  In the following, we denote by $\curlywedge$ the Grassmann product in $\Omega^\bullet(V)$.  The basis of $\Omega^k(V)$ is given by
			$$\mcal C =\{ e_{i_1}\curlywedge e_{i_2}\curlywedge\cdots \curlywedge e_{i_r} \curlywedge e_{m+j_1}\curlywedge e_{m+j_2}\curlywedge\cdots \curlywedge e_{m+j_l}:1 \leq i_1 < \dots <i_r \leq m,$$  $$  m+1 \leq j_1 \leq j_2 \leq \dots \leq j_l \leq m+n, r+l=k, r \geq 0, l \geq 0\}.$$

			Then $\Omega^k(V)$ naturally becomes a module for $\mathfrak{gl}(m|n)$ with the action given by \\
			$$ X.v_1\curlywedge\cdots\curlywedge v_k=\displaystyle{\sum_{i=1}^{k}}(-1)^{\wp(X)(\wp(v_1)+\cdots+\wp(v_{i-1}))}
			v_1\curlywedge\cdots\curlywedge X.v_{i}\curlywedge\cdots\curlywedge v_k, $$
			for all homogeneous $X \in \mathfrak{gl}(m|n)$ and $v_i \in V$ are homogeneous vectors for $1 \leq i \leq k$.

			\begin{lem}\label{l3.12} The module $\Omega^k(V)$
				is an  irreducible $\mathfrak{gl}(m|n)$-module for all $k \in \N$. Moreover, $\Omega^k(V) \simeq L^0(\Lambda)$, where
				\begin{equation}
					\Lambda=	\begin{cases}
						\epsilon_1+\cdots+\epsilon_k, \ \ \text{for} \ \  k \leq m,\\
						\epsilon_1+ \dots +\epsilon_m + (k-m) \de_{1},  \ \ \text{for} \ \  k\geq m.\\
						
					\end{cases}
				\end{equation}
				
			\end{lem}	
			\begin{proof}
				Let $W$ be a non-zero submodule of $\Omega^k(V)$, then $W$ contains a non-zero vector of length $r$ of the form
				\begin{align*}				
					v=&c_1e_{i_1^1}\curlywedge \dots \curlywedge e_{i_s^1}\cw e_{m+j^1_1} \cw \cdots \cw e_{m+j^1_t}+\cr
					&+\cdots+c_re_{i_1^r}\curlywedge \dots \curlywedge e_{i_s^r}\cw e_{m+j_1^r} \cw \cdots \cw e_{m+j_t^r}.
				\end{align*}
				{\bf Claim :}  $W$ contains a non-zero vector of the form $w=e_{i_1}\curlywedge e_{i_2}\curlywedge \dots \curlywedge e_{i_k}$.
				It is clear that there exists $j$ such that $ i_j^1 \neq i_a^r$ for all $1 \leq a \leq s$ or there exists some $l$ such that $\Card \{m+j^1_l \} > \Card \{m+j_a^r \} $ for all $ 1\leq a \leq r$.     Without loss of generality assume that $i_1^1 \neq i_a^r$ for all $ 1\leq a \leq s$. Then we consider the vector $E_{m+1,i_1^1} \cd v$,  which is a non-zero vector of length $ \leq r-1$. On the other case, let $e_{m+j_1^1}$ occurs $l$ times and $\Card \{m+j^1_1 \} > \Card \{m+j_a^r \} $ for all $ 1\leq a \leq r$. Then we consider the action $E_{m+n,j_1^1}^l \cd v$,  which becomes a non-zero vector with length $ \leq r-1 $.   Hence by induction claim follows.\\
				Now consider any $j_1, \dots ,j_k$ for $1 \leq j_1, \dots, j_k \leq m+n$. To prove that $W=\Omega^k(V)$, it is sufficient to show that $W$ contains $e_{j_1}\curlywedge e_{j_2}\curlywedge \dots \curlywedge e_{j_k}$. Since $W$ contains a vector $e_{i_1}\curlywedge e_{i_2}\curlywedge \dots \curlywedge e_{i_k}$,
				without loss of generality assume that $\{i_1,\dots, i_k\}\cap \{j_1, \dots,j_k\} =\{j_1, \dots, j_l\}$ for some $l <k$. Also assume that  $i_{l+1}= \dots =i_{l+r}$ for some $r$.
				Now consider the action
				\begin{align*}
					&E_{j_{l+1},i_{l+1}}.e_{j_1}\curlywedge \dots \curlywedge e_{j_l} \curlywedge e_{i_{l+1}}\curlywedge\dots \curlywedge e_{i_{l+r}} \curlywedge \dots \curlywedge e_{i_k}\cr
					=&(-1)^{\wp(E_{j_{l+1},i_{l+1}})(\wp(e_{j_1})+\wp(e_{j_2})+ \cdots+\wp(e_{j_l}))}e_{j_1}\curlywedge \dots \curlywedge e_{j_l} \curlywedge E_{j_{l+1},i_{l+1}}.e_{i_{l+1}}\curlywedge\dots \curlywedge e_{i_{l+r}} \curlywedge \dots \curlywedge e_{i_k}\cr
					&\quad +\dis{\sum_{s=1}^{r-1}} (-1)^{\wp(E_{j_{l+1},i_{l+1}})(\wp(e_{j_1})+\wp(e_{j_2})+ \cdots+\wp(e_{j_l})+\wp(e_{i_{l+1}}) + \cdots +\wp(e_{i_{l+s}}))}e_{j_1}\curlywedge \dots \curlywedge e_{j_l} \curlywedge e_{i_{l+1}} \curlywedge \dots \curlywedge \cr
					&\qquad\qquad  \curlywedge E_{j_{l+1},i_{l+1}}.e_{i_{l+s+1}}\curlywedge\dots \curlywedge e_{i_{l+r}} \curlywedge \dots \curlywedge e_{i_k}\cr
					=& r(-1)^{\wp(E_{j_{l+1},i_{l+1}})(\wp(e_{j_1})+\wp(e_{j_2})+ \cdots+\wp(e_{j_l}))}e_{j_1} \curlywedge \dots \curlywedge e_{j_l} \curlywedge e_{j_{l+1}}\curlywedge e_{i_{l+2}}\dots \curlywedge e_{i_k} \in W.
				\end{align*}
				
				Now we continue this process to conclude that $W=\Omega^k(V)$. It is easy to check that the vector $e_1\curlywedge e_2 \curlywedge \dots \curlywedge e_k$ is a highest weight vector for $k \leq m$ and for $ k=m+l$ the highest weight vector is $e_1\curlywedge e_2 \curlywedge \dots \curlywedge e_m \curlywedge e_{m+1} \curlywedge \dots \curlywedge e_{m+1}$ (wedge product of $e_{m+1}$ is taken $l$ times) with the required highest weight.
			\end{proof}	

			\subsubsection{}	We make some preparation to find a relation between the $\mf{gl}(m|n)$-modules $V$ and $V^*$.	
			For
			$$X =  \begin{pmatrix}
				A & B  \\
				C & D  \\
			\end{pmatrix}  \in \mf{gl}(m|n),$$
			define super transpose of $X $ by
			$$X^T= \begin{pmatrix}
				A^T & -C^T  \\
				B^T & D^T  \\
			\end{pmatrix}. $$
			In particular, we have
			\begin{align*}
				&E_{i,j}^T=E_{j,i}, \, \forall  \, i,j \in[1,m],\cr
				&E_{m+i,j}^T=-E_{j,m+i}, \, \forall  \, i \in [1,n],  j \in [1,m],\cr
				&E_{i,m+j}^T=E_{m+j,i} ,\, \forall \,  i \in [1,m], \,j \in [1,n],\cr    &E_{m+i,m+j}^T=E_{m+j,m+i}, \, \forall \, i,j \in [1,n].
			\end{align*}
			Moreover it is easy to see that $\tau:\mf{gl}(m|n) \to \mf{gl}(m|n)$ defined by
			$$\tau(X)=-X^T$$ is an automorphism of $  \mf{gl}(m|n)$ with the help of the property
			$$(XY)^T=(-1)^{\wp(X)\wp(Y)}Y^TX^T$$
			for all $X,Y\in \mf{gl}(m|n)$.
			
			Now we twist the natural representation $(V, \rho)$ of $\mf{gl}(m|n)$ by the automorphism $\tau$ and denote the resulting representation by $(V, \rho^\tau )$. Now consider the action of $\mf{gl}(m|n)$ on $V^*$,
			$$ (X.f)(v)=-(-1)^{\wp(X)\wp(f)}f(X.v), \, \, \forall \, X \in \mf{gl}(m|n), \, f \in V^*, \, v \in V.   $$
			We denote this representation by $(V^*, \rho^*)$. Let $\{e_i^*: 1 \leq i \leq m+n\}$ be the dual basis of the natural basis of $V$. Then one can see that $$E_{ij}.e_k^*=-(-1)^{\wp(E_{ij})\wp(e_k^*)}\de_{ik}e_j^*.$$
			
			\begin{lem}\label{equivalance lemma}
				$\mf{gl}(m|n)$ representations $(V, \rho^\tau)$ and $(V^*, \rho^*)$ are equivalent via the map $e_i \mapsto e_i^*$ for all $1 \leq i \leq m+n$.
			\end{lem}
			\begin{proof}
				Define a linear isomorphism $\phi: V \to V^*$ by $\phi(e_i)= e_i^*$, for $1 \leq i \leq m+n$. Now we show that $\phi(\rho^\tau(E_{ij})e_k)=\rho^*(E_{ij})\phi(e_k)$ for all $1 \leq i,j,k \leq m+n$.\\
				(i) Let $1 \leq i,j \leq m$ or $ m+1 \leq i,j \leq m+n$
				\begin{align}
					\phi(\rho^\tau(E_{ij})e_k) &= \phi(\rho(-E_{ji})e_k) \cr
					&= \phi(-\de_{ik}e_j) \cr
					&= -\de_{ik}e_j^*=\rho^*( E_{ij})e_k^*
				\end{align}
				(ii) Let $1 \leq i \leq n$ and $1 \leq j \leq m$.
				\begin{align}
					\phi(\rho^\tau(E_{m+i,j})e_k) & = \phi(\rho(-E_{m+i,j}^T)e_k) \cr
					&= \phi(\rho(E_{j,m+i})e_k) \cr
					&= \phi(\de_{m+i,k}e_j) \cr
					&= \de_{m+i,k}e_j^* \cr
					&= -(-1)^{\wp(e_k^*)} \de_{m+i,k}e_j^*, \text{ since } \, \de_{m+i,k} \neq 0 \text{ when, \,} k =m+i, \text{i.e,} \wp(e_k^*)=1\cr
					&=\rho^*( E_{m+i,j})e_k^*
				\end{align}
				(iii) Let $1 \leq i \leq m$ and $1 \leq j \leq n$.
				\begin{align}
					\phi(\rho^\tau(E_{i,m+j})e_k) & = \phi(\rho(-E_{i,m+j}^T)e_k) \cr
					&= \phi(\rho(-E_{m+j,i})e_k) \cr
					&= \phi(-\de_{i,k}e_{m+j}) \cr
					&= -\de_{i,k}e_{m+j}^* \cr
					&= -(-1)^{\wp(e_k^*)} \de_{i,k}e_j^*, \text{ since } \, \de_{i,k} \neq 0 \text{ when,\,} k =i, \text{i.e,} \wp(e_k^*)=0\cr
					&=\rho^*( E_{i,m+j})e_k^*
				\end{align}
				This completes the proof.
			\end{proof}

			\begin{cor}\label{Cor for irr V*}
				The module $\Omega^k(V^*)$
				is an  irreducible $\mathfrak{gl}(m|n)$-module for all $k \in \N$. Moreover, $\Omega^k(V^*) \simeq {L^0(-k\delta_n)}$.
			\end{cor}
			\begin{proof}
				Follows from Lemmas \ref{equivalance lemma} and \ref{l3.12} with the fact that $e_{m+n}^{*k}$ is a highest weight vector with highest weight vector $-k \de_n$.
			\end{proof}	
			\subsection{Exceptional weights  of $\bW(m,n)$}
			We fix the notations
			\begin{align}\label{eq: type I}
				\omega_k&= \epsilon_1+\cdots+\epsilon_k, \ \ \text{for} \ \  k \leq m, k \in \N, \, \text{where} \, \, \, \omega_0=0,\cr
				\omega_{m+l}&=   \epsilon_1+ \dots +\epsilon_m + l \epsilon_{m+1},  \ \ \text{for} \ \  l \in \Z_+
			\end{align}
			and
			\begin{align}\label{eq: type II}
				\theta_q=\sum_{i=1}^m\epsilon_i -\sum_{j=1}^n \delta_j-q\delta_n, \,  \forall \, q\in \N,
			\end{align}
			and call them exceptional weights of $\mf{gl}(m|n)$.
			
			We will show the mixed-product $\bW(m,n)$-modules corresponding to those exceptional weights will share special properties, which will called exceptional mixed-product modules.  We say that the exceptional mixed-product modules corresponding to (\ref{eq: type I}) and (\ref{eq: type II}) are of  type I and type II, respectively.
			
			\subsection{Exceptional mixed-product modules of type I}
			
			\subsubsection{}\label{sec: differential type I}
			Consider the map
			$$\mfk d_k:\mathcal{V}(\omega_k)\longrightarrow \mathcal{V}(\omega_{k+1})$$
			defined via
			$$\mfk d_k:\mathcal{V}(\omega_k)\longrightarrow \mathcal{V}(\omega_{k+1})$$
			\[x^{\alpha}y_{\eta}\otimes (e_{j_1}\curlywedge\cdots\curlywedge e_{j_k})\longmapsto\sum\limits_{i=1}^m\partial_i(x^{\alpha})y_{\eta} \otimes (e_{j_1}\curlywedge\cdots\curlywedge e_{j_k}\curlywedge e_i) \]
			\[+(-1)^{\wp(e_{j_1})+ \cdots + \wp(e_{j_k}) +|\eta|+1}\sum\limits_{i=1}^nx^{\alpha}D_i(y_{\eta}) \otimes (e_{j_1}\curlywedge\cdots\curlywedge e_{j_k}\curlywedge e_{m+i}),\]
			$\forall\,\alpha\in {\mathbb{N}}^m, \eta \subseteq [1,n], e_{j_1}\curlywedge\cdots\curlywedge e_{j_k} \in \mcal C.$
			
			\begin{lem}\label{lem: type I module map} The map $\mfk d_k$ is a $\bW(m,n)$-module homomorphism.
			\end{lem}		
			
			The proof is ordinary, but very tedious and involving lots of computation. We leave it to the appendix A (see \S\ref{sec: appendix A})

			\subsubsection{Exceptional mixed-product complex of type I} Now we consider the chain complex of $\bW(m,n)$-modules maps ( which we will confirm in Appendix A) for $\mcal V(\omega_k)$, given by: 			
			\begin{align}\label{eq: type I complex}
				0\longrightarrow\mathcal{V}(\omega_0)\xrightarrow{\,\,\,\mfk d_0\,\,\,}\mathcal{V}(\omega_1)\xrightarrow{\,\,\,\mfk d_1\,\,\,}\cdots\cdots \mathcal{V}(\omega_k)\xrightarrow{\,\,\,\mfk d_k\,\,\,}\mathcal{V}(\omega_{k+1})\xrightarrow{\mfk d_{k+1}}\cdots\cdots
			\end{align}
			
			We prove in appendix A (see Corollary \ref{cor d^2=0} ) that $\mfk d_k \circ \mfk d_{k-1}=0$ for all $k \geq 1$. Furthermore, one can observe that both the kernels and images of $\mfk d_k$ are non-zero for all $k \geq 1$. In particular, $\mcal V(\omega_k)$ is reducible for all $k \geq 1$. Now we find a spanning set for $\mcal V(\omega_k)/\im \mfk d_{k-1}$ for all $k \geq 1$. Fix some $\al \in \N^m$ with $\al_{i_r}=0$ for all $ r \notin[1,s] $. Then from $\mfk d_{k-1}(x^{\alpha +\epsilon_i}y_{\eta}\otimes e_{i_1}\curlywedge\cdots \curlywedge e_{i_s}\curlywedge e_{m+j_1}\curlywedge e_{m+j_t}) =0$, we have
			\begin{align*}			
				&(\al_i+1)x^{\alpha}y_{\eta} \otimes (e_{i_1}\curlywedge\cdots \cw e_{i_s} \cw e_{m+j_1}\curlywedge \cdots \cw e_{m+j_t}\curlywedge e_i)  \cr
				=&(-1)^{\wp(e_{m+j_1})+ \cdots + \wp(e_{m+j_t}) +|\eta|}\sum\limits_{j=1}^nx^{\alpha +\epsilon_i}D_j(y_{\eta}) \otimes (e_{i_1}\curlywedge\cdots \cw e_{i_s} \cw e_{m+j_1}\curlywedge \cdots \cw e_{m+j_t} \cw e_{m+j})
			\end{align*}
			for all $1 \leq i_1 < \cdots <i_s \leq m$ and $ 1 \leq j_1 \leq \cdots \leq j_t \leq n$ and some $i \notin \{i_1, \cdots ,i_s\}$. Hence we get that
			\begin{align*}	
				\mcal A=& \{ x^{\alpha}y_{\eta} \otimes e_{i_1}\curlywedge\cdots \cw e_{i_s} \cw e_{m+j_1}\curlywedge \cdots \cw e_{m+j_t} + \im \mfk d_{k-1}  \mid \al_{i_k} \neq 0,  \text{ for some }  k \notin [1,s], \cr
				&		 \qquad\qquad \eta \subseteq [1,n], 1 \leq i_1 < \cdots <i_s \leq m  \text{ and }  1 \leq j_1 \leq \cdots \leq j_t \leq n \}
			\end{align*}		
			is a spanning set for $\mcal V(\omega_k)/\im \mfk d_{k-1}$ for all $k \geq 1$.
			
			\begin{lem}
				\begin{itemize}
					\item[(1)] $ \mcal V(\omega_k)/ \im \mfk d_{k-1}$ is irreducible.
					\item[(2)] 	$\ker \mfk d_k$= $\im \mfk d_{k-1}$ for all $k \geq 1$.
				\end{itemize}
				
			\end{lem}	
			\begin{proof}
				At first observe that as a $U(\mf g)$-module $\mcal V(\omega_k)$ is generated by the single vector $\{ x_i \ot e_{m+1}^k \}$ for all $i \in [1,m]$. This follows from the fact
				$$x^\al y_\eta \pr_i\cd (x_i \ot e_{m+1}^k)=x^\al y_\eta  \ot e_{m+1}^k \in U(\mf g)( x_i \ot e_{m+1}^k) \,\,\forall \, \al \in \N^m , \eta \subseteq [1,n].$$ Hence with the help of Lemma \ref{l3.15}, $U(\mf g)( x_i \ot e_{m+1}^k)= \mcal V(\omega_k)$. Therefore as $U(\mf g)$-module $\mcal V(\omega_k)/ \text{Im} \mfk d_{k-1}$ is generated by the single vector $\{ x_i \ot e_{m+1}^k + \text{Im} \mfk d_{k-1} \}$ for all $i \in [1,m]$.\\
				Let $W$ be a non-zero submodule of $  \mcal V(\omega_k)/ \text{Im}  \mfk d_{k-1}$. As $W$ is a weight module, there exists a non-zero weight vector in $W$. Let $w =\dis{\sum_{i =1}^{R}}x^{\be^i}y_{\eta_i}\ot w_i + \text{Im}  \mfk d_{k-1}$ be a non-zero weight vector of weight $\dis{\sum_{i=1}^{m}} \al_i \epsilon_{i} + \dis\sum_{j \in \eta}  \de_{i} $ such that $x^{\be^i}y_{\eta_i}\ot w_i + \text{Im}  \mfk d_{k-1} \in \mcal A$ for all $i \in [1,R]$.  Now note that the $j$th component of $\be^i$ must be $\al_j$ or $\al_j -1$, as $w_i$ is of the form $\hat e_j e_{i_1} \cdots e_{i_s} e_{m+j_1} \cdots e_{m+j_t}$ or $ e_j e_{i_1} \cdots e_{i_s} e_{m+j_1} \cdots e_{m+j_t}$, due to the weight reason(here $ \hat .$  means omission). Moreover, if $w_i =w_k$ for some $i \neq k$, then $\be^i =\be^k$ and $\eta_i=\eta_k$ due to weight reason.
				
				Now considering the action of $D_j$'s sufficiently many times on $ w$ we assume that,
				$$w =\dis{\sum_{i=1}^{ R_0}}x^{\be^i}\ot w_i + \text{Im}  \mfk d_{k-1} \in W, \, $$
				for some $R_0\leq R$.
				Now looking at $\mcal A$ we can assume that $w_1=\hat e_{i_l} e_{i_1} \cdots e_{i_s}e_{m+j_1} \cdots e_{m+j_t}$ and ${\be^1_{i_l}}= {\al_{i_l}} \neq 0$.\\
				Now we consider the action
				
				$$\pr_{i_l}^{\al_{i_l}-1} \cd w = (\al_{i_l}-1)!  \dis{\sum_{r \leq R } }x_{i_l}\ot w_r + \text{Im} \mfk d_{k-1} \in W, $$ with all $w_r$ is of the form $\hat e_{i_l} e_{i_1} \cdots e_{i_s}e_{m+j_1} \cdots e_{m+j_t}$. \\
				Now consider $X=y_1\pr_{i_1} \cdots y_1\pr_{i_s}y_1D_{j_1} \cdots y_1D_{j_t}$, then we have the following\begin{align}
					X.w_0 &= \dis{\sum_{r  \leq k } }x_{i_l}\ot (X\cd w_i) + \text{Im} \mfk d_{k-1} \cr
					&= x_{i_l}\ot (X\cd w_1) + \text{Im} \mfk d_{k-1} \cr
					&= x_{i_l}\ot e_{m+1}^k + \text{Im} \mfk d_{k-1}  \in W,
				\end{align}
				since $w_i$'s are different all other $X.w_r=0$, for $r \neq 1$. Therefore $W$ contains a generator of $\mcal V(\omega_k)/ \text{Im} \mfk d_{k-1}$, hence $W=\mcal V(\omega_k)/ \text{Im} \mfk d_{k-1}$. This completes the proof of (1). Moreover, $\mfk d_{k} \mfk d_{k-1}=0$ implies that $\text{Im} \mfk d_{k-1} \subseteq \text{Ker} \mfk d_k$ and hence (2) follows from (1).
			\end{proof}
			
			\subsubsection{Composition series of $\mathcal{V}(\omega_k)$}
			\begin{lem}
				$\im \mfk d_{k}$ is irreducible for all $k \geq 0$. Moreover, {$\im \mfk d_k\cong L(\omega_{k+1})$.}
				
			\end{lem}
			\begin{proof}
				Let $W$ be a non-zero submodule of $\im \mfk d_{k}.$ We show that $\mfk d_k(x^\al y_\eta \ot v) \in W$ for all $\al \in \N^m, \eta \subseteq [1,n], v \in \Omega^k(V)$. Let us fix a non-zero basis vector $v= e_{j_1}\curlywedge\cdots\curlywedge e_{j_k} \in \Omega^k(V) $. Also let $j_1, \cdots , j_r= m+1$. Note that from Corollary \ref{Soc of V(la)} we have $ \Omega^{k+1}(V) \subseteq W$.  Now consider the following action for all $\al \in \N^m, \eta \subseteq [1,n]$:
				\begin{align*}		
					&x^\al y_\eta D_1 \cd (1 \ot e_{m+1}\cw v)\cr
					=& \sum\limits_{i=1}^m \partial_i(x^\al)y_\eta\otimes x_iD_1\cd (e_{m+1}\cw v) +(-1)^{|\eta|+1} \sum\limits_{i=1}^n x^\al D_i(y_\eta)\otimes y_iD_1 \cd (e_{m+1}\cw v)\cr
					=& (r+1)\{\sum\limits_{i=1}^m \partial_i(x^\al)y_\eta\otimes e_{i}\cw v +(-1)^{|\eta|+1} \sum\limits_{i=1}^n x^\al D_i(y_\eta)\otimes e_{m+i}\cw v \}\cr
					=&(r+1)(-1)^{k} \{\sum\limits_{i=1}^m \partial_i(x^\al)y_\eta\otimes v \cw e_{i} +(-1)^{|\eta|+\wp(v)+1} \sum\limits_{i=1}^n x^\al D_i(y_\eta)\otimes v \cw e_{m+i} \}\cr
					=&(r+1)(-1)^{k} \mfk d_k(x^\al y_\eta \ot v)  \in W,
				\end{align*}
				since $r \geq 0$, this proves that $W=\im \mfk d_{k}$.	Moreover, from the above computation we also have {$ \im \mfk d_{k} \subseteq U(g)(1 \ot L_0(\omega_{k+1})) = \soc(\mcal V(w_{k+1})) \cong L(\omega_{k+1})$.} Hence completes the proof.

			\end{proof}
			\subsubsection{Irreducible module $L(\omega_k)$}	Combining the above Lemmas we have the following theorem:
			\begin{thm}\label{thm: 4.3} The following statements hold.
				\begin{itemize}
					\item[(1)] $\im \mfk d_k= \ker \mfk d_{k+1}$ is irreducible and  isomorphic to $ L(\omega_{k+1})$, for all $k \geq 0$
					\item[(2)] $\calv(\omega_k)$ has two composition factors of multiplicity free: $L(\omega_k)$, $L(\omega_{k+1})$ for $k\geq1$.
				\end{itemize}	
			\end{thm}


				\subsection{Exceptional mixed-product modules of type II} 
			
			Now we discuss about the modules $\mcal V(\theta_q)$ for all $ q\geq 0$.	Denote $\mcale=\sum_{i=1}^m\epsilon_i -\sum_{j=1}^n \delta_j $, then  $\theta_q=\mcale -q\delta_n$ for $q\geq 0$. Consider $M_q=\Omega^q(V^*) \ot \bbf_{\mcale}$, where $\bbf_{\mcale}$ is the one dimensional $\mf{gl}(m|n)$ module with highest weight $\mcale$. Then  by Corollary \ref{Cor for irr V*}, $M_q$ is an irreducible module with highest weight vector $e_{m+n}^{*q} \ot 1_{\mcale}$ and highest weight as $\theta_q$. Hence $L_{0}(\theta_q)$ can be realized as $L_0(\theta_q)=\Omega^q(V^*) \ot \bbf_{\mcale}$.  Here and further, for notational convenience  we write elements of $\Omega^k(V^*)$ just by product instead to wedge product.
			
			\begin{lem}\label{Generating set of V(la)}
				For $q \geq 0$,	$\mcal V(\theta_q)$ is generated by $y_n \ot {e_{m+n}^{*q}} \ot 1_{\mcale}$ as $\bW(m,n)$-module.
			\end{lem}
			\begin{proof}
				To prove this just note that, for all $\al \in \N^m$ we have,
				$$ x^{\al+\epsilon_m} y_1 \cdots y_{n-1}\pr_m \cdot (y_n \ot {e_{m+n}^{*q}} \ot 1_{\mcale})=(\al_m+1)x^{\al} y_1 \cdots y_{n} \ot ({e_{m+n}^{*q}} \ot 1_{\mcale}).$$
				This completes the proof, thanks to Lemma \ref{l3.12}.	
			\end{proof}
			
			\subsubsection{Exceptional mixed-product complex of type II}\label{sec: differential type II}			
			Now we construct a sequence of $\bW(m,n)$-module maps given by the following:
			\begin{align}\label{eq:type II complex}
				\mcal \cdots\cdots\longrightarrow\mathcal{V}(\theta_{q+1})\xrightarrow{\,\,\,\mathbf d_{q+1}\,\,\,}\mathcal{V}(\theta_q)\xrightarrow{\,\,\,\mathbf d_q\,\,\,} \mathcal{V}(\theta_{q-1})\xrightarrow{\,\,\,\mathbf d_{q-1}\,\,\,}\cdots\cdots\mathcal{V}(\theta_{1})\xrightarrow{\,\,\,\mathbf d_{1}\,\,\,}\mcal V(\theta_0) \xrightarrow{\,\,\,\mathbf d_{0}\,\,\,} 0,
			\end{align}
			where 	$$\mathbf d_{q}(y_n \ot {e_{m+n}^{*q}} \ot 1_{\mcale})=1 \ot {e_{m+n}^{*{q-1}}} \ot 1_{\mcale}.$$ \\
			Note that the above sequence of maps is well defined by Lemma \ref{Generating set of V(la)}. Also observe that for all $X \in U(\mf g)$ we have
			\begin{align}
				&\mbf d_q \mbf d_{q+1}(X \cdot y_n \ot {e_{m+n}^{*{q+1}}} \ot 1_{\mcale}) \cr
				= &\mbf d_q(X \cdot 1 \ot {e_{m+n}^{*{q}}} \ot 1_{\mcale}) \cr
				=& X \cd \mbf  d_q( 1 \ot {e_{m+n}^{*{q}}} \ot 1_{\mcale}) \cr
				=& X \cd \mbf d_q( D_n.(  y_n \ot {e_{m+n}^{*{q}}} \ot 1_{\mcale})) \cr
				=& X \cd D_n \cd \mbf d_q( y_n \ot {e_{m+n}^{*{q}}} \ot 1_{\mcale}) \cr
				=& X \cd D_n \cd ( 1 \ot {e_{m+n}^{*{q-1}}} \ot 1_{\mcale}) =0.
			\end{align} 				
			Hence we have $\text{Im} \mbf d_{q+1} \subseteq \text{Ker}\mbf  d_q$ for all $q \geq 0. 	$ To prove other inclusion relation we start with the following Lemma.
			\begin{lem}\label{expression for d_q}
				For all $q \geq 0$, $\eta \subseteq [1,n-1]$, $0 \leq s+t \leq q$, we have the following:\\
				\begin{itemize}
					\item[(1).]
					
					$\mbf d_{q+1}(x^\al y_{\eta}y_n \ot (e^*_{i_1}\cdots e^*_{i_s}) e^*_{m+j_1} \cdots e^*_{m+j_t}) e^{*{q+1-s-t}}_{m+n} \ot 1_{\mcale})$
					$$= \frac{1}{q+1}[- \dis{\sum_{k=1}^{t}} x^\al y_n D_{j_k}(y_\eta) \ot F_{m+j_k}(e^*_{i_1} \cdots e^*_{i_s} e^*_{m+j_1} \cdots e^*_{m+j_t} e^{*{q+1-s-t}}_{m+n}) \ot 1_{\mcale} $$
					$$+(-1)^{|\eta|+1}\dis{\sum_{l=1}^{s}} \pr_{i_l}(x^\al) y_n y_\eta \ot F_{i_l}(e^*_{i_1} \cdots e^*_{i_s} e^*_{m+j_1} \cdots e^*_{m+j_t} e^{*{q+1-s-t}}_{m+n}) \ot 1_{\mcale}   $$
					$$ +(-1)^s \, (q+1-s-t) x^\al y_{\eta} \ot e^*_{i_1} \cdots e^*_{i_s} e^*_{m+j_1} \cdots e^*_{m+j_t} e^{*{q-s-t}}_{m+n} \ot 1_{\mcale} ].$$
					\item[(2).]	$\mbf d_{q+1}(x^\al y_{\eta} \ot e^*_{i_1} \cdots e^*_{i_s} e^*_{m+j_1} \cdots e^*_{m+j_t} e^{*{q+1-s-t}}_{m+n} \ot 1_{\mcale})$
					$$= \frac{1}{q+1}[(-1)^{|\eta|+1} \dis{\sum_{k=1}^{t}} x^\al  D_{j_k}(y_\eta) \ot F_{m+j_k}(e^*_{i_1} \cdots e^*_{i_s} e^*_{m+j_1} \cdots e^*_{m+j_t} e^{*{q+1-s-t}}_{m+n}) \ot 1_{\mcale} $$
					$$-\dis{\sum_{l=1}^{s}} \pr_{i_l}(x^\al) y_\eta \ot F_{i_l}(e^*_{i_1} \cdots e^*_{i_s} e^*_{m+j_1} \cdots e^*_{m+j_t} e^{*{q+1-s-t}}_{m+n}) \ot 1_{\mcale} ]$$
					where  $1 \leq i_1 < \cdots < i_s \leq m$ and  $ 1 \leq j_1 \leq \cdots \leq j_t \leq n-1$ with super derivations $F_r$ given by $F_r(e^*_j)=\de_{rj}$ for all $1 \leq r,j \leq m+n$.	
				\end{itemize}
			\end{lem}				
			\begin{proof}
				First note that for all $\eta \subseteq [1,n-1]$ and $q >0$,  we have
				\begin{align}
					&(\al_m+1) \mbf d_q(x^{\al} y_\eta y_n \ot  e^{*q}_{m+n} \ot 1_{\mcale}  ) \cr
					=& \mbf d_q( x^{\al+\epsilon_m} y_\eta \pr_m \cd ( y_n \ot  e^{*q}_{m+n} \ot 1_{\mcale} )) \cr
					=&  x^{\al+\epsilon_m} y_\eta \pr_m \cd ( 1 \ot  e^{*{q-1}}_{m+n} \ot 1_{\mcale}  )) \cr
					=&(\al_m+1) x^\al y_\eta \ot  e^{*{q-1}}_{m+n} \ot 1_{\mcale}.
				\end{align}
				Hence $\mbf d_q(x^{\al} y_\eta y_n \ot  e^{*{q}}_{m+n} \ot 1_{\mcale}  ) =x^\al y_\eta \ot  e^{*{q-1}}_{m+n} \ot 1_{\mcale},$ for all $ \al \in \N^m, \eta \subseteq [1,n-1]$.
				Now consider the following expressions which are obtained by direct use of (\ref{W(n)-module structure}-\ref{W(n)-module structure-2}):
				\begin{align*}					
					&y_n\pr_{i_s} \cdots y_n \pr_{i_1}y_nD_{j_t} \cdots y_nD_{j_1}\cd (x^{\al} y_\eta y_n \ot  e^{*{q}}_{m+n} \ot 1_{\mcale} )\cr
					=& (-1)^{t+s(|\eta|+1)} \Perm{q}{s+t}x^\al y_{\eta}y_n \ot e^*_{i_1} \cdots e^*_{i_s} e^*_{m+j_1} \cdots e^*_{m+j_t} e^{*{q-s-t}}_{m+n} \ot 1_{\mcale}.
				\end{align*}				
				Again we have
				\begin{align*}
					&y_n\pr_{i_s} \cdots y_n \pr_{i_1}y_nD_{j_t} \cdots y_nD_{j_1}\cd (x^{\al} y_\eta \ot  e^{*{q-1}}_{m+n} \ot 1_{\mcale} )\cr	
					=& P^{q-1}_{s+t-1} [(-1)^{t-1 +s|\eta|+s}  \dis{\sum_{k=1}^{t}} x^\al y_n D_{j_k}(y_\eta) \ot F_{m+j_k}(e^*_{i_1} \cdots e^*_{i_s} e^*_{m+j_1} \cdots e^*_{m+j_t} e^{*{q-s-t}}_{m+n}) \ot 1_{\mcale}\cr &\qquad\qquad+(-1)^{t+(s-1)(|\eta|+1)} \dis{\sum_{l=1}^{s}} \pr_{i_l}(x^\al) y_n y_\eta \ot F_{i_l}(e^*_{i_1} \cdots e^*_{i_s} e^*_{m+j_1} \cdots e^*_{m+j_t} e^{*{q-s-t}}_{m+n}) \ot 1_{\mcale} ] \cr
					&+ P^{q-1}_{s+t} (-1)^{t+s|\eta|}x^\al y_{\eta} \ot e^*_{i_1} \cdots e^*_{i_s} e^*_{m+j_1} \cdots e^*_{m+j_t} e^{*{q-s-t-1}}_{m+n} \ot 1_{\mcale}.
				\end{align*}
				
				Now applying $y_n\pr_{i_s} \cdots y_n \pr_{i_1}y_nD_{j_t}\cdots y_nD_{j_1}$ on the both sides of $$\mbf d_q(x^{\al} y_\eta y_n \ot  e^{*{q}}_{m+n} \ot 1_{\mcale}  ) =x^\al y_\eta \ot  e^{*{q-1}}_{m+n} \ot 1_{\mcale}$$
				and using above two computations we have,
				\begin{align*}
					&\mbf d_q(x^\al y_{\eta}y_n \ot e^*_{i_1} \cdots e^*_{i_s} e^*_{m+j_1} \cdots e^*_{m+j_t} e^{*{q-s-t}}_{m+n} \ot 1_{\mcale})\cr
					= &\frac{1}{q}[(-1)^{1} \dis{\sum_{k=1}^{t}} x^\al y_n D_{j_k}(y_\eta) \ot F_{m+j_k}(e^*_{i_1} \cdots e^*_{i_s} e^*_{m+j_1} \cdots e^*_{m+j_t} e^{*{q-s-t}}_{m+n}) \ot 1_{\mcale}\cr
					&\quad +(-1)^{|\eta|+1}\dis{\sum_{l=1}^{s}} \pr_{i_l}(x^\al) y_n y_\eta \ot F_{i_l}(e^*_{i_1} \cdots e^*_{i_s} e^*_{m+j_1} \cdots e^*_{m+j_t} e^{*{q-s-t}}_{m+n}) \ot 1_{\mcale}\cr
					&\quad +(-1)^s \, (q-s-t) x^\al y_{\eta} \ot e^*_{i_1} \cdots e^*_{i_s} e^*_{m+j_1} \cdots e^*_{m+j_t} e^{*{q-s-t-1}}_{m+n} \ot 1_{\mcale}].
				\end{align*}
				Note that (ii) follows from the (i) just by taking action of $D_n$ on both sides.
			\end{proof}						
			Let us denote $w^{i,j}_{s,t,k}=	e^*_{i_1} \cdots e^*_{i_s} e^*_{m+j_1} \cdots e^*_{m+j_t} e^{*{k}}_{m+n} $ for all $k \in \N$. Now by putting $\mbf d_{q+1}(x^\al y_n y_{\eta} \ot w^{i,j}_{s,t,q+1-s-t}\ot 1_{\mcale})=0$ one can observe that, for all $\al \in \N^m, \eta \subseteq [1,n-1]$, we have
			\begin{align*}
				&(-1)^{s+1}\, (q+1-s-t) x^\al y_{\eta} \ot w^{i,j}_{s,t,q-s-t} \ot 1_{\mcale}\cr
				=& \frac{1}{q+1}[(-1)^{1} \dis{\sum_{k=1}^{t}} x^\al y_n D_{j_k}(y_\eta) \ot F_{m+j_k}(w^{i,j}_{s,t,q+1-s-t}) \ot 1_{\mcale}\cr
				&+(-1)^{|\eta|+1}\dis{\sum_{l=1}^{s}} \pr_{i_l}(x^\al) y_n y_\eta \ot F_{i_l}(w^{i,j}_{s,t,q+1-s-t}) \ot 1_{\mcale}].
			\end{align*}
			Now since in $\mcal V(\theta_q)$, $ 0 \leq s+t \leq q$, we have
			\begin{align*}
				\mcal S= \{ &x^\al y_n y_{\eta} \ot w^{i,j}_{s,t,q-s-t}\ot 1_{\mcale} + \text{Im} \mbf d_{q+1}, x^\al y_1 \cdots y_n \ot w^{i,j}_{s,t,0} \ot 1_{\mcale} + \text{Im} \mbf d_{q+1} \cr
				& \mid   q-s-t >0,  \al \in \N^m, \eta \subseteq [1,n-1],
				1 \leq i_1 < \cdots < i_s \leq m,
				1 \leq j_1 \leq \cdots \leq j_t \leq n-1\}
			\end{align*}
			is a spanning set for $\mcal V(\theta_q)/\text{Im} \mbf d_{q+1}$ (may not be linearly independent for all $q$). 	Since $\mbf d_{q+1} \circ \mbf d_{q}=0$, we have a factor map $\bar{ \mbf d}_{q}:\mcal V(\theta_{q})/\text{Im} \mbf d_{q+1} \to \mcal V(\theta_{q-1})$.
			\begin{rem}\label{Rem for independent q>m}
				\begin{itemize}
					\item[(1)]	For some fix $k $, $\eta \subseteq [1,n-1], j_1, \cdots, j_t <n$,  the weight of $x^{\al} y_n y_{\eta} \ot (e^*_{i_1} \cdots e^*_{i_s} e^*_{m+j_1} \cdots e^*_{m+j_t} e^{*{k}}_{m+n} \ot 1_{\mcale}) $ is
					{	$$\dis{\sum_{j=1}^{m}\al_j \epsilon_j + \sum_{j \in \eta } \de_j - \sum_{j=1}^{s}\epsilon_{i_j} }- \sum_{l=1 }^{t} \de_{j_l} -k \de_n + (\mcale +\de_n).$$}
					In particular, coefficient of $\de_n$ is $-k$, i.e for different $k$ all vectors of this type  have different weight.
					\item[(2)]  The weight of $x^{\al} y_1 \cdots y_n  \ot (e^*_{i_1} \cdots e^*_{i_s} e^*_{m+j_1} \cdots e^*_{m+j_t}  \ot 1_{\mcale})$ is
					{	$$\dis{\sum_{j=1}^{m}\al_j \epsilon_{j} + \sum_{j =1 }^{n-1} \de_j - \sum_{j=1}^{s}\epsilon_{i_{j}} }- \sum_{l=1 }^{t} \de_{j_l} + (\mcale +\de_n).$$}
					In particular, coefficient of $\de_n$ is $0$. Moreover, for distinct set $\{j_1, \cdots , j_t\}$ all vectors of this type have different weights.
					
				\end{itemize}
				
			\end{rem}
			\begin{lem}\label{Lem for q>m}
				For $q >m$, $\bar {\mbf d}_q(\mcal S)=\{\bar {\mbf d}_q(w): w \in \mcal S\}$ is linearly independent.
			\end{lem}
			
			\begin{proof}

				We prove it by the induction on the cardinality of finite subsets in $\mcal S$.\\
				At first we note that, for $q-s-t >0$, $\bar {\mbf d}_q(x^\al y_n y_{\eta} \ot w^{i,j}_{s,t,q-s-t} \ot 1_{\mcale}  + \im \mbf d_{q+1}) \neq 0$, follows immediately from Lemma \ref{expression for d_q}. Again $q >m$, so for vectors of the form $w^{i,j}_{s,t,0}$ we must have $t >0$ and hence $\bar {\mbf d}_q(x^\al y_1 \cdots y_n \ot w^{i,j}_{s,t,0} \ot 1_{\mcale} + \im \mbf d_{q+1}) \neq 0$, follows from Lemma \ref{expression for d_q}. Thus we have $\bar {\mbf d}_q(w) \neq 0$ for all vectors $w$ of $\mcal S$.
				
				Above induction principal holds for all subsets $S$ of $\mcal S$ with $\Card(S)=1$. Assume that the induction is true for set $S$ with $\Card(S) \leq R$.
				Let us consider a subset $S$ of $ \mcal S$ such that $\Card (S)=R+1$.  Now consider the relation $\dis{\sum_{i=1}^{R+1} } a_i \bar {\mbf d}_q(w_i ) =0$, where $w_i \in S$.
				
				{\bf Case I:} Suppose $S$ contains $l$ vectors of the form  $w_i=x^{\al^i} y_n y_{\eta} \ot w^{i,j}_{s,t,k} \ot 1_{\mcale} +\text{Im}{\mbf d}_{q+1}$ for some fixed $k \neq 0$, $i_1 ,\cdots, i_s; j_1, \ldots,  j_t$, $1 \leq i \leq l$. Also let $w_r$ be the vector such that $|\eta|$  and $||\al^r||$ are the maximum among all $w_i$, $1 \leq i \leq l$ ( note that there may exists more than one $\eta$ and $\al$ with the above property, we choose one of them).
				
				We claim that $\bar {\mbf d}_q(w_r)$ has a term  $x^{\al^r} y_{\eta} \ot e^*_{i_1} \cdots e^*_{i_s} e^*_{m+j_1} \cdots e^*_{m+j_t} e^{*{k-1}}_{m+n} \ot 1_{\mcale}$ which can not be obtained from any  $\bar {\mbf d}_q(w_i)$ with $i \neq r$, $1 \leq i \leq R+1$. We prove this claim as follows:
				\begin{align*}
					&\mbf d_{q}(x^\al y_{\eta}y_n \ot e^*_{i_1} \cdots e^*_{i_s} e^*_{m+j_1} \cdots e^*_{m+j_t} e^{*{q+1-s-t}}_{m+n} \ot 1_{\mcale})\cr
					= &\frac{1}{q}[- \dis{\sum_{k=1}^{t}} x^\al y_n D_{j_k}(y_\eta) \ot F_{m+j_k}(e^*_{i_1} \cdots e^*_{i_s} e^*_{m+j_1} \cdots e^*_{m+j_t} e^{*{q-s-t}}_{m+n}) \ot 1_{\mcale}\cr
					&+(-1)^{|\eta|+1}\dis{\sum_{l=1}^{s}} \pr_{i_l}(x^\al) y_n y_\eta \ot F_{i_l}(e^*_{i_1} \cdots e^*_{i_s} e^*_{m+j_1} \cdots e^*_{m+j_t} e^{*{q-s-t}}_{m+n}) \ot 1_{\mcale}   \cr
					& +(-1)^s \, (q-s-t) x^\al y_{\eta} \ot e^*_{i_1} \cdots e^*_{i_s} e^*_{m+j_1} \cdots e^*_{m+j_t} e^{*{q-s-t-1}}_{m+n} \ot 1_{\mcale} ].
				\end{align*}
				
				From the expression of ${\mbf d}_{q}$, it is clear that the above mentioned term can appear only from the last term of ${\mbf d}_{q}.$ Therefore the corresponding $w_i$ has to be of the form $x^{\al} y_n y_{\mu} \ot w^{i,j}_{s,t,k} \ot 1_{\mcale} +\text{Im}{\mbf d}_{q+1}$, for some $\al, \mu$. Now comparing the weight of both $w_r$ and $w_i$ of the above form with the maximality of $\al^r$ and $\eta$, we must have $\mu =\eta$ and $\al=\al^r.$  Moreover, by Remark \ref{Rem for independent q>m} that term can not appear from images of elements of the form $x^\al y_1 \cdots y_n \ot w^{i,j}_{s,t,0} \ot 1_{\mcale} + \im \mbf d_{q+1}$ due to weight reason.  This proves the claim. Hence we have $a_r=0$.

				{\bf Case II:} Now assume that $S$ contains only  vectors of the form $x^{\al} y_1 \cdots y_n \ot w_{s,t,0} + \text{Im}{\mbf d}_{q+1}$. Also let for $1 \leq r \leq l$,  $w_r=x^{\al^r} y_1 \cdots y_n \ot w^{i,j}_{s,t,0} \ot 1_{\mcale} +\im{\mbf d}_{q+1}$, for some fix $i_1 ,\cdots, i_s; j_1, \cdots,  j_t$. Let $||\al^1||= \max\{||\al^r|| : 1 \leq r \leq l\}$ (if more $\al$ satisfy this property we choose one of them). Since $q>m$ there exists $j \in \{ j_1, \cdots , j_t\}$. Now one can observe that $\bar {\mbf d}_q(x^{\al^1} y_1 \cdots y_n\ot w_{s,t,0}+ \im{\mbf d}_{q+1})$ has a term  $x^{\al^1} y_n D_{j}(y_1 \cdots y_{n-1}) \ot F_{m+j}(w_{s,t,0}) \ot 1_{\mcale}$ which can not be obtained from any $\bar{\mbf d}_q(w_r)$ with $r \neq 1$ follows from Remark \ref{Rem for independent q>m} and Lemma \ref{expression for d_q}. Hence $a_1$ has to be zero. This completes the proof by induction.
			\end{proof}
			One can observe from the above proof that, even if $q<m$ and $q-s-t>0$, then $\bar {\mbf d}_q(x^\al y_n y_{\eta} \ot w^{i,j}_{s,t,q-s-t} \ot 1_{\mcale} + \im {\mbf d}_{q+1}) \neq 0$. But in case $q<m$, some vectors of $\mcal S$ have image zero. For an example $\bar {\mbf d}_q(y_1 \cdots y_n \ot e_1 \cdots e_q \ot 1_{\mcale} +  \im{\mbf d}_{q+1}) =0.$ Therefore for $q <m$ we delete some vectors from the spanning set $\mcal S$ and find a new spanning set $\mcal S_1$ such that $\bar {\mbf d}_q(w) \neq 0$ holds for all $w \in \mcal S_1$.
			Since $q <m$ for every $w^{i,j}_{s,t,0}$ there exists a $i \notin \{i_1, \cdots, i_s\}$. Now fix $\al$ with $\al_{i_1}=\cdots =\al_{i_s}=0$. Then from $\mbf d_{q+1}(x^{\al+\epsilon_i}y_n y_1 \cdots y_{n-1} \ot w_{s,t.0}^{i,j}\ot 1_{\mcale})=0$, we have
			\begin{align*}				
				&x^\al y_1 \cdots  y_n \ot e^*_{i_1} \cdots e^*_{i_s} e^*_{m+j_1} \cdots e^*_{m+j_t}  \ot 1_{\mcale}\cr
				=&c \dis{\sum_{k=1}^{t}} x^{\al+\epsilon_i}y_n D_{j_k}(y_1 \cdots y_{n-1}) \ot F_{m+j_k}(e_i^*e^*_{i_1} \cdots e^*_{i_s} e^*_{m+j_1} \cdots e^*_{m+j_t} ) \ot 1_{\mcale},
			\end{align*}
			for some non-zero $c \in \bbf$.
			Hence for $q <m$, we describe a set $\mcal S_1$  consisting of elements
			\begin{align*}				
				&x^\al y_n y_{\eta} \ot w^{i,j}_{s,t,q-s-t} \ot 1_{\mcale} +  \im{\mbf d}_{q+1},  \cr
				&x^\ga y_1 \cdots y_n \ot w^{i,j}_{s,t,0}\ot 1_{\mcale} +  \im{\mbf d}_{q+1}, \cr
				& x^{\be} y_n y_{\eta_1} \ot w^{i,j}_{s,t,0} \ot 1_{\mcale} +				\im{\mbf d}_{q+1}
			\end{align*}
			with
			$\al, \be, \ga \in \N^m, q-s-t >0 ,\, \eta, \eta_1 \subseteq \{1, \cdots, n-1\}, \ga_{i_k} \neq 0$ for some  $1 \leq k \leq s$  and  $ \be_{i_k}=0$ for  $1 \leq k \leq s$, except for one $k,|\eta_1|=n-2$ and $1 \leq i_1 < \cdots < i_s \leq m$,  $1 \leq j_1 \leq \cdots \leq j_t \leq n-1$.
			Clearly $\mcal S_1$ is  a spanning set for $\mcal V(\theta_q)/\im \mbf d_{q+1}$, for $q<m$.
			
			\begin{rem}\label{Rem for q <m}
				For a fixed set $\{j_1, \cdots , j_t\}$ and $\eta_1 \subseteq [1,n]$ with $\Card(\eta_1)=n-2$, the  weights of
				$$x^{\al} y_1  \cdots y_n \ot (e^*_{i_1} \cdots e^*_{i_s} e^*_{m+j_1} \cdots e^*_{m+j_t}  \ot 1_{\mcale})$$
				and  {$$x^{\al} y_n  y_{\eta_1} \ot (e^*_{i_1} \cdots e^*_{i_s} e^*_{m+j_1} \cdots e^*_{m+j_t}  \ot 1_{\mcale})$$}
				are different for all $\al \in \N^m$.			
			\end{rem}
			\begin{lem}\label{Lem for q<m}
				For $q <m$, $\bar {\mbf d}_q(\mcal S_1)=\{\bar {\mbf d }_q(w): w \in \mcal S_1\}$ is linearly independent.	
			\end{lem}
			
			\begin{proof}
				It is easy to observe that $\bar {\mbf d}_q(w) \neq 0$ for all $w \in \mcal S_1$. Now we proceed similarly to Lemma \ref{Lem for q>m}. Let $S$ be any finite subset of $\mcal S_1$ and consider the relation
				$$\dis{\sum_{i=1}^{R+1} } a_i \bar {\mbf d}_q(w_i ) =0 \quad\text{ with }w_i \in S.$$

				{\bf Case I:} Suppose $S$ contains $l$ vectors of the form  $w_i=x^{\al^i} y_n y_{\eta} \ot w_{s,t,k} \ot 1_{\mcale} +  \text{Im}{\mbf d}_{q+1}$ for some $k \neq 0$, $1 \leq i \leq l$. Then with similar argument to Lemma \ref{Lem for q>m} we have the result.

				{\bf Case II:} Let $S$ contain only vectors of the form $x^\ga y_1 \cdots y_n \ot w^{i,j}_{s,t,0} \ot 1_{\mcale}+  \text{Im}{\mbf d}_{q+1}$, $ x^{\be} y_n y_{\eta_1} \ot w^{i,j}_{s,t,0} \ot 1_{\mcale} +  \text{Im}{\mbf d}_{q+1}$ for all $\ga, \be, \eta_1, w^{i,j}_{s,t,0}$. Fix some $w^{i,j}_{s,t,0}$. Since there exists some $\ga_{i_k}$ such that $\ga_{i_k} \neq 0$, let $w_r= x^\ga y_1 \cdots y_n \ot w^{i,j}_{s,t,0} \ot 1_{\mcale} +  \text{Im}{\mbf d}_{q+1}$ be such that $x_{i_k}^{\ga_{i_k}}$ is maximum among all $w_i$. Then $\bar{\mbf d}_q(w_r)$ has a term $ \pr_{i_k}(x^\ga) y_n y_1 \cdots y_{n-1}  \ot F_{i_l}(e^*_{i_1} \cdots e^*_{i_s} e^*_{m+j_1} \cdots e^*_{m+j_t} ) \ot 1_{\mcale}. $  We assert that this term can not be present in any other $\bar{\mbf d}_q(w_i)$ for $i \neq r$. It is clear that this term can not appear from images of any other $x^\ga y_1 \cdots y_n \ot w^{i^1,j^1}_{s^1,t^1,0} \ot 1_{\mcale} +  \text{Im}{\mbf d}_{q+1}$, due to weight reason, see Remark \ref{Rem for independent q>m}(2). Now note that \begin{align*}					
					&\bar d_{q}(x^\be y_{\eta_1}y_n \ot e^*_{i_1} \cdots e^*_{i_s} e^*_{m+j_1} \cdots e^*_{m+j_t}  \ot 1_{\mcale}+  \im{\mbf d}_{q+1})\cr
					= &\frac{1}{q}[- \dis{\sum_{k=1}^{t}} x^\al y_n D_{j_k}(y_{\eta_1}) \ot F_{m+j_k}(e^*_{i_1} \cdots e^*_{i_s} e^*_{m+j_1} \cdots e^*_{m+j_t} ) \ot 1_{\mcale}\cr
					&+(-1)^{|\eta|+1}\dis{\sum_{l=1}^{s}} \pr_{i_l}(x^\al) y_n y_{\eta_1} \ot F_{i_l}(e^*_{i_1} \cdots e^*_{i_s} e^*_{m+j_1} \cdots e^*_{m+j_t} ) \ot 1_{\mcale} ].
				\end{align*}
				From this it follows that the term  $\pr_{i_k}(x^\ga) y_n y_1 \cdots y_{n-1}  \ot F_{i_l}(e^*_{i_1} \cdots e^*_{i_s} e^*_{m+j_1} \cdots e^*_{m+j_t} ) \ot 1_{\mcale} $ can occur only from the second summation of above expression, which is absurd because $\Card(\eta_1)=n-2$. Hence $a_r=0$. Now we can use induction principal.

				{\bf Case III:} Let $S$ contains only vectors of the form $ x^{\be} y_n y_{\eta_1} \ot w^{i,j}_{s,t,0} \ot 1_{\mcale} +  \text{Im}{\mbf d}_{q+1}$ for all $ \be, \eta_1, w^{i,j}_{s,t,0}$. In this we fix some $w^{i,j}_{s,t,0}$ and find the $w_r$ such that $x_{i_k}^{\be_{i_k}}$ is maximum among all $w_i$. Then proceed similarly to Case II to conclude that $a_r=0$. This completes the proof by induction principal.
			\end{proof}
			\begin{lem}\label{Lem q=m}
				For $q=m$, $\ker  {\mbf d}_q/\im {\mbf d}_{q+1}$ is one dimensional trivial module.
			\end{lem}
			\begin{proof}
				One can see that, $\bar {\mbf d}_q( y_1 \cdots y_n \ot e_1^* \cdots e_m^* \ot 1_{\mcale} + \text{Im}{\mbf d}_{q+1}) =0$ and $\bar {\mbf d}_q(w) \neq 0$ for all other $w \in \mcal S.$  Now it follows with same method of Lemma \ref{Lem for q>m}  that $\{\bar {\mbf d}_q(w): w \in \mcal S, w \neq  y_1 \cdots y_n \ot e_1^* \cdots e_m^* \ot 1_{\mcale} + \text{ Im} \mbf d_{q+1} \}$ is linearly independent. Moreover, it is easy to observe from Lemma \ref{expression for d_q}(1) that $ y_1 \cdots y_n \ot e_1^* \cdots e_m^*\ot 1_{\mcale} \notin \text{Im}\mbf d_{q+1}$.
				Hence $\ker\mbf d_q/\text{Im}\mbf d_{q+1} =\spann_\bbf \{y_1 \cdots y_n \ot e_1^* \cdots e_m^* \ot 1_{\mcale} + \text{Im}\mbf d_{q+1}\}$. Also we have the following relations
				\begin{align*}				
					x^\al D_i.y_1 \cdots y_n \ot e_1^* \cdots e_m^* \ot 1_{\mcale}
					&= \mbf d_{m+1}(x^\al y_1 \cdots y_n \ot e_1^* \cdots e_m^*e_{m+1}^* \ot 1_{\mcale}), \, \forall \, \al  \in \N^m,\cr
					x^\al y_\eta \pr_i \cd y_1 \cdots y_n \ot e_1^* \cdots e_m^* \ot 1_{\mcale}&=0, \forall  \, \al, \eta
				\end{align*}
				and
				$$				x^\al y_\eta D_i \cd y_1 \cdots y_n \ot e_1^* \cdots e_m^* \ot 1_{\mcale}=0 , \, \forall \, \al, \eta \neq \emptyset.$$
				With the above relations it is clear that $\text{Ker}\mbf d_q/\text{Im}\mbf d_{q+1}$ is a trivial module.
			\end{proof}

			\begin{thm}\label{thm on V(la_q)} The following statements hold.
				\begin{itemize}
					\item[(1)] $\mcal V(\theta_q)$ is reducible for all $q \neq 0$.
					\item[(2)] $\im \mbf d_{q+1} \simeq L(\theta_q)$, for all $q \geq 0$.
					\item[(3)] $\ker \mbf  d_q=\im\mbf d_{q+1}$, for all $q \neq m$.
					\item[(4)] For $q=m$, $\ker\mbf d_m/\im\mbf d_{m+1}$ is one dimensional trivial module.
				\end{itemize}
			\end{thm}
			
			\begin{proof}
				
				Note that $\mbf d_{q+1}(X.(y_n  \ot e_{m+n}^{*{q+1} } \ot 1_\mu))= X. ( 1\ot e_{m+n}^{*{q}} \ot 1_\mu)$ for all $X \in U(\mf g)$. Hence $\text{Im}\mbf d_{q+1} \subseteq U(\mf g)(1 \ot L_0(\theta_q))$.
				Hence by Corollary \ref{Soc of V(la)}, $\text{Im}\mbf d_{q+1}$ is irreducible and isomorphic to $L(\theta_q)$. Now the theorem follows from Lemmas \ref{Lem for q>m}, \ref{Lem for q<m} and \ref{Lem q=m}.
			\end{proof}

			\section{Independence of differential operators and the irreducibility of non-exceptional mixed-product modules}	\label{sec: skry theory}

			\subsection{The category of $(\calr,\ggg)$-modules}
				We define a new category of modules for Lie superalgebra $\mf g$ following \cite{Skr} and study some properties of this category to understand the irreducibility of the modules $\mcal V(\la)$.

				\begin{defn}
					Denote by $(\calr,\ggg)\hmod$ the category whose objects are additive groups $M$ endowed with an $\mcal R$-module structure $f_\mcal R$, a $\mf g$-module structure $\rho_\mf{g}$  and a  $\mf g_{\geq 0}$-module structure $\sigma$ so that the following properties are satisfied:
					\begin{itemize}
						\item[(i)] $[\rho_\mf{g}(D),f_\mcal R]=(Df)_{\mcal R}$;
						\item[(ii)] $[\sigma(D'),f_\mcal R]=0$;
						\item[(iii)] $[\rho_\mf{g}(X),\sigma(D')]= 0$;
						\item[(iv)] $\rho_\mf{g}(fX)=f_\mcal R \circ \rho_\mf{g}(X) +\dis\sum_{i=1 }^{m}(\pr_i f)_{\mcal R}\circ \sigma (x_i X) +(-1)^{\wp(f)+1} \sum_{j =1 }^{n}(D_j f)_{\mcal R}\circ \sigma (y_j X)$
					\end{itemize}
					for all $f \in \mcal R, D \in \mf g, D' \in \mf g_{\geq 0}, X=\pr_i, D_j, i \in [1,m], \, j \in [1,n]. $
					
					We call modules in this category as $(\mcal R, \mf g)$-modules. This category will be denoted by $(\calr,\ggg)\hmod$, whose objects are usually called $(\calr,\ggg)$-modules.
				\end{defn}
				Let $V$ be a module for $\mf g_{\geq 0}$ with trivial action of $\mf g_{\geq 1}$. Then we define $\sigma(\mf g)$-module action on $\mcal R \ot V$ as follows:
				\begin{align}\label{sigma action}
					fX.(h \ot v)=(-1)^{\wp(h)\wp(fX)}h \ot fX.v,
				\end{align}					
				for all $f,h \in \mcal R, v \in V, X\in \{\pr_i,D_j\mid 1\leq i \leq m, 1 \leq j \leq n\}$. Then with respect to the left multiplication as $\mcal R$ action, $\sigma(\mf g)$ as above and $\rho_{\mf g}$ actions as (\ref{W(n)-module structure}), (3.3.2) makes $\mcal R \ot V$ an object of the category $(\calr,\ggg)\hmod$. In particular, $\mcal V(\la)$ belongs to  $(\calr,\ggg)\hmod$ for $\lambda\in \Lambda^+$.	 	
				
				
				\subsection{Properties of the category $(\calr,\ggg)\hmod$}		
				\begin{lem}\label{Lemma R, g submod}
					Let $M$ be an  $(\mcal R, \mf g)$-module such that $M$ is annihilated by $\sigma(g_{\geq 1})$. Then every  $\mcal R$- and $\rho(\mf g)$-submodule  $M'$ of $M$ is $(\mcal R, \mf g)$-submodule.
				\end{lem}		
				\begin{proof}
					To prove the statement we need to show that $M'$ is a $\sigma(\mf g)$-submodule of $M$. Note that by the assumption we only need to show $M'$ is invariant under $\sigma(\mf g_{0})$-action. Now from property (iv) of Definition 4.7 we have,
					\begin{align}\label{eqn R,g sub mod}
						f\rho(gX) &=f(g\rho(X) +\sum_{j=1  }^{m}\pr_j g \sigma (x_jX ) + (-1)^{\wp(g)+1} \sum_{j =1 }^{n}D_j g \sigma (y_j X) ),
					\end{align}
					for all $f,g \in \mcal R$ and $X \in \{\pr_i, D_j: i \in [1,m], \, j \in  [1,n]\}$.
					Now we put various $f, g$ and $X$ in the equation (\ref{eqn R,g sub mod}) to prove that $ \sigma(x_k\pr_i), \sigma(x_kD_s), \sigma(y_lD_s),$ $ \sigma(y_l\pr_i) $ leaves $M'$ invariant, for all $i,k \in [1,m]$ and $s,l \in [1,n]$.\\
					
					(i)	 Put $f=f_1=1, g=g_1=x_k$ and $f=f_2=-x_k, g=g_2=1$ in the equation (\ref{eqn R,g sub mod}) and take \\
					\begin{align*}\dis{\sum_{r=1}^{2}f_r\rho(g_r\pr_i)} &=\dis{\sum_{r=1}^{2}}f_rg_r\rho(\pr_i) +\sum_{j=1  }^{m}\dis{\sum_{r=1}^{2}}f_r\pr_j g_r \sigma (x_j\pr_i ) +\sum_{j =1 }^{n}\dis{\sum_{r=1}^{2}}(-1)^{\wp(g_r)+1} f_r D_j g_r \sigma (y_j \pr_i) ) \cr
						&=   \sum_{j =1 }^{m}\dis{\sum_{r=1}^{2}}f_r\pr_j g_r \sigma (x_j\pr_i )\cr
						&=\sigma(x_k\pr_i)
					\end{align*}
					where the second last equation holds because
					\begin{align}\label{eq: reason}
						\dis{\sum_{r=1}^{2}}f_rg_r=0 \;\text{ and }\;D_j(g_r) =0\;\text{ for }\;r=1,2.
					\end{align}
					Again,
					\begin{align*}\dis{\sum_{r=1}^{2}f_r\rho(g_rD_s)} &=\dis{\sum_{r=1}^{2}}f_rg_r\rho(D_s) +\sum_{j =1 }^{m}\dis{\sum_{r=1}^{2}}f_r\pr_j g_r \sigma (x_jD_s ) +\sum_{j =1 }^{n}\dis{\sum_{r=1}^{2}}(-1)^{\wp(g_r)+1} f_r D_j g_r \sigma (y_j D_s) ) \cr
						&=   \sum_{j=1  }^{m}\dis{\sum_{r=1}^{2}}f_r\pr_j g_r \sigma (x_jD_s )\cr
						&=\sigma(x_kD_s)
					\end{align*}
					where the second last equation holds because of (\ref{eq: reason}).
					
					(ii) 	Put $f=f_1=1, g=g_1=y_l$ and $f=f_2=-y_l, g=g_2=1$ in the  equation (\ref{eqn R,g sub mod})  and take \\
					\begin{align*}\dis{\sum_{r=1}^{2}f_r\rho(g_r\pr_i)} &=\dis{\sum_{r=1}^{2}}f_rg_r\rho(\pr_i) +\sum_{j=1  }^{m}\dis{\sum_{r=1}^{2}}f_r\pr_j g_r \sigma (x_j\pr_i ) +\sum_{j =1 }^{n}\dis{\sum_{r=1}^{2}}(-1)^{\wp(g_r)+1} f_r D_j g_r \sigma (y_j \pr_i) ) \cr
						&= \sum_{j =1 }^{n}\dis{\sum_{r=1}^{2}}(-1)^{\wp(g_r)+1} f_r D_j g_r \sigma (y_j \pr_i)\cr
						&=\sigma(y_{l}\pr_i).
					\end{align*}	
					Again,	
					\begin{align*}\dis{\sum_{r=1}^{2}f_r\rho(g_rD_s)} &=\dis{\sum_{r=1}^{2}}f_rg_r\rho(D_s) +\sum_{j=1  }^{m}\dis{\sum_{r=1}^{2}}f_r\pr_j g_r \sigma (x_jD_s ) +\sum_{j =1 }^{n}\dis{\sum_{r=1}^{2}}(-1)^{\wp(g_r)+1} f_r D_j g_r \sigma (y_j D_s) ) \cr
						&= \sum_{j =1 }^{n}\dis{\sum_{r=1}^{2}}(-1)^{\wp(g_r)+1} f_r D_j g_r \sigma (y_j \pr_i)\cr
						&=\sigma(y_{l}D_s)
					\end{align*}
					where  a similar argument like  (\ref{eq: reason}) is also involved in  the above equations.

					Now it is clear from (i) and (ii) that $ \sigma(x_k\pr_i), \sigma(x_kD_s), \sigma(y_lD_s),$ $ \sigma(y_l\pr_i) $ leaves $M'$ invariant, because $M'$ is both $\mcal R$ and $\rho(\ggg)$-module, for all $i,k \in [1,m]$ and $s,l \in [1,n]$.	
					This completes the proof.					
				\end{proof}

				\begin{thm}\label{Thm g submod}
					Let $M$ be a $(\mcal R, \mf g)$-module such that  $M$ is annihilated by $\sigma(g_{\geq 1})$. Additionally,  assume that $M$ is completely reducible as a $\mf g_{\geq 0}$-module with finite dimensional irreducible components and none of its irreducible components have exceptional highest weight. Let $M'$ be any  $\rho(\ggg)$-submodule of $M$ containing a highest weight vector $v_\la$ of $\sigma(\mf g_{\geq 0})$ action such that $\rho_{\mf g}(\mf g_{-1})(v_\la)=0$.  Then $M'$ is a $(\mcal R, \mf g)$-submodule.
				\end{thm}
				\begin{proof}
					In view of Lemma \ref{Lemma R, g submod}, we need to show that any $\rho(g)$-submodule $M'$ is $\mcal R$-submodule. Let $P=\{m \in M: \mcal R.m \subseteq M'\}$ be the largest $\mcal R$-submodule contained in $M'$ and $Q=\mcal R .M'$ be the smallest $\mcal R$-submodule containing $M'$. It follows immediately from Definition 4.3 that $P,Q$ are $\mf g$-submodules and hence $(\mcal R, \mf g)$-submodules. Therefore we can construct a $\mf g$-submodule $M'/P$ of $Q/P$. This observation shows that we can assume that $M'$ contains no non-zero $\mcal R$-submodules of $M$ and $\mcal R .M'=M$, then to complete the proof we need to show that $M=0$. We proceed similarly like \cite{Skr}. Let $S$ and $S'$ be the algebras generated by $\{\sigma(X): X \in \mf g_{\geq 0}\}$ and $\{ \rho(X): X \in \mf g \}$ respectively.  We are looking for endomorphisms $\phi$ of $M$ lying in $S$ with the property that for any $f \in \mcal R$ the endomorphism $f \phi \in S'.$ Now the $\mf g$-submdoule $M'$ is stable under $f \phi$ for any $f \in \mcal R$ and hence it contains the $\mcal R$-submodule $\mcal R \phi(M')$. Now by hypothesis $\phi(M')=0$ and from Definition 4.3 (ii), $\phi$ is an $\mcal R$-module homomorphism. Hence we have $\phi(M)=0$, i.e $\phi =0$. This gives us certain relation between endomorphisms of $\sigma(D')$.
					At the same time if a class of polynomial $f_\nu, g_\nu \in \mcal R$ satisfy the property
					$$\dis \sum_{\nu}	\rho(ff_{\nu}X)\rho(g_{\nu}Y)\cd v_\la=f\sigma(Z) \cd v_\la,$$
					for $X,Y \in \mf g_{-1}$ and $Z \in \mf g_0$ and $f \in \mcal R$, then we get $\sigma(Z) \cd v_\la=0  $. Otherwise $\mcal R \ot \sigma(Z) \cd v_\la$ will be a non-zero $\calr$-submodule contained in $M'$.
					In the subsequent section using the above properties we compute the restrictions on the action of $\sigma(\mf g_{\geq 0})$, which turns out the highest weight to be exceptional weights. We complete the proof of this theorem with the following arguments.
				\end{proof}

				\subsection{Applications of the above properties  and relations arising }

				For $1\le i,i^{'}\le m$ and a class of elements $f_\nu, g_\nu \in \mcal R$, we consider $\dis \sum_{\nu}	\rho(ff_{\nu}\pr_{i})\rho(g_{\nu}\pr_{i^{'}})$, {which gives us (see  \S\ref{sec: app B} Appendix B, the equation (\ref{equn 1st rel}))
					\begin{align*}
						&\dis \sum_{\nu}	\rho(ff_{\nu}\pr_{i})\rho(g_{\nu}\pr_{i^{'}})\cr
						= & f(\delta_{ik}\sigma(E_{k^{'}i^{'}})+\delta_{ik^{'}}\sigma(E_{ki^{'}})-\sigma(E_{ki})\sigma(E_{k^{'}i^{'}})-\sigma(E_{k^{'}i})\sigma(E_{ki^{'}})), \, for \, k\ne k^{'},
					\end{align*}		
					where $f_\nu, g_\nu$ depends on the polynomials involving $x_k, x_{k'}$.}
				Hence by the discussion of Theorem \ref{Thm g submod} we have, $$\delta_{ik}\sigma(E_{k^{'}i^{'}})+\delta_{ik^{'}}\sigma(E_{ki^{'}})-\sigma(E_{ki})\sigma(E_{k^{'}i^{'}})-\sigma(E_{k^{'}i})\sigma(E_{ki^{'}})=0.$$ In a similar way we get, action of $\sigma(\mf g_{\geq 0})$ at $v_\la$ satisfy all the following relations (see  \S\ref{sec: app B} Appendix B).
				
				\subsubsection{} Situation (1): For $1\le i,i^{'}\le m$  and $1\le k,k^{'}\le m$, we have
				\begin{align}\label{eq:i)Case1}
					\left\{\begin{array}{lll}
						\delta_{ik}\sigma(E_{k^{'}i^{'}})+\delta_{ik^{'}}\sigma(E_{ki^{'}})-\sigma(E_{ki})\sigma(E_{k^{'}i^{'}})-\sigma(E_{k^{'}i})\sigma(E_{ki^{'}})=0, &k\ne k^{'}\\
						\delta_{ik}\sigma(E_{ki^{'}})-\sigma(E_{ki})\sigma(E_{ki^{'}})=0,&k=k^{'}
					\end{array}\right.
				\end{align}

				Situation (2): For $1\le i\le m$, $m+1\le i^{'}\le m+n$
				and $1\le k\le m, m+1\le k^{'}\le m+n$,we have
				\begin{align}\label{key eqn ty weight 1}
					\delta_{ik}\sigma(E_{k^{'}i^{'}})-\sigma(E_{ki})\sigma(E_{k^{'}i^{'}})-\sigma(E_{k^{'}i})\sigma(E_{ki^{'}})=0,
				\end{align}

				\subsubsection{}Situation (3): For $m+1\le i,i^{'}\le m+n$ and $m+1\le k,k^{'}\le m+n,k\ne k^{'}$, we have
				\begin{align}\label{eq:(iv)Case3}
					\delta_{ik}\sigma(E_{k^{'}i^{'}})-\delta_{ik^{'}}\sigma(E_{ki^{'}})+
					\sigma(E_{ki})\sigma(E_{k^{'}i^{'}})-\sigma(E_{k^{'}i})\sigma(E_{ki^{'}})=0.
				\end{align}

				\subsection{Arguments on  the highest weights for $\sigma_{\ggg_0}$}	Let $\lambda=\sum\limits_{i=1}^{m}\lambda_{i}\epsilon_{i}+\sum\limits_{i=1}^{n}\lambda_{m+i}\delta_{i}$ and $v_{\lambda}$ be a highest weight vector {of the representation $\sigma=\sigma({\ggg_{\geq 0}})$}. Then with the relations above, we can get some restrictions on $\lambda$.
				
				\subsubsection{} In equation (\ref{eq:i)Case1}), we set $i=i^{'}=k=k^{'}$. Then we have $\sigma(E_{ii}) \cd v_{\lambda}=\sigma(E_{ii})^{2}\cd v_{\lambda}$. This gives $\lambda_{i}=0$ or $ 1$. If we set $i=k<i^{'}=k^{'}$,
				then $$(\sigma(E_{i^{'}i^{'}})-\sigma(E_{ii})\sigma(E_{i^{'}i^{'}})-\sigma(E_{i^{'}i})\sigma(E_{ii^{'}}))\cd v_{\lambda}=0.$$
				From this we get, $\lambda_{i^{'}}-\lambda_{i}\lambda_{i^{'}}=\lambda_{i^{'}}(1-\lambda_{i})=0$. Therefore we have the following relations
				\begin{align}\label{eq: 4.17.4}
					&\lambda_{i^{'}}(1-\lambda_{i})=0, \, \forall \, 1 \leq i < i' \leq m;\cr
					&\la_i=0 \, \, \text{or}  \,\, 1.
				\end{align}
				
				In equation (\ref{key eqn ty weight 1}), we set $i=k=s$, $i^{'}=k^{'}=m+t$  $(1\le s\le m, 1\le t\le n)$. Then it follows that $$(\sigma(E_{m+t,m+t})-\sigma(E_{s,s})\sigma(E_{m+t,m+t})-\sigma(E_{m+t,s})\sigma(E_{s,m+t})) \cd v_{\lambda}=0.$$ Hence we have
				\begin{align}\label{eq: type 1-1}
					\lambda_{m+t}-\lambda_{s}\lambda_{m+t}=\lambda_{m+t}(1-\lambda_{s})=0
				\end{align}
				
				In equation (\ref{eq:(iv)Case3}), we set $i=k=m+s$, $i^{'}=k^{'}=m+t$ $(1\le s<t\le n)$. Then it follows that $$(\sigma(E_{m+t,m+t})+\sigma(E_{m+s,m+s})\sigma(E_{m+t,m+t})-\sigma(E_{m+t,m+s})\sigma(E_{m+s,m+t}))\cd v_{\lambda}=0.$$ From this relation we get
				\begin{align}\label{eq: type 3}
					\lambda_{m+t}+\lambda_{m+s}\lambda_{m+t}=\lambda_{m+t}(1+\lambda_{m+s})=0
					\text{ for any  pair }\,  s, t \, \text{with} \, 1\leq s<t\leq n.
				\end{align}
				
				From the above fact, we have the following consequence.\\
				
				(i) If $\la_k =0$ for all $k$. Then $\la=0=\omega_0$.\\
				
				(ii) If $\lambda_{k} \neq 0 $ for some $k \leq m$, then from equations in (\ref{eq: 4.17.4}), (\ref{eq: type 1-1}) and (\ref{eq: type 3}), we have $\la= \epsilon_{1} + \cdots + \epsilon_{k}= \omega_k$. \\
				
				(ii)	Suppose  $\lambda_{m+1}\ne 0$ and $\lambda_{m+2}=0$.  Then $\lambda_{m+t}=0$ for  $t\ge 3$, by (\ref{eq: type 1-1}). Therefore we have $\la=\epsilon_{1}+\epsilon_{2}+\cdots+\epsilon_{m}+l\delta_{1}$. Furthermore, $l$ must be an positive integer because $\lambda\in \Lambda^+$, i,e $\la=\omega_{m+l}$, $l >0$.\\
				
				(iii) 	Suppose $\lambda_{m+1} \ne 0$ and $\lambda_{m+2} \neq 0$. By (\ref{eq: type 3}), we have $\lambda_{m+1}=-1$. By the requirement $\lambda\in \Lambda^+$, we further have that $\lambda_{m+j}\ne0$ for all $j\geq 2$. Again by (\ref{eq: type 3}), we have
				{$$\la = \epsilon_{1}+\epsilon_{2}+\cdots+\epsilon_{m}-\sum_{j=1}^n\delta_{j}-q\delta_n =\theta_q$$}
				with $q\geq 0$. This proves that the highest weight of the representation $\sigma(\mf g_{\geq 0})$ is exceptional, a contradiction. Hence $M=0$ by the hypothesis. This completes the proof of Theorem \ref{Thm g submod}.

				\begin{thm}\label{thm: 5.5}
					$\mcal V(\la)$ is irreducible if $\la \neq \omega_k $ or $\theta_q$, for $k, q \geq 1$.
				\end{thm}
				\begin{proof}
					Note that  $\mcal V(\la)$ is an $(\mcal R, \mf g)$-module. Also  $\mcal V(\la)= \dis{\bgop }_{\al, \mu}x^\al y_\mu \ot L_0(\la) $ and none of $\la$ is exceptional. Therefore $\mcal V(\la)$ is completely reducible (see ( \ref{sigma action})) as $\sigma(\mf g_{\geq 0})$-module and satisfy the property of Theorem \ref{Thm g submod}. Moreover every $\mf g$-submodule $M'$ of $\mcal V(\la)$ also satisfy the property of Theorem \ref{Thm g submod}. Hence  every $\mf g$-submodule $M'$ of $\mcal V(\la)$ is  a $(\mcal R, \mf g)$-submodule.
					
					Let $M'$ be a non-zero submodule of $\mcal V(\la)$. Then $1 \ot L_0(\la) \subseteq M'$. This proves that $M'=\mcal V(\la)$ with the  help of $\mcal R$-module action. This  completes the proof.
				\end{proof}

				\subsection{Irreducibility theorem} 	
				\begin{thm}\label{Final irr thm}
					The following statements hold.
					\begin{itemize}
						\item[(1)]	$\mcal V(\lambda)$  is irreducible if and only if $\lambda\ne\omega_k$ or $\theta_k$ for any $k \in \Z_+ $.
						\item[(2)] For all $k \geq 1$, $\mcal V(\omega_k)$ has two composition factors $L(w_k) $ and $L(\omega_{k+1})$ with free multiplicity.
						\item[(3)] For all $q \neq m$ and $q \geq 1$, $\mcal V(\theta_q)$ has two composition factors $L(\theta_{q})$ and $L(\theta_{q+1})$ with free multiplicity.
						\item[(4)] $\mcal V(\theta_m)$ has three composition factors $L(\theta_{m})$, $L(\theta_{m+1})$ and $L(0)$ with free multiplicity.
					\end{itemize}
				\end{thm}
				\begin{proof} It is a direct consequence of Theorems \ref{thm: 4.3}, \ref{thm on V(la_q)} and \ref{thm: 5.5}. %
				\end{proof}

				\subsection{Irreducible characters}
				Let $\mf g_0= \mf n^- \op \mf h \op \mf n^+$ be the triangular decomposition of $\mf g_0$ and $\Phi_{\bar 0}^+$ be the positive root system of $\mf g_0$.  Furthermore, denote by $\Phi_{\geq 1}$ be the root system of $\mf g_{\geq 1}$ relative to $\mf h$, i.e, $\Phi_{\geq 1}= \{\al \in \mf h^*: (\mf g_{\geq 1})_\al \neq 0\}  $, where
				$$\mf g_{\geq 1}=\{x \in \mf g_{\geq 1}: [h,x]=\al(h) x, \, \forall \, h \in \mf h\}.$$
				Then we have $\mf g_{\geq 1}=\dis{\sum_{\al \in \Phi_{\geq 1} }} \mf g_\al. $  For $\la \in \La^+$, associate a subset of $\mf h^*$ by $D(\lambda)=\{\mu\in \mf h^* : \mu \succeq \lambda\}$, where $\mu \succeq \lambda$ means that $\mu-\lambda$ lies in the $\Z_{\geq 0}$-span of $\Phi_{\geq 1}\cup \Phi^+_{\bar 0}$. Now we define an $\bbf$-algebra $\mcal B$ whose elements are series of the form $\sum_{\lambda\in {\mf h}^*}c_\lambda e^\lambda$ with $c_\lambda\in \bbf$ and $c_\lambda=0$ for $\lambda$ outside the union of a finite number of sets of the form $D(\mu)$.
				Consequently $\mcal B$ becomes a commutative associative algebra if we define $e^\la e^\mu = e^{\la+\mu}$ with $e^0$ as identity element. It is well known that formal exponents $\{e^\la : \la \in \mf h^*\}$ are linearly independent and are in one to one correspondence with $\mf h^*$. Let $W$ be a semi-simple $\mf h$-module with finite dimensional weight spaces such that $W=\dis{\sum_{\la \in \mf h^*} }W_\la. $ Then we define $\ch(W)=\sum_{\lambda\in {\mf h}^*}\dim W_\lambda e^\lambda$. In particular, if $V$ is an object in $\mathcal{O}$,
				then $\ch(V)\in \mathcal{B}$.

				\begin{lem}\label{Lem Char}  The following statements hold.
					\begin{itemize}
						\item[(1)] Let $V_1, V_2$ and $V_3$ be three $\ggg$-modules in the category $\co$. If there is an exact sequence of $\mf g$-modules $0\rightarrow V_1\rightarrow V_2\rightarrow V_3\rightarrow 0 $, then $\ch(V_2)=\ch(V_1)+\ch(V_3)$.
						\item[(2)] Suppose $W=\dis\sum_{\lambda \in \mf h^*}W_\lambda$ is a semi-simple $\mf h$-module with finite-dimensional weight spaces, and $U=\dis\sum_{\lambda}U_\lambda$ is a finite-dimensional $\mf h$-module. If $\ch(W)=\dis\sum_{\la \in \mf h^*}c_{\lambda}e^\lambda$ falls in $\mathcal{B}$, then $\ch(W\otimes U)$ must fall in $\mathcal{B}$ and $\ch(W\otimes U)=\ch(W)\ch(U)$.
					\end{itemize}
				\end{lem}
				Now we study formal character for $\Delta(\la)$, $\la \in \La^+$. Since $\Delta(\la)= U(\mf g_{\geq 1}) \ot_{\bbf} L^0(\la)$, so as a $U(\mf g_{\geq 1})$-module $\Delta(\la)$ is free module of rank dim $L_0(\la)$ and generated by $L_0(\la)$. Hence by Lemma \ref{Lem Char}, $\ch(\Delta(\la))=\ch( U(\mf g_{\geq 1})) \ch(L_0(\la))$. Note that $\Phi_{\geq 1}=\Phi_{\bar 0}^{\geq 1} \cup \Phi_{\bar 1}^{\geq 1},$ where $\Phi_{\bar i}^{\geq 1} = \Phi_{\geq 1} \cap \Phi_{\bar i}$, $i \in \Z_2$. Set
				
				$$\Upsilon=\prod\limits_{\alpha\in\Phi^{\geq 1}_{\bar 1}}
				(1+e^{\alpha})^{-1} \prod\limits_{\alpha\in\Phi^{\geq 1}_{\bar 0}}
				(1-e^{\alpha})^{-1}.$$
				Then we have $\ch(\Delta(\lambda))=\Upsilon\ch L_0(\lambda)$.
				
				{				
					\subsection{Character formula for $L(\la)$} Set $ \Gamma=\ch(R)$, where $\mcal R$ is viewed as a semi-simple $\mf h$-module. Then by Lemma \ref{Lem Char} we have $\ch (\mcal V(\la))= \ch(\mcal R \ot L^0(\la))= \Gamma \ch(L^0(\la))$. The following theorem presents the character formulas for irreducible modules of category $\mcal O$.
					\begin{thm}\label{thm: irr char}
						\begin{itemize}
							\item[(1)] \ch$(L(\la))	=\Gamma \ch(L^0(\la))$, for $\la \neq \theta_k, \omega_k$, $k \geq 1$.
							\item[(2)] $\ch(L(\omega_k))=\dis{\sum_{i=0}^{k} } (-1)^{2k-i}\Gamma \ch(L^0(\omega_{k-i})),  \,\,\, \forall \, k \geq 1.$
							\item[(3)] $\ch(L(\theta_q))=\dis{\sum_{i=0}^{q-1} } (-1)^{2q-i}\Gamma \ch(L^0(\theta_{q-i})) +(-1)^q\Gamma e^{\mcale},  \,\,\, \forall \, q <m.$
							\item[(4)] $\ch(L(\theta_q))=\dis{\sum_{i=0}^{q-1} } (-1)^{2q-i}\Gamma \ch(L^0(\theta_{q-i})) +(-1)^q\Gamma e^{\mcale} +(-1)^{q-m+1} \ch(L(0))$ for all $q \geq m$.
							
						\end{itemize}
						
					\end{thm}	
					\begin{proof}
						At first we note that from the character formula
						$$\ch(\mcal V(\la))=\dis{\sum_{\mu}}[\mcal V(\la): L(\mu)]L(\mu),   $$ we get the following inductive formulas with the help of Theorem \ref{Final irr thm}. Inductive character formulas for $L(\omega_k)$ are given by
						$$\ch(L(\omega_k))= \Gamma \ch(L^0(\omega_k)) - \ch(L(\omega_{k-1})),$$ for all $k \geq 1$, where $L(\omega_0)= \mcal R$. Further inductive character formulas for $L(\theta_q)$ are given by
						$$\ch(L(\theta_q))= \Gamma \ch(L^0(\theta_q)) - \ch(L(\theta_{q-1})) \, \, \text{and}$$
						$$\ch(L(\theta_m))= \Gamma \ch(L^0(\theta_m))- \ch(L(0))- \ch(L(\theta_{m-1})),$$
						for all  $q \geq 1$ with $  q \neq m $,  where $L(\theta_0)= \mcal R \ot \C_{\mcale}$, i.e, $\ch(L(\theta_0))=\Gamma e^{\mcale}$.
						From the above inductive formulas, we have the result.		
					\end{proof}	
				}

				\begin{rem}
					(1) As to the characters of finite-dimensional irreducible $\gl(m|n)$-modules $L^0(\lambda)$, it has been already known. The reader may refer to \cite{CW12} for typical weights, and \cite{B03, BLW, CLW15}, {\sl{etc}}. for general integral dominant weights.

					(2) In our arguments, we actually exclude the cases $m=n=1$ and $n=1$ for the technical reason although Theorem \ref{thm: 4.3}, Theorem \ref{thm on V(la_q)}(1)(2) are still true for those cases.
				\end{rem}

				\section{Tilting modules and character formulas}\label{sec: tilt sec}
				In this section, we determine the character formulas for tilting modules in $\mathcal{O}$.

				\subsection{Tilting modules}  
			\begin{defn}
				An object $M\in\mathcal{O}$ is said to admit a $\Delta$-flag if there exists an increasing filtration
				$$0=M_0\subset M_1\subset M_2\subset\cdots\cdots$$
				such that $M=\bigcup\limits_{i=0}^{\infty}M_i$ and $M_{i+1}/M_{i}\cong \Delta(\lambda_i)$ for all $i\geq 1$  where $\lambda_i\in\Lambda^+$ for all $i$.
			\end{defn}
			
			According to Soergel's construction (see \cite{Soe}) and its super counterpart (see \cite{JB}), we know that associated with each $(d,\lambda)\in \bbz\times \Lambda^+$, there exists a uniquely determined indecomposable tilting module ${}^dT(\lambda)$  whose $\Delta$-flag starts with ${}^d\Delta(\lambda)$. Moreover, if the semi-infinite character property defined in the next subsection is ensured, then following Soergel's reciprocity  we can read out the multiplicities of ${}^d\Delta(\mu)$ in the $\Delta$-flag of ${}^dT(\lambda)$ for every pair $(\lambda,\mu)$.
			
			We will first confirm that for $\ggg=\bW(m,n)$, the semi-infinite character property is satisfied,  then investigate tilting modules in $\co$ along Soergel's way. For the simplicity, we will ignore the depth $d$ in the following arguments.

			\subsection{Semi-infinite character}
			\begin{defn}
				Let $\textbf g=\displaystyle{\sum_{i \in \bbz}}\textbf g_i$ be a $\bbz$ graded Lie superalgebra with dim $\textbf g_i < \infty$ for all $i \in \bbz$. A character $\gamma: \textbf g_0 \to \bbf$ is said to be a semi-infinite character for $\textbf g$ if it satisfy the following properties.
				\begin{itemize}
					\item[(SI-1)] As a Lie superalgebra, $\textbf g$ is generated by $\textbf g_1, \textbf g_0$ and $ \textbf g_{-1}$.\\
					\item[(SI-2)] $\gamma([X,Y])=\str(\ad X\circ \ad Y \mid_{\textbf g_0})$, for all $X \in \textbf g_1$ and $Y \in \textbf g_{-1}$.
				\end{itemize}
			\end{defn}		
			
			\begin{thm}\label{semi-inf}
				A linear map $\mcal E $ is a semi-infinite character on $\bW(m,n)$ if and only if $\mcal E(x_i\partial_j)=\delta_{ij}$,  $\mcal E(y_r\partial_j)=0$, $\mcal E(x_iD_t)=0$ and $\mcal E(y_rD_t)=-\delta_{rt}$ for all $1 \leq i,j \leq m$ and $1 \leq r,t \leq n$. Consequently, $\mcale=\mcal E|_\hhh$.
				
			\end{thm}
			\begin{proof}
				At first we prove that as a Lie superalgebra $\mf g$ is generated by $\mf g_1, \mf g_0$ and $ \mf g_{-1}$. Let $S$ be the algebra generated by $\mf g_0,\mf g_1, \mf g_{-1}$.  Note that from \cite[Lemma 2.1]{DSY20}   and \cite[Lemma 1.3]{DSY24} we have $x^\alpha \pr_j$, $y_\mu D_j \in S$ for all $\al \in \N^m, \mu \subseteq [1,n]$, $i \in [1,m], j \in [1,n]$. Therefore we need to show $x_{i_1}\dots x_{i_k}y_{j_1}\dots y_{j_t}\pr_s$ and $x_{i_1}\dots x_{i_k}y_{j_1}\dots y_{j_t}D_s \in [\mf g_1, \mf g_{r-1}] $ for all $r \geq 2$. Now we consider few cases to complete the proof.\\
				
				{\bf Case I:} Let $x_{i_1}\dots x_{i_k}y_{j_1}\dots y_{j_t}\pr_s \in \mf g_r$ for some $r \geq 2$.\\
				
				Consider the following bracket
				$$[x_{i_l}y_{j_1}D_{j_1},x_{i_1}\dots x_{i_{l-1}}x_{i_{l+1}}\dots x_{i_k}y_{j_1}\dots y_{j_t}\pr_s]= x_{i_1}\dots x_{i_k}y_{j_1}\dots y_{j_t}\pr_s,  $$
				
				{\bf Case II:} Consider the vectors of the form $x_{i_1}\dots x_{i_k}y_{j_1}\dots y_{j_t}D_s \in \mf g_r,$ for some $r \geq 2$.\\
				{Subcase (i):} Let $k\geq 2$. Also let $i_2$ appears $l$ times in the set $\{i_2, \cdots, i_k\}$.\\
				Now consider $$ [x_{i_1}x_{i_2}\pr_{i_2},x_{i_2}\dots x_{i_k}y_{j_1}\dots y_{j_t}D_s]=lx_{i_1}\dots x_{i_k}y_{j_1}\dots y_{j_t}D_s
				$$
				{Subcase (ii):} Let $k= 1$.
				Since $x_ry_{j_1}\dots y_{j_t}D_s \in \mf g_r$ for some $r \geq 2$, so $\Card(\{j_1, \cdots, j_t\}) \geq 2$. So there exists $p \in \{1,2,\dots,t\}$ such that $s\neq j_p$. Assume that $p=1$ and consider
				$$[x_ry_{j_1}\pr_r,x_ry_{j_2}\dots y_{j_t}D_s ] = x_ry_{j_1}\dots y_{j_t}D_s.$$
				Therefore from Case I and Case II along with results of \cite{DSY20,DSY24} we have $S=\mf g$. We prove the (SI-2) property in  \S\ref{sec: app C} Appendix C. 			
			\end{proof}

			\subsection{Soergel reciprocity}\label{sec: soeg reci}
			The following result follows from \cite[Theorem 5.2]{Soe} and \cite[Threorem 5.1]{JB}.
			
			\begin{prop}\label{tilting}
				For each $\lambda\in\Lambda^+$, up to isomorphism, there exists a unique indecomposable object $T(\lambda)\in\mathcal{O}$ such that
				\begin{itemize}
					\item[(1)] $\Ext_{\mathcal{O}}^1(\Delta(\mu), T(\lambda))=0,\,\forall\,\mu\in\Lambda^+$.
					\item[(2)] $T(\lambda)$ admits a $\Delta$-flag, starting with $\Delta(\lambda)$ at the bottom.
				\end{itemize}
			\end{prop}
			
			We call the indecomposable module $T(\lambda)$ as the tilting module corresponding to $\lambda\in\Lambda^+$.
			Now we get the following tilting modules formula by Theorem \cite[Theorem 5.12]{Soe}, Corollary 5.8 of \cite{JB} together with the help of Theorem \ref{semi-inf}.
			\begin{prop} (Soergel reciprocity)\label{soegrel formula}
				Let $\lambda, \mu\in\Lambda^+$. 
				Then we have
				$$[T(\lambda):\Delta(\mu)]=[\mathcal{V}(-w_\circ\mu-\mathcal{E}_\circ):
				L(-w_\circ\lambda-\mathcal{E}_\circ)],$$
				where $w_\circ$ is the longest element in the Weyl group of $\ggg_{0}$.
			\end{prop}
			
			\begin{rem}\label{weyl act}
				It should be noted that
				\begin{equation*}
					w_\circ(X)=\begin{cases}
						\epsilon_{m+1-i},&\text{if}\,\, X=\epsilon_i, \, 1 \leq i \leq m\\
						\de_{n+1-i},&\text{if}\,\, X=\de_i, \, 1 \leq i \leq n.
					\end{cases}
				\end{equation*}
				
			\end{rem}
			
			We will precisely determine the multiplicity of the standard module $\Delta(\mu)$ occurring in the $\Delta$-filtration of the tilting module $T(\lambda)$ for any $\lambda, \mu\in\Lambda^+$.
			
			\begin{prop}\label{main thm1}
				Let $\lambda,\mu\in\Lambda^+$. Then the following statements hold.
				\begin{itemize}
					\item[(1)] If $\mu=-2\sum\limits_{i=1}^k\epsilon_{m+1-i}-\sum\limits_{i=k+1}^m\epsilon_{m+1-i} + \sum\limits_{i=1}^n\de_{i}$  for some $k$ with $1\leq k< m$. Then $[T(\lambda):\Delta(\mu)]\neq 0$ if and only if $\lambda=\mu$ or $\lambda=\mu-\epsilon_{m-k}$. Moreover, $[T(\mu):\Delta(\mu)]=[T(\mu-\epsilon_{m-k}):\Delta(\mu)]=1$.
					\item[(2)] If  $ \mu=-2\sum\limits_{i=1}^m\epsilon_{i} +\sum\limits_{i=1}^n\de_{i} -(k-m)\de_n $ for some $k \geq m$. Then $[T(\lambda):\Delta(\mu)]\neq 0$ if and only if $\lambda=\mu$ or $\lambda=\mu-\de_{n}$. Moreover, $[T(\mu):\Delta(\mu)]=[T(\mu-\de_{n}):\Delta(\mu)]=1$.
					
					\item[(3)] If $\mu=-2\mcale +k\de_1 $ and $k \neq m$, $k \in \Z_+$. Then $[T(\lambda):\Delta(\mu)]\neq 0$ if and only if $\lambda=\mu$ or $\lambda=\mu+\de_{1}$. Moreover, $[T(\mu):\Delta(\mu)]=[T(\mu+\de_{1}):\Delta(\mu)]=1$.
					\item[(4)]  If $\mu=-2\mcale +m\de_1. $  Then $[T(\lambda):\Delta(\mu)]\neq 0$ if and only if $\lambda=\mu$ or $\lambda=\mu+\de_{1}$ or $\la =\mu +\mcale -m\de_1$. Moreover, $[T(\mu):\Delta(\mu)]=[T(\mu+\de_{1}):\Delta(\mu)]=[T(\mu+\mcale -m\de_1):\Delta(\mu)]=1$.
					\item[(5)] If $\mu$ be none of type (1)-(4). Then $[T(\lambda):\Delta(\mu)]\neq 0$ if and only if $\lambda=\mu$. Moreover,  $[T(\mu):\Delta(\mu)]=1$.
					
				\end{itemize}
			\end{prop}
			
			\begin{proof}
				It follows from Proposition \ref{soegrel formula} that $[T(\lambda):\Delta(\mu)]=
				[\mathcal{V}(-w_0\mu-\mcale):L(-w_0\lambda-\mcale)]$, where we recall
				$\mcale=\sum\limits_{i=1}^m\epsilon_i - \sum\limits_{i=1}^n\de_i$.
				
				{\bf Case 1:} $\mathcal{V}(-w_0\mu-\mcale)$ is not simple. Then from Theorem \ref{Final irr thm} we have, $-w_0\mu-\mcale=\omega_k$ or $\theta_k$ for some $k \in \Z_+$, where
				\begin{align*}
					\omega_k=\begin{cases}
						\epsilon_{1}+\cdots+\epsilon_k, &\text{ if} \, 1 \leq k < m\cr
						\epsilon_{1}+\cdots+\epsilon_m +(k-m)\de_1,  & \text{ if}\,  k \geq m
					\end{cases}
				\end{align*}
				and $$\theta_k=
				\dis{\sum_{i=1}^{m}}\epsilon_{i} -\dis{\sum_{i=1}^{n}}\de_{i} -k \de_n, \, k \geq 1 $$

				{ Subcase (i):} Let $-w_0\mu-\mcale=\omega_k$ for some $k < m$. Then using Remark \ref{weyl act} we have, $\mu=-2\sum\limits_{i=1}^k\epsilon_{m+1-i}-\sum\limits_{i=k+1}^m\epsilon_{m+1-i} + \sum\limits_{i=1}^n\de_{i}$. Therefore by Theorem \ref{Final irr thm}, $[\mathcal{V}(-w_0\mu-\mcale):L(-w_0\lambda-\mcale)]\neq 0$ if and only if $-w_0\lambda-\mcale =\omega_k$ or $ \omega_{k+1},$ i.e $\la =  \mu $ or $  \mu -\epsilon_{m-k}$. In fact, $[T(\mu):\Delta(\mu)]=[T(\mu-\epsilon_{m-k}):\Delta(\mu)]=1$.

				{ Subcase (ii):} Let $-w_0\mu-\mcale=\omega_k$ for some $k \geq m$. Then $ \mu=-2\sum\limits_{i=1}^m\epsilon_{i} +\sum\limits_{i=1}^n\de_{i} -(k-m)\de_n $. Now by Theorem \ref{Final irr thm}, $[\mathcal{V}(-w_0\mu-\mcale):L(-w_0\lambda-\mcale)]\neq 0$ if and only if $-w_0\lambda-\mcale =\omega_k$ or $ \omega_{k+1},$ i.e $\la =  \mu $ or $  \mu -\de_{n}$. In fact, $[T(\mu):\Delta(\mu)]=[T(\mu-\de_{n}):\Delta(\mu)]=1$.

				{ Subcase (iii):} Let $-w_0\mu-\mcale=\theta_k$ for some $k \in \Z_+$ and $k \neq m$. Then $\mu=-2\sum\limits_{i=1}^m\epsilon_{i} +2\sum\limits_{i=1}^n\de_{i} +k\de_1 $. Now by Theorem \ref{Final irr thm}, $[\mathcal{V}(-w_0\mu-\mcale):L(-w_0\lambda-\mcale)]\neq 0$ if and only if $-w_0\lambda-\mcale =\theta_k$ or $ \theta_{k+1},$ i.e $\la =  \mu $ or $  \mu +\de_{1}$. In fact, $[T(\mu):\Delta(\mu)]=[T(\mu+\de_{1}):\Delta(\mu)]=1$.
				
				{ Subcase (iv):} Let $-w_0\mu-\mcale=\theta_m$. Then $\mu=-2\sum\limits_{i=1}^m\epsilon_{i} +2\sum\limits_{i=1}^n\de_{i} +m\de_1 $. Now by Theorem \ref{Final irr thm}, $[\mathcal{V}(-w_0\mu-\mcale):L(-w_0\lambda-\mcale)]\neq 0$ if and only if $-w_0\lambda-\mcale =\theta_m$ or $ \theta_{m+1},$ or $0$ i.e $\la =  \mu $ or $  \mu +\de_{1}$ or $- \mcale$. In fact, $[T(\mu):\Delta(\mu)]=[T(\mu+\de_{1}):\Delta(\mu)]=[T(-\mcale):\Delta(\mu)]=1$.\\
				{\bf Case 2:} $\mathcal{V}(-w_0\mu-\mcale)$ is simple.
				In this case, statements (5) follows from above computations and Theorem \ref{Final irr thm}.
			\end{proof}

			\subsection{Proof of Theorem \ref{thm: tilting character}}\label{sec: pf of tilt ch}	 Theorem \ref{thm: tilting character} directly follows from Proposition \ref{main thm1}.

				\section{Appendix A: Proof of Lemma \ref{lem: type I module map}}\label{sec: appendix A}
				
				In this appendix we prove that $\mfk d_k:\mcal V(\omega_k) \longrightarrow \mcal V(\omega_{k+1})$ is a $\mf g$-module map.
				Let $v=e_{p_1} \curlywedge \dots \curlywedge e_{p_k}$ be a basis vector of $\bigwedge^k{\bbf^{(m|n)}}$. We first have
				\begin{align*}				
					&x^\al y_\mu \pr_i.\mfk d_k(x^\be y_\eta\ot v)\cr
					= &x^\al y_\mu \pr_i.[\sum\limits_{k=1}^m\be_kx^{\be -\epsilon_k}y_{\eta} \otimes e_{p_1}\curlywedge\cdots\curlywedge e_{p_k}\curlywedge e_k\cr
					&\qquad 		+(-1)^{\wp(e_{p_1})+ \cdots + \wp(e_{p_k}) +|\eta|+1}\sum\limits_{k=1}^nx^{\alpha}D_k(y_{\eta}) \otimes e_{p_1}\curlywedge\cdots\curlywedge e_{p_k}\curlywedge e_{m+k}]\cr
					=& \sum\limits_{k=1}^m\be_k(\be_i-\de_{k,i})x^{\al+\be -\epsilon_k-\epsilon_i} y_\mu y_{\eta} \otimes e_{p_1}\curlywedge\cdots\curlywedge e_{p_k}\curlywedge e_k\cr
					&+ \sum\limits_{k=1}^m\sum\limits_{j=1}^m\be_k\al_jx^{\al+\be -\epsilon_k-\epsilon_j} y_\mu y_{\eta} \otimes x_j\pr_i.(e_{p_1}\curlywedge\cdots\curlywedge e_{p_k}\curlywedge e_k)     \cr
					&+(-1)^{|\mu|+|\eta|+1}  \sum\limits_{k=1}^m\sum\limits_{j=1}^n\be_kx^{\al+\be -\epsilon_k} D_j(y_\mu) y_{\eta} \otimes y_j\pr_i.(e_{p_1}\curlywedge\cdots\curlywedge e_{p_k}\curlywedge e_k )\cr
					& + (-1)^{\wp(e_{p_1})+ \cdots + \wp(e_{p_k}) +|\eta|+1}\{  \sum\limits_{k=1}^n\be_ix^{\al+\be -\epsilon_i} y_\mu D_k( y_{\eta}) \otimes e_{p_1}\curlywedge\cdots\curlywedge e_{p_k}\curlywedge e_{k+m}         \cr
					&	\qquad+ \sum\limits_{k=1}^n\sum\limits_{j=1}^m\al_jx^{\al+\be -\epsilon_j} y_\mu D_k (y_{\eta}) \otimes x_j\pr_i.(e_{p_1}\curlywedge\cdots\curlywedge e_{p_k}\curlywedge e_{k+m} )\cr
					&\qquad	+(-1)^{|\mu|+|\eta|}	\sum\limits_{k=1}^n\sum\limits_{j=1}^nx^{\al+\be } D_j(y_\mu) D_k (y_{\eta}) \otimes y_j\pr_i.(e_{p_1}\curlywedge\cdots\curlywedge e_{p_k}\curlywedge e_{k+m} ) \}
					\cr
					=&E_1+E_2+E_3+E_4+E_5+E_6,
				\end{align*}
				where we take $E_i$ for $1 \leq i \leq 6$ as the above six summations,  respectively. We continue to make calculations.
				\begin{align*}
					&\mfk d_k(x^\al y_\mu \pr_i.x^\be y_\eta\ot v)\cr
					=& \mfk d_k[\be_ix^{\al+\be -\epsilon_i} y_\mu y_{\eta} \otimes e_{p_1}\curlywedge\cdots\curlywedge e_{p_k} + \sum\limits_{j=1}^m\al_jx^{\al+\be -\epsilon_j} y_\mu y_{\eta} \otimes x_j\pr_i.(e_{p_1}\curlywedge\cdots\curlywedge e_{p_k}) \cr &\qquad +(-1)^{|\mu|+|\eta|+1} \sum\limits_{j=1}^nx^{\al+\be }D_j( y_\mu) y_{\eta} \otimes y_j\pr_i.(e_{p_1}\curlywedge\cdots\curlywedge e_{p_k}) ] \cr
					=&\be_i \sum\limits_{k=1}^m(\al_k+\be_k-\de_{k,i})x^{\al+\be -\epsilon_i-\epsilon_k} y_\mu y_{\eta} \otimes e_{p_1}\curlywedge\cdots\curlywedge e_{p_k}\curlywedge e_k\cr
					& +(-1)^{\wp(e_{p_1})+ \cdots + \wp(e_{p_k}) +|\mu|+|\eta|+1}\sum\limits_{k=1}^n\be_ix^{\alpha+\be-\epsilon_i}D_k(y_\mu y_{\eta}) \otimes e_{p_1}\curlywedge\cdots\curlywedge e_{p_k}\curlywedge e_{m+k}\cr
					& +  \sum\limits_{j=1}^m\sum\limits_{k=1}^m\al_j(\al_k+\be_k-\de_{j,k})x^{\al+\be -\epsilon_k-\epsilon_j} y_\mu y_{\eta} \otimes x_j\pr_i.(e_{p_1}\curlywedge\cdots\curlywedge e_{p_k})\curlywedge e_k \cr   & +(-1)^{\wp(e_{p_1})+ \cdots + \wp(e_{p_k}) +|\mu|+|\eta|+1} \sum\limits_{j=1}^m\sum\limits_{k=1}^n\al_jx^{\al+\be-\epsilon_j}D_k( y_\mu y_{\eta}) \otimes x_j\pr_i.(e_{p_1}\curlywedge\cdots\curlywedge e_{p_k})\curlywedge e_{k+m} \cr
					&			+(-1)^{|\mu|+|\eta|+1} \sum\limits_{j=1}^n\sum\limits_{k=1}^m(\al_k+\be_k)x^{\al+\be-\epsilon_k}D_k( y_\mu) y_{\eta} \otimes y_j\pr_i.(e_{p_1}\curlywedge\cdots\curlywedge e_{p_k})\curlywedge e_{k} \cr
					& +(-1)^{\wp(e_{p_1})+ \cdots + \wp(e_{p_k}) } \sum\limits_{j=1}^n\sum\limits_{k=1}^n\al_jx^{\al+\be}D_k(D_j( y_\mu) y_{\eta}) \otimes y_j\pr_i.(e_{p_1}\curlywedge\cdots\curlywedge e_{p_k})\curlywedge e_{k+m}\cr
					=&A_1+A_2+A_3+A_4+A_5+A_6,
				\end{align*}
				where we take $A_i$ for $1 \leq i \leq 6$ as the above six summations,  respectively. Now we have
				\begin{align*}
					E_2=& \sum\limits_{k=1}^m\sum\limits_{j=1}^m\be_k\al_jx^{\al+\be -\epsilon_k-\epsilon_j} y_\mu y_{\eta} \otimes x_j\pr_i.(e_{p_1}\curlywedge\cdots\curlywedge e_{p_k}\curlywedge e_k ) \cr
					=&\sum\limits_{k=1}^m\sum\limits_{j=1}^m\be_k\al_jx^{\al+\be -\epsilon_k-\epsilon_j} y_\mu y_{\eta} \otimes ((e_{p_1}\curlywedge\cdots\curlywedge e_{p_k})\de_{ik}\curlywedge e_j  + x_j\pr_i.(e_{p_1}\curlywedge\cdots\curlywedge e_{p_k})\curlywedge e_k) \cr
					=&\be_i\sum\limits_{j=1}^m\be_k\al_jx^{\al+\be -\epsilon_k-\epsilon_j} y_\mu y_{\eta} \otimes ((e_{p_1}\curlywedge\cdots\curlywedge e_{p_k})\curlywedge e_j \cr
					& + \sum\limits_{k=1}^m\sum\limits_{j=1}^m\be_k\al_jx^{\al+\be -\epsilon_k-\epsilon_j} y_\mu y_{\eta} \otimes x_j\pr_i.(e_{p_1}\curlywedge\cdots\curlywedge e_{p_k})\curlywedge e_k.
				\end{align*}
				Now one can easily observe that $E_1+E_2=A_1+A_3$ with the help of the fact $$\sum\limits_{k=1}^m\sum\limits_{j=1}^m\al_j(\al_k-\de_{k,j})x^{\al+\be -\epsilon_k-\epsilon_j} y_\mu y_{\eta} \otimes x_j\pr_i.(e_{p_1}\curlywedge\cdots\curlywedge e_{p_k})\curlywedge e_k=0.$$
				Furthermore, we have
				\begin{align*}
					&E_3+E_4\cr
					= &(-1)^{|\mu|+|\eta|+1}  \sum\limits_{k=1}^m\sum\limits_{j=1}^n\be_kx^{\al+\be -\epsilon_k} D_j(y_\mu) y_{\eta} \otimes y_j\pr_i.(e_{p_1}\curlywedge\cdots\curlywedge e_{p_k}\curlywedge e_k) +E_4\cr
					=&(-1)^{|\mu|+|\eta|+1}  \sum\limits_{k=1}^m\sum\limits_{j=1}^n\be_kx^{\al+\be -\epsilon_k} D_j(y_\mu) y_{\eta} \otimes (y_j\pr_i.(e_{p_1}\curlywedge\cdots\curlywedge e_{p_k})\curlywedge e_k \cr
					&\qquad+(-1)^{\wp(e_{p_1})+ \cdots + \wp(e_{p_k})} e_{p_1}\curlywedge\cdots\curlywedge e_{p_k}\curlywedge \de_{k,i}e_{j+m}  ) +E_4\cr
					=& (-1)^{|\mu|+|\eta|+1}  \sum\limits_{k=1}^m\sum\limits_{j=1}^n\be_kx^{\al+\be -\epsilon_k} D_j(y_\mu) y_{\eta} \otimes y_j\pr_i.(e_{p_1}\curlywedge\cdots\curlywedge e_{p_k})\curlywedge e_k \cr
					&\qquad+(-1)^{\wp(e_{p_1})+ \cdots + \wp(e_{p_k})+|\mu|+|\eta|+1} \sum\limits_{j=1}^n\be_ix^{\al+\be -\epsilon_i} D_j(y_\mu y_{\eta}) \otimes e_{p_1}\curlywedge\cdots\curlywedge e_{p_k}\curlywedge e_{j+m} +E_4\cr
					=& (-1)^{|\mu|+|\eta|+1}  \sum\limits_{k=1}^m\sum\limits_{j=1}^n\be_kx^{\al+\be -\epsilon_k} D_j(y_\mu) y_{\eta} \otimes y_j\pr_i.(e_{p_1}\curlywedge\cdots\curlywedge e_{p_k})\curlywedge e_k  +A_2.
				\end{align*}
				Now we claim that
				\begin{align}\label{eq: e5}
					A_4+(-1)^{|\mu|+|\eta|+1}  \sum\limits_{k=1}^m\sum\limits_{j=1}^n\al_kx^{\al+\be -\epsilon_k} D_j(y_\mu) y_{\eta} \otimes y_j\pr_i.(e_{p_1}\curlywedge\cdots\curlywedge e_{p_k})\curlywedge e_k =E_5.
				\end{align}
				
				First note that except the case when $i$ occur in the set $\{p_1, \dots, p_k\}$ only one time, the above claim is true. So we consider the case when exactly one of $p_1, \dots, p_k$ is $i$. Without loss of generality assume that $p_1=i$. Then we have,
				\begin{align*}
					&A_4\cr
					=&
					(-1)^{\wp(e_{p_2})+ \cdots + \wp(e_{p_k})+|\mu|+|\eta|+1} \sum\limits_{j=1}^m\sum\limits_{k=1}^n\al_jx^{\al+\be -\epsilon_j} D_k(y_\mu y_{\eta}) \otimes e_j\curlywedge e_{p_2}\curlywedge\cdots\curlywedge e_{p_k}\curlywedge e_{k+m}\cr
					=&(-1)^{k+\wp(e_{p_2})+ \cdots + \wp(e_{p_k})+|\mu|+|\eta|} \sum\limits_{j=1}^m\sum\limits_{k=1}^n\al_jx^{\al+\be -\epsilon_j} D_k(y_\mu y_{\eta}) \otimes e_{p_2}\curlywedge\cdots\curlywedge e_{p_k}\curlywedge e_j \curlywedge e_{k+m}
				\end{align*}
				and
				\begin{align*}
					&E_5\cr
					=&(-1)^{\wp(e_{p_2})+ \cdots + \wp(e_{p_k})+|\eta|+1} \sum\limits_{k=1}^n\sum\limits_{j=1}^m\al_jx^{\al+\be -\epsilon_j} y_\mu D_k(y_{\eta}) \otimes e_j\curlywedge e_{p_2}\curlywedge\cdots\curlywedge e_{p_k}\curlywedge e_{k+m}\cr
					=&(-1)^{k+\wp(e_{p_2})+ \cdots + \wp(e_{p_k})+|\eta|} \sum\limits_{k=1}^n\sum\limits_{j=1}^m\al_jx^{\al+\be -\epsilon_j} y_\mu D_k(y_{\eta}) \otimes  e_{p_2}\curlywedge\cdots\curlywedge e_{p_k}\curlywedge e_j\curlywedge e_{k+m}.
				\end{align*}
				Note that the second summation of the left hand in (\ref{eq: e5}) becomes
				\begin{align*}
					& (-1)^{|\mu|+|\eta|+1}  \sum\limits_{k=1}^m\sum\limits_{j=1}^n\al_kx^{\al+\be -\epsilon_k} D_j(y_\mu) y_{\eta} \otimes y_j\pr_i.(e_{p_1}\curlywedge\cdots\curlywedge e_{p_k})\curlywedge e_k \cr
					=&(-1)^{|\mu|+|\eta|+1}  \sum\limits_{j=1}^m\sum\limits_{k=1}^n\al_jx^{\al+\be -\epsilon_j} D_k(y_\mu) y_{\eta} \otimes e_{k+m}\curlywedge e_{p_2}\curlywedge\cdots\curlywedge e_{p_k}\curlywedge e_j \cr
					=&(-1)^{k+\wp(e_{p_2})+ \cdots + \wp(e_{p_k})+|\mu|+|\eta|+1}  \sum\limits_{j=1}^m\sum\limits_{k=1}^n\al_jx^{\al+\be -\epsilon_j} D_k(y_\mu) y_{\eta} \otimes e_{p_2}\curlywedge\cdots\curlywedge e_{p_k}\curlywedge e_j\curlywedge e_{k+m}
				\end{align*}
				Now it is easy to see by using the fact $D_k(y_\mu y_\eta)=D_k(y_\mu)y_\eta +(-1)^{|\mu|}y_\mu D_k(y_\eta)$ in the expression $A_4$ that the claim is true.

				Now using the fact $y_j\pr_i.e_{k+m}=0$ and derivation formula for $D_k$ in the expression of $A_6$ we have,
				$$A_6=E_6 +(-1)^{\wp(e_{p_1})+ \cdots + \wp(e_{p_k}) +1} \sum\limits_{j=1}^n\sum\limits_{k=1}^n\al_jx^{\al+\be}D_k(D_j( y_\mu)) y_{\eta} \otimes y_j\pr_i.(e_{p_1}\curlywedge\cdots\curlywedge e_{p_k})\curlywedge e_{k+m}.$$
				Now it is easy to observe that the second term in the above expression is zero. Therefore $A_6=E_6$. Now we have
				\begin{align*}
					&E_1+E_2+E_3+E_4+E_5+E_6\cr
					=&A_1+A_3+A_2+(-1)^{|\mu|+|\eta|+1}  \sum\limits_{k=1}^m\sum\limits_{j=1}^n\be_kx^{\al+\be -\epsilon_k} D_j(y_\mu) y_{\eta} \otimes y_j\pr_i.(e_{p_1}\curlywedge\cdots\curlywedge e_{p_k})\curlywedge e_k +A_4 \cr
					&\qquad +  (-1)^{|\mu|+|\eta|+1}  \sum\limits_{k=1}^m\sum\limits_{j=1}^n\al_kx^{\al+\be -\epsilon_k} D_j(y_\mu) y_{\eta} \otimes y_j\pr_i.(e_{p_1}\curlywedge\cdots\curlywedge e_{p_k})\curlywedge e_k + A_6\cr
					=&A_1+A_2+A_3+A_4+A_5+A_6.
				\end{align*}	
				Therefore we have proved that  $\mfk d_k(x^\al y_\mu \pr_i.x^\be y_\eta\ot v)=x^\al y_\mu \pr_i.\mfk d_k(x^\be y_\eta\ot v)$ for all $\al , \be \in \N^m, v \in \bigwedge^k \bbf(m|n), \mu , \eta \subseteq [1,n]$. {  We need to check the other one equality as below:
					$$d_k(x^\al y_\mu D_i.x^\be y_\eta\ot v)=x^\al y_\mu D_i.d_k(x^\be y_\eta\ot v)$$
					for all $\al , \be \in \N^m, v \in \bigwedge^k \C(m|n), \mu , \eta \subseteq [1,n]$ and $i \in [1,n]$, which we leave for the reader.		
					\begin{cor}\label{cor d^2=0}
						$\mfk d_{k+1} \circ \mfk d_{k} =0 $ for all $ k \geq 0$.
					\end{cor}		
					\begin{proof}
						From the proof of Lemma 4.4, we have $\mcal V(\omega_k)$ is generated by $x_1 \ot e_{m+1}^k$. For all $X \in U(\mf g)$ we have,
						\begin{align}
							\mfk d_{k+1} \circ \mfk  d_{k}(X \cd(x_1 \ot e_{m+1}^k)) &= 	\mfk d_{k+1} (X \cd (1 \ot e_{m+1}^k \cw e_1)) \cr
							&= X \cd 0 =0.
						\end{align}
						
					\end{proof}		
				}
				\section{Appendix B:  Proof of the relations (\ref{eq:i)Case1}), (\ref{key eqn ty weight 1}) and (\ref{eq:(iv)Case3})
					for Theorem \ref{Thm g submod}}\label{sec: app B}
				
				(1) For $1\le i,i^{'}\le m$ and homogeneous $f, f_\nu, g_\nu \in \mcal R$, we compute the following operator:
				\begin{align*}
					&\rho(ff_{\nu}\pr_{i})\rho(g_{\nu}\pr_{i'})\cr
					=&\left(ff_{\nu}\rho(\pr_{i})+\sum\limits^{m}_{j=1}\pr_{j}
					(ff_{\nu})\sigma(E_{ji})+(-1)^{\wp(ff_{\nu})+1}
					\sum\limits_{s=1}^{n}D_{s}(ff_{\nu})\sigma(E_{m+s,i})\right)\cdot\cr
					&\qquad \left(g_{\nu}\rho(\pr_{i^{'}})
					+\sum\limits^{m}_{j^{'}=1}\pr_{j^{'}}(g_{\nu})
					\sigma(E_{j^{'}i^{'}})+
					(-1)^{\wp(g_{\nu})+1}\sum\limits_{s^{'}=1}^{n}D_{s'}
					(g_{\nu})\sigma(E_{m+s^{'},i^{'}})\right)\cr
					=&ff_{\nu}\rho(\pr_{i})g_{\nu}\rho(\pr_{i^{'}})+
					\sum\limits^{m}_{j^{'}=1}(ff_{\nu}\pr_{j^{'}}
					(g_{\nu})\rho(\pr_{i})+ff_{\nu}\pr_{i}\pr_{j^{'}}
					(g_{\nu}))\sigma(E_{j^{'}i^{'}})\cr
					&\quad+(-1)^{\wp(g_{\nu})+1}\sum\limits_{s^{'}=1}^{n} (ff_{\nu}D_{s^{'}}(g_{\nu})\rho(\pr_{i})+
					ff_{\nu}\pr_{i}D_{s^{'}}(g_{\nu}))\sigma(E_{m+s^{'},i^{'}})\cr
					&\quad+\sum\limits_{j=1}^{m}\pr_{j}(ff_{\nu})\sigma(E_{ji})
					g_{\nu}\rho(\pr_{i^{'}})+
					\sum\limits_{j=1}^{m}\sum\limits_{j^{'}=1}^{m}
					(\pr_{j}(ff_{\nu}\pr_{j^{'}}(g_{\nu}))
					-ff_{\nu}\pr_{j}\pr_{j^{'}}(g_{\nu}))\sigma(E_{ji})
					\sigma(E_{j^{'}i^{'}})\cr
					&\quad
					+(-1)^{\wp(g_{\nu})+1}\sum\limits_{j=1}^{m}\sum\limits_{s^{'}=1}^{n}
					(\pr_{j}(ff_{\nu}D_{s^{'}}(g_{\nu}))-ff_{\nu}
					\pr_{j}D_{s^{'}}(g_{\nu}))\sigma(E_{ji})\sigma(E_{m+s^{'},i^{'}})\cr
					&\quad+(-1)^{\wp(ff_{\nu})+1}\sum\limits^{n}_{s=1}D_{s}
					(ff_{\nu})\sigma(E_{m+s,i})g_{\nu}\rho(\pr_{i^{'}})\cr
					&\quad+(-1)^{\wp(ff_{\nu})+\wp(g_{\nu})+1}
					\sum\limits_{s=1}^{n}\sum\limits_{j^{'}=1}^{m}
					(D_{s}(ff_{\nu}\pr_{j^{'}}(g_{\nu}))-(-1)^{\wp(ff_{\nu})}
					ff_{\nu}D_{s}\pr_{j^{'}}(g_{\nu}))\sigma(E_{m+s,i})
					\sigma(E_{j^{'}i^{'}})\cr
					&\quad +(-1)^{\wp(ff_{\nu})+1}
					\sum\limits_{s=1}^{n}\sum\limits_{s^{'}=1}^{n}
					(D_{s}(ff_{\nu}D_{s^{'}}(g_{\nu}))-(-1)^{\wp(ff_{\nu})}
					ff_{\nu}D_{s}D_{s^{'}}(g_{\nu}))\sigma(E_{m+s,i})
					\sigma(E_{m+s^{'},i^{'}}).
				\end{align*}

				\begin{lem}\label{Lem rel 1}
					Consider the monomials $f_{1}=1,f_{2}=-x_{k^{'}},f_{3}=-x_{k},f_{4}=x_{k}x_{k^{'}},g_{1}=x_{k}x_{k^{'}},g_{2}=x_{k},g_{3}=x_{k^{'}},g_{4}=1,$ for some $ k,k' \in [1,m]$ with $k \neq k'$. Then for any homogeneous $f \in \mcal R$ we have the following relations:
					\begin{itemize}
						\item[(1)] $\dis{\sum_{\nu=1}^{4}} f_\nu g_\nu=0  $.
						\item[(2)]   $\dis{\sum_{\nu=1}^{4} \sum_{j'=1}^{m}} f_\nu \pr_{j'}(g_\nu)\sigma(E_{j'i'})=0$.
						\item[(3)] $\sum\limits_{\nu=1}^{4}\sum\limits_{j=1}^{m}\pr_{j}
						(ff_{\nu})g_{\nu}\sigma(E_{ji}) =0$.
						\item[(4)] $\sum\limits_{\nu=1}^{4}\sum\limits^{n}_{s=1}D_{s}(ff_{\nu})g_{\nu} \sigma(E_{m+s,i})=0$ for $i,i' \in [1,m]$.	
					\end{itemize}
				\end{lem}

				Now note that with the choice of $f_\nu$ and $g_\nu$ as Lemma \ref{Lem rel 1} we have
				{\begin{align*}	
						&\dis{\sum_{\nu=1}^{4}} \rho(ff_{\nu}\pr_{i})\rho(g_{\nu}\pr_{i^{'}})\cdot v_{\la}   \cr
						\overset{{\tiny\text{By Lemma}} \ref{Lem rel 1}}{=}&\dis{\sum_{\nu=1}^{4}}(\sum\limits_{j^{'}=1}^{m}
						ff_{\nu}\pr_{i}\pr_{j^{'}}(g_{\nu})\sigma(E_{j^{'}i^{'}}) -\sum\limits_{j=1}^{m}\sum\limits_{j^{'}=1}^{m}
						ff_{\nu}\pr_{j}\pr_{j^{'}}(g_{\nu})\sigma(E_{ji})\sigma(E_{j^{'}i^{'}}))\cdot v_{\la}
					\end{align*}
					\begin{align}\label{equn 1st rel}
						&=f(\delta_{ik}\sigma(E_{k^{'}i^{'}})+\delta_{ik^{'}}\sigma(E_{ki^{'}}) - \sigma(E_{ki})\sigma(E_{k^{'}i^{'}})-\sigma(E_{k^{'}i})\sigma(E_{ki^{'}}))\cdot v_{\la}	\end{align}}
				
				Hence we have one relation of (\ref{eq:i)Case1}), 	$$\delta_{ik}\sigma(E_{k^{'}i^{'}})+\delta_{ik^{'}}\sigma(E_{ki^{'}}) - \sigma(E_{ki})\sigma(E_{k^{'}i^{'}})-\sigma(E_{k^{'}i})\sigma(E_{ki^{'}})=0.$$
				
				\begin{lem}\label{Lem rel 2}
					Consider the monomials $f_{1}=\frac{1}{2},f_{2}=-x_{k},f_{3}=\frac{1}{2}x_{k}^{2},g_{1}=x_{k}^{2},g_{2}=x_{k},g_{3}=1$ for some $ k \in [1,m]$. Then for any homogeneous $f \in \mcal R$ we have the following relations  {for $i,i' \in [1,m]$} :
					\begin{itemize}
						\item[(1)] $\dis{\sum_{\nu=1}^{3}} f_\nu g_\nu=0$.
						\item[(2)]   $\dis{\sum_{\nu=1}^{3} \sum_{j'=1}^{m}} f_\nu \pr_{j'}(g_\nu) \sigma(E_{j'i'})=0$.
						\item[(3)] $\sum\limits_{\nu=1}^{3}\sum\limits_{j=1}^{m}
						\pr_{j}(ff_{\nu})g_{\nu}\sigma(E_{ji}) =0$.
						\item[(4)] $\sum\limits_{\nu=1}^{3}\sum\limits^{n}_{s=1}D_{s}(ff_{\nu})g_{\nu} \sigma(E_{m+s,i})=0$.
					\end{itemize}	
					
				\end{lem}	
				Now note that with the choice of $f_\nu$ and $g_\nu$ as Lemma \ref{Lem rel 2} we have 	
				{	\begin{align*}
						&\dis{\sum_{\nu=1}^{3}} \rho(ff_{\nu}\pr_{i})\rho(g_{\nu}\pr_{i^{'}})\cdot v_{\la} \cr
						\overset{\text{By Lemma}\ref{Lem rel 2}}{=}&\dis{\sum_{\nu=1}^{3}}(\sum\limits_{j^{'}=1}^{m}ff_{\nu}\pr_{i}\pr_{j^{'}}(g_{\nu})\sigma(E_{j^{'}i^{'}}) -\sum\limits_{j=1}^{m}\sum\limits_{j^{'}=1}^{m}ff_{\nu}\pr_{j}\pr_{j^{'}}(g_{\nu})\sigma(E_{ji})\sigma(E_{j^{'}i^{'}}))\cdot v_{\la} \cr
						=& f(\delta_{ik^{'}}\sigma(E_{ki^{'}}) - \sigma(E_{k^{}i})\sigma(E_{ki^{'}}))\cdot v_{\la}.
				\end{align*}}
				{The above equation gives us another relation of (\ref{eq:i)Case1}).}

				(2) For $1\le i\le m$, $1\le i^{'}\le n$ we consider the operator\\
				
				$	\rho(ff_{\nu}\pr_{i})\rho(g_{\nu}D_{i^{'}})$\\
				$=\left(ff_{\nu}\rho(\pr_{i})+\sum\limits^{m}_{j=1}
				\pr_{j}(ff_{\nu})\sigma(E_{ji})+(-1)^{\wp(ff_{\nu})+1}
				\sum\limits_{s=1}^{n}D_{s}(ff_{\nu})\sigma(E_{m+s,i})\right)\cdot$\\
				$\left(g_{\nu}\rho(D_{i^{'}})
				+\sum\limits^{m}_{j^{'}=1}\pr_{j^{'}}(g_{\nu})
				\sigma(E_{j^{'},m+i^{'}})+(-1)^{\wp(g_{\nu})+1}
				\sum\limits_{s^{'}=1}^{n}D_{s^{'}}(g_{\nu})\sigma(E_{m+s^{'},m+i'})\right)$\\
				$=ff_{\nu}\rho(\pr_{i})g_{\nu}\rho(D_{i^{'}})+
				\sum\limits^{m}_{j^{'}=1}(ff_{\nu}\pr_{j^{'}}
				(g_{\nu})\rho(\pr_{i})+ff_{\nu}\pr_{i}\pr_{j^{'}}(g_{\nu}))
				\sigma(E_{j^{'},m+i'})$\\
				$+(-1)^{\wp(g_{\nu})+1}\sum\limits_{s^{'}=1}^{n} (ff_{\nu}D_{s^{'}}(g_{\nu})\rho(\pr_{i})+
				ff_{\nu}\pr_{i}D_{s^{'}}(g_{\nu}))\sigma(E_{m+s^{'},m+i'})$\\
				$+\sum\limits_{j=1}^{m}\pr_{j}(ff_{\nu})
				\sigma(E_{ji})g_{\nu}\rho(D_{i^{'}})+
				\sum\limits_{j=1}^{m}\sum\limits_{j^{'}=1}^{m}(\pr_{j}
				(ff_{\nu}\pr_{j^{'}}(g_{\nu}))-ff_{\nu}\pr_{j}
				\pr_{j^{'}}(g_{\nu}))\sigma(E_{ji})\sigma(E_{j^{'},m+i'})$\\
				$+(-1)^{\wp(g_{\nu})+1}\sum\limits_{j=1}^{m}
				\sum\limits_{s^{'}=1}^{n}(\pr_{j}(ff_{\nu}D_{s^{'}}
				(g_{\nu}))-ff_{\nu}\pr_{j}D_{s^{'}}(g_{\nu}))\sigma(E_{ji})
				\sigma(E_{m+s^{'},m+i'})$\\
				$+(-1)^{\wp(ff_{\nu})+1}\sum\limits^{n}_{s=1}D_{s}(ff_{\nu})
				\sigma(E_{m+s,i})g_{\nu}\rho(D_{i^{'}})$\\
				$+(-1)^{\wp(ff_{\nu})+\wp(g_{\nu})+1}
				\sum\limits_{s=1}^{n}\sum\limits_{j^{'}=1}^{m}
				(D_{s}(ff_{\nu}\pr_{j^{'}}(g_{\nu}))-
				(-1)^{\wp(ff_{\nu})}ff_{\nu}D_{s}\pr_{j^{'}}(g_{\nu}))\sigma(E_{m+s,i})
				\sigma(E_{j^{'},m+i'})$\\
				$+(-1)^{\wp(ff_{\nu})+1}\sum\limits_{s=1}^{n}
				\sum\limits_{s^{'}=1}^{n}(D_{s}(ff_{\nu}D_{s^{'}}(g_{\nu}))-
				(-1)^{\wp(ff_{\nu})}ff_{\nu}D_{s}D_{s^{'}}(g_{\nu}))
				\sigma(E_{m+s,i})\sigma(E_{m+s^{'},m+i'}).$
				\begin{lem}\label{Lem rel 3}
					Consider the monomials 	$f_{1}=1,f_{2}=-y_{k^{'}},f_{3}=-x_{k},f_{4}=x_{k}y_{k^{'}},g_{1}=x_{k}y_{k^{'}},g_{2}=x_{k},g_{3}=y_{k^{'}},g_{4}=1$, for some $k\in [1,m]$ and $k^{'}\in [1,n]$. Then we have the following relations:
					\begin{itemize}
						\item[(1)] $\sum\limits_{\nu=1}^{4}f_{\nu}g_{\nu}=0$.
						\item[(2)] $\sum\limits_{\nu=1}^{4}\sum\limits_{s^{'}=1}^{n}\pr_{j}
						((-1)^{\wp(g_{\nu})+1}f_{\nu}D_{s^{'}}(g_{\nu}))
						\sigma(E_{m+s^{'},m+i^{'}})=0$.
						\item[(3)]$\sum\limits_{\nu=1}^{4}\sum\limits_{j^{'}=1}^{m}f_{\nu}\pr_{j^{'}}(g_{\nu}) \sigma(E_{j',m+i'})=0.$
						\item[(4)] $\sum\limits_{\nu=1}^{4}\sum\limits_{s'=1}^{n}(-1)^{\wp(f_{\nu})+1}
						f_{\nu}D_{s'}(g_{\nu})\sigma(E_{m+s',m+i'})=0$, for $ i' \in [1, n]$.
					\end{itemize}	
				\end{lem}
				
				Now with the choice of $f_\nu$ and $g_\nu$ as Lemma \ref{Lem rel 3} we have 	
				\begin{align*}		
					&\dis{\sum_{\nu=1}^{4}} \rho(ff_{\nu}D_{i})\rho(g_{\nu}D_{i^{'}}) \cd v_\la  \cr
					=&\dis{\sum_{\nu=1}^{4}}(-1)^{\wp(g_\nu)+1}
					(\sum\limits_{s^{'}=1}^{n}ff_{\nu}\pr_{i}D_{s^{'}}
					(g_{\nu})\sigma(E_{{m+s'},m+i^{'}}) -\sum\limits_{j=1}^{m}\sum\limits_{j^{'}=1}^{m}ff_{\nu}\pr_{j}
					D_{s^{'}}(g_{\nu})\sigma(E_{ji})\sigma(E_{m+s^{'},m+i^{'}}))\cd v_\la \end{align*}
{\color{red}and}
				\begin{align*}
					&-\dis{\sum_{\nu=1}^{4}}(-1)^{\wp(g_\nu)+1}	\sum\limits_{s=1}^{n}\sum\limits_{j^{'}=1}^{m}
					ff_{\nu}D_{s}\pr_{j^{'}}(g_{\nu}))\sigma(E_{m+s,i})
					\sigma(E_{j^{'},m+i'}).v_\la\cr
					=&f\left(\delta_{ik}\sigma(E_{m+k^{'},m+i^{'}})-\sigma(E_{ki})
					\sigma(E_{m+k^{'},i^{'}})-\sigma(E_{m+k^{'},i})\sigma(E_{k,m+i^{'}})\right) \cd v_\la.
				\end{align*}
				This gives us the relation (\ref{key eqn ty weight 1}).

				(3)For $1\le i,i^{'}\le n$, we consider the operator:
				\begin{align*}
					&\rho(ff_{\nu}D_{i})\rho(g_{\nu}D_{i^{'}})\cr
					=&\left(ff_{\nu}\rho(D_{i})+\sum\limits^{m}_{j=1}
					\pr_{j}(ff_{\nu})\sigma(E_{j,m+i})+(-1)^{\wp(ff_{\nu})+1}
					\sum\limits_{s=1}^{n}D_{s}(ff_{\nu})\sigma(E_{m+s,m+i})\right)\cdot\cr
					&\qquad\qquad\left(g_{\nu}\rho(D_{i^{'}})
					+\sum\limits^{m}_{j^{'}=1}\pr_{j^{'}}(g_{\nu})
					\sigma(E_{j^{'},m+i'})+(-1)^{\wp(g_{\nu})+1}
					\sum\limits_{s^{'}=1}^{n}D_{s^{'}}(g_{\nu})
					\sigma(E_{m+s^{'},m+i'})\right)\cr
					=&ff_{\nu}\rho(D_{i})g_{\nu}\rho(D_{i^{'}})+
					\sum\limits^{m}_{j^{'}=1}((-1)^{\wp(g_{\nu})}
					ff_{\nu}\pr_{j^{'}}(g_{\nu})\rho(D_{i})+
					ff_{\nu}D_{i}\pr_{j^{'}}(g_{\nu}))\sigma(E_{j^{'},m+i'})\cr
					&\quad +(-1)^{\wp(g_{\nu})+1}\sum\limits_{s^{'}=1}^{n} ((-1)^{\wp(g_{\nu})-1}ff_{\nu}D_{s^{'}}(g_{\nu})
					\rho(D_{i})+ff_{\nu}D_{i}D_{s^{'}}(g_{\nu}))\sigma(E_{m+s^{'},m+i'})\cr
					&\quad+\sum\limits_{j=1}^{m}\pr_{j}(ff_{\nu})\sigma(E_{j,m+i})
					g_{\nu}\rho(D_{i^{'}})\cr
					&\quad+(-1)^{\wp(g_{\nu})}\sum\limits_{j=1}^{m}
					\sum\limits_{j^{'}=1}^{m}
					(\pr_{j}(ff_{\nu}\pr_{j^{'}}(g_{\nu}))-ff_{\nu}\pr_{j}
					\pr_{j^{'}}(g_{\nu}))\sigma(E_{j,m+i})\sigma(E_{j^{'}, m+i'})\cr
					&\quad+\sum\limits_{j=1}^{m}\sum\limits_{s^{'}=1}^{n}(\pr_{j}
					(ff_{\nu}D_{s^{'}}(g_{\nu}))-ff_{\nu}\pr_{j}D_{s^{'}}
					(g_{\nu}))\sigma(E_{j,m+i})\sigma(E_{m+s^{'},m+i'})\cr
					&\quad+(-1)^{\wp(ff_{\nu})+1}\sum\limits^{n}_{s=1}D_{s}
					(ff_{\nu})\sigma(E_{m+s,m+i})g_{\nu}\rho(D_{i^{'}})\cr
					&\quad	+(-1)^{\wp(ff_{\nu})+1}\sum\limits_{s=1}^{n}\sum\limits_{j^{'}=1}^{m}
					(D_{s}(ff_{\nu}\pr_{j^{'}}(g_{\nu}))-(-1)^{\wp(ff_{\nu})}
					ff_{\nu}D_{s}\pr_{j^{'}}(g_{\nu}))\sigma(E_{m+s,m+i})
					\sigma(E_{j^{'},m+i'})\cr
					&\quad+(-1)^{\wp(ff_{\nu})+\wp(g_{\nu})}
					\sum\limits_{s=1}^{n}\sum\limits_{s^{'}=1}^{n}
					(D_{s}(ff_{\nu}D_{s^{'}}(g_{\nu}))-(-1)^{\wp(ff_{\nu})}
					ff_{\nu}D_{s}D_{s^{'}}(g_{\nu}))\sigma(E_{m+s,m+i})
					\sigma(E_{m+s^{'},m+i'}).
				\end{align*}
				\begin{lem}\label{Lem rel 4}
					Consider the monomials $f_{1}=1,f_{2}=y_{k^{'}},f_{3}=-y_{k},f_{4}=
					-y_{k^{'}}y_{k},g_{1}=y_{k^{'}}y_{k},g_{2}=y_{k},g_{3}=y_{k^{'}},g_{4}=3$   for some $ k,k' \in [1,n]$ with $k \neq k'$. Then we have the following relations:
					\begin{itemize}
						\item[(1)] $\dis{\sum_{\nu=1}^{4}} f_\nu g_\nu=0  $.
						\item[(2)]   $\dis{\sum_{\nu=1}^{4} \sum_{j'=1}^{m}} f_\nu D_{s'}(g_\nu)\sigma(E_{m+s',m+i'})=0,  $ for $i'\in [1,n]$.
					\end{itemize}			
				\end{lem}
				Now note that with the choice of $f_\nu$ and $g_\nu$ as Lemma \ref{Lem rel 4} we have
				\begin{align*}	
					&\dis{\sum_{\nu=1}^{4}} \rho(ff_{\nu}D_{i})\rho(g_{\nu}D_{i^{'}}) \cd v_\la  \cr =&\dis{\sum_{\nu=1}^{4}}(-1)^{\wp(g_\nu)+1}(\sum\limits_{s^{'}=1}^{n}
					ff_{\nu}D_{i}D_{s^{'}}(g_{\nu})\sigma(E_{{m+s'},m+i^{'}})\cr
					&\qquad +\sum\limits_{j=1}^{m}\sum\limits_{j^{'}=1}^{m}ff_{\nu}D_{s}D_{s^{'}}
					(g_{\nu})\sigma(E_{m+s,m+i})\sigma(E_{m+s^{'},m+i^{'}}))\cd v_\la\cr
					=& -f(\delta_{ik}\sigma(E_{m+k^{'},m+i^{'}})-\delta_{ik^{'}}
					\sigma(E_{m+k,m+i^{'}})+\cr
					&\qquad\sigma(E_{m+k,m+i})\sigma(E_{m+{k^{'},m+i^{'}}})-
					\sigma(E_{m+k^{'},m+i})\sigma(E_{m+k,m+i^{'}})) \cd v_{\la}.
				\end{align*}
				This gives us relation (\ref{eq:(iv)Case3}).

				\section{Appendix C: Proof of Theorem \ref{semi-inf}}\label{sec: app C}	
				
				It is easy to see that $\mcal E$ is a homomorphism of Lie superalgebra. Hence it is sufficient to check the property (SI-2). To check (SI-2),there are several possibilities for $X \in \mathfrak{g}_1$ and $Y\in \mathfrak{g}_{-1}$. We consider all these possibilities case by case. One should observe the cases (1,2,11,12), which ensures that if $\cale$ is a semi-infinite character then $\cale$ satisfies the properties of the theorem. \\
				
				(1) Let $X=x_iy_jD_k$ and $Y=\pr_s$. Then $[X,Y]=-\de_{si}y_jD_k$ and
				\begin{equation*}
					[X,[Y,Z]]=
					\begin{cases}
						-\de_{sr}\de_{ti}y_jD_k	, \; Z=x_r\pr_t \\
						0, \; Z=y_rD_t\\
						-\delta_{tj}\de_{sr}x_iD_k, \; Z=x_rD_t\\
						0, \; Z=y_r\pr_t.
					\end{cases}
				\end{equation*}
				(1.i) Note that when $s \neq i$, then $\str(\ad X\circ \ad Y|_{\ggg_0})=0=\cale([X,Y])$. Since in this case only diagonal elements comes from $Z=x_rD_t$, but $s \neq i$ implies that there is no non-zero diagonal element in $(\ad X\circ \ad Y|_{\ggg_0})$. \\
				(1.ii) On the other hand when $s \neq i$, we have two cases. Observe that for $j=k$ only non-zero diagonal term comes from the image of $x_iD_j$ and for $j \neq k$ there is no non-zero diagonal term. Therefore we have $\str(\ad X\circ \ad Y|_{\ggg_0})=0=\cale(-y_jD_k)$ for all $j \neq k$ and for $j=k$,
				$\str(\ad X\circ \ad Y|_{\ggg_0})=1=\cale(-y_jD_j)$.\\
				
				(2)  Let $X=x_iy_jD_k$ and $Y=D_s$. Then $[X,Y]=-\de_{sj}x_iD_k$ and\\
				\begin{equation*}
					[X,[Y,Z]]=
					\begin{cases}
						0	, \; Z=x_r\pr_t \\
						-\de_{sr}\de_{tj}x_iD_k, \; Z=y_rD_t\\
						0, \; Z=x_rD_t\\
						-\de_{ti}\de_{sr}y_jD_k, \; Z=y_r\pr_t.
					\end{cases}
				\end{equation*}
				Hence in this case $\str(\ad X\circ \ad Y|_{\ggg_0})=0=\cale([X,Y]),$ because it is clear that there can not be any non-zero diagonal term.\\
				
				(3)  Let $X=x_ix_jD_k$ and $Y=\pr_s$. Then $[X,Y]=-\pr_s(x_ix_jD_k)$ and\\
				\begin{equation*}
					[X,[Y,Z]]=
					\begin{cases}
						-\de_{sr}\pr_t(x_ix_j)D_k	, \; Z=x_r\pr_t \\
						0 \; Z=y_rD_t\\
						0, \; Z=x_rD_t\\
						0, \; Z=y_r\pr_t.
					\end{cases}
				\end{equation*}
				
				Now it is clear that $\str(\ad X\circ \ad Y|_{\ggg_0})=0=\cale([X,Y]).$\\
				(4)  Let $X=x_ix_jD_k$ and $Y=D_s$. Then $[X,Y]=0$ and
				\begin{equation*}
					[X,[Y,Z]]=
					\begin{cases}
						0, \; Z=x_r\pr_t \\
						0 \; Z=y_rD_t\\
						0, \; Z=x_rD_t\\
						-\de_{sr}\pr_t(x_ix_j)D_k, \; Z=y_r\pr_t.
					\end{cases}
				\end{equation*}
				Therefore $\str(\ad X\circ \ad Y|_{\ggg_0})=0=\cale([X,Y]).$\\
				(5)  Let $X=y_iy_jD_k$ and $Y=\pr_s$.  Then $[X,Y]=0$ and
				\begin{equation*}
					[X,[Y,Z]]=
					\begin{cases}
						0, \; Z=x_r\pr_t \\
						0 \; Z=y_rD_t\\
						0, \; Z=x_rD_t\\
						0, \; Z=y_r\pr_t.
					\end{cases}
				\end{equation*}
				Therefore $\str(\ad X\circ \ad Y|_{\ggg_0})=0=\cale([X,Y]).$\\
				
				(6)  Let $X=y_iy_jD_k$ and $Y=D_s$.  Then $[X,Y]=-D_s(y_iy_j)D_k$ and
				\begin{equation*}
					[X,[Y,Z]]=
					\begin{cases}
						0, \; Z=x_r\pr_t \\
						[y_iy_jD_k,[D_s,Z]], \; Z=y_rD_t\\
						0, \; Z=x_rD_t\\
						0, \; Z=y_r\pr_t.
					\end{cases}
				\end{equation*}\\
				
				(7) Let $X=y_iy_j\pr_k$ and $Y=\pr_s$. Then $[X,Y]=0$ and
				\begin{equation*}
					[X,[Y,Z]]=
					\begin{cases}
						0	, \; Z=x_r\pr_t \\
						0, \; Z=y_rD_t\\
						-\delta_{sr}D_t(y_iy_j)\pr_k, \; Z=x_rD_t\\
						0, \; Z=y_r\pr_t
					\end{cases}
				\end{equation*}
				Therefore $\str(\ad X\circ \ad Y|_{\ggg_0})=0=\cale([X,Y]).$\\
				
				(8) Let $X=y_iy_j\pr_k$ and $Y=D_s$. Then $[X,Y]=[y_iy_j\pr_k,D_s]$ and  \\
				\begin{equation*}
					[X,[Y,Z]]=
					\begin{cases}
						0, \; Z=x_r\pr_t \\
						-\delta_{sr}D_t(y_iy_j)\pr_k, \; Z=y_rD_t\\
						0, \; Z=x_rD_t\\
						0, \; Z=y_r\pr_t.
					\end{cases}
				\end{equation*}
				
				Therefore in any case (2), $\str(\ad X\circ \ad Y|_{\ggg_0})=0=\cale([X,Y]).$\\

				Hence in this case $\str(\ad X\circ \ad Y|_{\ggg_0})=\str(\ad X\circ \ad Y|_{Z_0})=\cale([X,Y])$, last equality follows from appendix A of \cite{DSY24}, where $Z_0=\spann_\bbf \{y_rD_t: 1\leq r,t \leq n\}$. \\
				
				(9) Let $X=x_ix_j\pr_k$ and $Y=\pr_s$.  Then $[X,Y]=[x_ix_j\pr_k,\pr_s]$ and
				\begin{equation*}
					[X,[Y,Z]]=
					\begin{cases}
						[x_ix_j\pr_k,[\pr_s,Z]], \; Z=x_r\pr_t \\
						0, \; Z=y_rD_t\\
						0, \; Z=x_rD_t\\
						0, \; Z=y_r\pr_t.
					\end{cases}
				\end{equation*}
				Note that $\str(\ad X\circ \ad Y|_{\ggg_0})=\textsf{tr}(\ad X\circ \ad Y|_{Z_0})=\cale([X,Y])$, last equality follows from Lemma 2.1 of \cite{DSY20}, where $Z_0=\spann_\bbf \{x_r\pr_t: 1\leq r,t \leq m\}$. \\
				
				(10)  Let $X=x_ix_j\pr_k$ and $Y=D_s$. Then $[X,Y]=0$ and
				\begin{equation*}
					[X,[Y,Z]]=
					\begin{cases}
						0, \; Z=x_r\pr_t \\
						0, \; Z=y_rD_t\\
						0, \; Z=x_rD_t\\
						-\de_{sr}\pr_t(x_ix_j))\pr_k, \; Z=y_r\pr_t.
					\end{cases}
				\end{equation*}
				Hence $ \str(\ad X\circ \ad Y|_{\ggg_0})=\cale([X,Y])=0.$\\
				
				(11) Let $X=x_iy_j\pr_k$ and $Y=\pr_s$. Then $[X,Y]=-\de_{si}y_j\pr_k$ and
				\begin{equation*}
					[X,[Y,Z]]=
					\begin{cases}
						-\de_{sr}\de_{ti}y_j\pr_k, \; Z=x_r\pr_t \\
						0, \; Z=y_rD_t\\
						\de_{sr}\de_{tj}x_i\pr_k \; Z=x_rD_t\\
						0, \; Z=y_r\pr_t.
					\end{cases}
				\end{equation*}
				Hence $ \str(\ad X\circ \ad Y|_{\ggg_0})=\cale([X,Y])=0.$\\
				
				(12) Let $X=x_iy_j\pr_k$ and $Y=D_s$. Then $[X,Y]=\de_{sj}x_i\pr_k$ and
				
				\begin{equation*}
					[X,[Y,Z]]=
					\begin{cases}
						0, \; Z=x_r\pr_t \\
						-\de_{sr}\de_{tj}x_i\pr_k, \; Z=y_rD_t\\
						0
						, \; Z=x_rD_t\\
						
						\de_{sr}\de_{ti}y_j\pr_k, Z=y_r\pr_t.
					\end{cases}
				\end{equation*}
				Therefore we have $ \str(\ad X\circ \ad Y|_{\ggg_0})=0$ for $s \neq j$ or $i \neq k$ and for $s=j$, $i=k$, we have $ \str(\ad X\circ \ad Y|_{\ggg_0})=1.$ This completes the proof.


\begin{thebibliography}{99}
					
					\bibitem{BWW} H. Bao, W. Wang and H. Watanabe, {\em Canonical bases for tensor products and super Kazhdan-Lusztig theory},
					J. Pure Appl. Algebra   224 (2020), no. 8, 106347, 9 pp.
					
					\bibitem{BBG} I.N. Bernstein, I.M. Gelfand, and S.I. Gelfand, {\em  On a category of $\ggg$-modules}, Funktsional. Anal. i Prilozhen. 10 (1976), no. 2, 1-8; English transl., Funct. Anal. Appl. 10 (1976), 87-92.
					
					\bibitem{BL83} I. N. Bernstein and D. A. Leites, {\itshape Irreducible representations of finite-dimensional Lie superalgebras of type
						$W$}, Selecta Math. Soviet. 3 (1983/84), no. 1, 63-68.
					\bibitem{B03} J. Brundan,{\em Kazhdan-Lusztig polynomials and character formulae for the Lie superalgebra $\gl(m|n)$}
					J. Amer. Math. Soc.   16 (2003), no. 1, 185-231.	
					
					\bibitem{BLW} J. Brundan, I. Losev, Ivan and B. Webster, {\em Tensor product categorifications and the super Kazhdan-Lusztig conjecture},
					Int. Math. Res. Not. IMRN 2017, no. 20, 6329-6410.
					
					\bibitem{CCC} C.-W. Chen, S.-J. Cheng, and K. Coulembier, {\em Tilting modules for classical Lie superalgebras}, J. Lond. Math. Soc. (2)   103 (2021), no. 3, 870-900.
					
					
						\bibitem{ChKac98} S.-J. Cheng and V. G Kac, {\em Generalized Spencer cohomology and filtered deformations of Z-graded Lie superalgebras},
						Adv. Theor. Math. Phys.   2 (1998), no. 5, 1141-1182.
						
						
						\bibitem{ChKac99} S.-J. Cheng and V. G Kac, {\em Structure of some Z-graded Lie superalgebras of vector fields},  Transform. Groups   4 (1999), no. 2-3, 219-272.
					
					\bibitem{CKW17} S.-J. Cheng, J.-H. Kwon and W. Wang, {\em
						Character formulae for queer Lie superalgebras and canonical bases of types $A/C$}, Comm. Math. Phys.   352 (2017), no. 3, 1091-1119.
					
					
					\bibitem{CLW15} S.-J. Cheng, N. Lam and W. Wang, {\em The Brundan-Kazhdan-Lusztig conjecture for general linear Lie superalgebras},
					Duke Math. J. 164 (2015), no. 4, 617-695.
					
					\bibitem{CW12}  S.-J. Cheng and W. Wang, {\it Dualities and representations of Lie \textit{superalgebras}} Grad. Stud. Math., 144, Amer.
					Math. Soc., Providence, RI, 2012.
					\bibitem{CW19}  S.-J. Cheng and W. Wang, {\em Character formulae in category O for exceptional Lie superalgebras $D(2|1;\xi)$}, Transform. Groups   24 (2019), no. 3, 781-821.
					
					\bibitem{CW22} S.-J. Cheng and W. Wang, {\em Character formulas in category O for exceptional Lie superalgebra G(3)}, Kyoto J. Math.   62 (2022), no. 4, 719-751.
					
					\bibitem{DSY20} F.-F. Duan, B. Shu and Y.-F. Yao, {\em The category O for Lie algebras of vector fields (I): Tilting modules and character formulas}, Publ. Res. Inst. Math. Sci. 56 (2020), 743-760.
					
					
					\bibitem{DSY24} F.-F. Duan, B. Shu and Y.-F. Yao, {\em Parabolic BGG categories and their block decomposition for Lie superalgebras of Cartan type}, . J. Math. Soc. Japan (2024).
					
					\bibitem{DSYC} F.-F. Duan, B. Shu, Y.-F. Yao and P. Chakraborty, {\em  The category O for Lie algebras of vector fields (II): Lie-Cartan modules and cohomology}, arXiv:2401.11071[math.RT].
					
					
					\bibitem{Fu86} D. B. Fuks, {\em Cohomology of infinite-dimensional Lie algebras}, Translated from the Russian by A. B. Sosinski\u{i},
					Contemp. Soviet Math. Consultants Bureau, New York, 1986.
					
					
					\bibitem{G70}  V. Guillemin, {\em Infinite dimensional primitive Lie algebras}, J. Differential Geometry 4 (1970), 257-282.
					
					\bibitem{GQS66}
					V. W. Guillemin, D. Quillen and S. Sternberg, {\em The classification of the complex primitive infinite pseudogroups}, Proc. Nat. Acad. Sci. U.S.A. 55 (1966), 687-690.
					
					
					
					
					
					\bibitem{Hum1} J. E. Humphreys, {\itshape Introduction to Lie algebras and representation thoery}, Springer-Verlag, New York, 1972.
					
					\bibitem{Hum} J. E. Humphreys, {\itshape Representations of semisimple Lie algebras in the BGG category O}, Graduate Studies in Mathematics 94, American Mathematical Society, Providence, RI, 2008.
					\bibitem{JB} J. Brundan, Tilting modules for Lie superalgebras, Comm. Algebra, 32 (2004), 2251–2268.
					\bibitem{Kac1} V. G. Kac, {\itshape Simple graded Lie algebras of finite growth}, Izv. Akad. Nauk SSSR, Ser. Mat. 32 (1968), 1323-1367.
					\bibitem{Kac77} V. G. Kac, {\itshape Lie superalgebras}, Adv. Math., 26 (1977), 8-96.
					\bibitem{Kac} V. G. Kac, {\itshape Representations of classical Lie superalgebras}, Lecture Notes in Math.
					676,  Springer-Verlag, Berlin,  1978, 597-626.
					
					
						\bibitem{Kac98} V. G. Kac, {\itshape
							Classification of infinite-dimensional simple linearly compact Lie superalgebras}, Adv. Math.   139 (1998), no. 1, 1-55.
					
					
					\bibitem{KoN} S. Kobayashi and T. Nagano, {\itshape On filtered Lie algebras and geometric structure, I $\slash$ II $\slash$ III $\slash$ IV $\slash$ V}, J. Math. Mech. 13(1964), 875-908. $\slash$ 14(1965), 513-522. $\slash$ 14(1965), 679-706. $\slash$ 15(1966), 163-175. $\slash$ 15(1966), 315-328.
					
					
					\bibitem{Mu12} I. M. Musson, {\em Lie superalgebras and enveloping algebras},
					Grad. Stud. Math., 131 American Mathematical Society, Providence, RI, 2012.
					
					\bibitem{Rud} A. Rudakov, {\itshape Irreducible representations of infinte-dimensional Lie algebras of Cartan type}. Math. USSR Izvestija 8(1974), 836-866.
					
					
					\bibitem{Ser05} V. Serganova, {\itshape On representations of Cartan type Lie superalgebras}, Amer.
					Math. Soc. Transl. 213 (2005), 223-239.
					
					\bibitem{Shap82}  A. V. Shapovalov, {\itshape Invariant differential operators and irreducible representations of finite dimensional Hamiltonian and Poisson Lie superalgebras}, Serdica 7 (1981), no. 4, 337-342.

\bibitem{Shc1} I. Shchepochkina, {\em New exceptional simple Lie superalgebras}, C.R. Bul. Sci. 36, No. 3 (1983), 313-314.
\bibitem{Shc2} I. Shchepochkina, {\em The five exceptional simple Lie superalgebras of vector fields}, preprint.


				
					\bibitem{Shen1} G. Shen, {\itshape Graded modules of graded Lie algebras of Cartan type(I)--mixed products of modules}. Scientia Sinica (Ser. A) 29(6), 570-581, 1986.
					\bibitem{Shen2} G. Shen, {\itshape Graded modules of graded Lie algebras of Cartan type. II. Positive and negative graded modules}, Sci. Sinica Ser. A29(1986), no.10, 1009-1019.
					
					\bibitem{Shen3} G. Shen, {\itshape Graded modules of graded Lie algebras of Cartan type(III)--irreducible modules}. Chin. Ann. Math. 9B(4), 404-417, 1988.
					
					\bibitem{Skr} S. Skryabin, {\itshape Independent systems of derivations and Lie algebra representations}. in ``Algebra and Analysis, Eds: Archipov / Parshin / Shafarvich." Walter de Gruyter\& Co, Berlin-New York, 115-150, 1994.
					
					\bibitem{Soe} W. Soergel, {\itshape Character formulas for tilting modules over Kac-Moody algebras}. Represent. Theor. 00, 432-448, 1998.
					
					
					\bibitem{SS65} I. M. Singer and S. Sternberg, The infinite groups of Lie and Cartan. I: The transitive groups, J. Analyse Math. 15 (1965), 1-114.
					
					
					
				\end{thebibliography}
			\end{document}